\documentclass[12pt]{article}
\usepackage{color}
\usepackage{amsmath,amsfonts,amssymb,amsthm}
\usepackage{rsfso}
\usepackage{comment}
\usepackage{enumerate}
\usepackage[dvips]{graphicx}
\usepackage{float}

\usepackage[numbers]{natbib}
\newcommand{\ds}{\displaystyle }

\newcommand{\vc}[1]{{\boldsymbol #1}}

\newcommand{\sr}[1]{{\mathcal #1}}
\newcommand{\dd}[1]{\mathbb{#1}}

\newcommand{\ol}{\overline}
\newcommand{\ul}{\underline}

\newcommand{\eq}[1]{(\ref{eq:#1})}
\newcommand{\mylem}[1]{Lemma~\ref{mylem:#1}}
\newcommand{\mycor}[1]{Corollary~\ref{mycor:#1}}
\newcommand{\mythr}[1]{Theorem~\ref{mythr:#1}}

\newcommand{\mypro}[1]{Proposition~\ref{mypro:#1}}

\newcommand{\myrem}[1]{Remark~\ref{myrem:#1}}
\newcommand{\myfig}[1]{Figure~\ref{myfig:#1}}
\newcommand{\myapp}[1]{Appendix~\ref{myapp:#1}}
\newcommand{\sectn}[1]{Section~\ref{mysec:#1}}

\newcommand{\mylemt}[1]{\ref{mylem:#1}}

\newcommand{\sect}[1]{\ref{mysec:#1}}

\newcommand{\pend}{\hfill \thicklines \framebox(6.6,6.6)[l]{}}

\newenvironment{proof*}[1]{\noindent {\sc  #1} \rm}{\pend}

\newtheorem{theorem}{Theorem}[section]
\newtheorem{lemma}{Lemma}[section]

\newtheorem{proposition}{Proposition}[section]

\newtheorem{remark}{Remark}[section]

\newtheorem{corollary}{Corollary}[section]

\newenvironment{mylist}[1]{\begin{list}{}
{\setlength{\itemindent}{#1mm}}
{\setlength{\itemsep}{0ex plus 0.2ex}}
{\setlength{\parsep}{0.5ex plus 0.2ex}}
{\setlength{\labelwidth}{10mm}}
}{\end{list}}

\newcommand{\setnewcounter} {
\setcounter{subsection}{0}
\setcounter{equation}{0}
\setcounter{figure}{0}
\setcounter{conjecture}{0}
\setcounter{assumption}{0}
\setcounter{question}{0}
\setcounter{definition}{0}
\setcounter{theorem}{0}
\setcounter{corollary}{0}
\setcounter{lemma}{0}
\setcounter{proposition}{0}
\setcounter{remark}{0}
}

\begin{document}

\title{\Large \bf Diffusion approximation of the stationary distribution of a two-level single server queue}


\author{Masakiyo Miyazawa\footnotemark}
\date{\today}

\maketitle

\begin{abstract}
We consider a single server queue which has a threshold to change its arrival process and service speed by its queue length, which is referred to as a two-level $GI/G/1$ queue. This model is motivated by an energy saving problem for a single server queue whose arrival process and service speed are controlled. To get its performance in tractable form, we study the limit of the stationary distribution of the queue length in this two-level queue under scaling in heavy traffic. Except for a special case, this limit corresponds to its diffusion approximation. It is shown that this limiting distribution is truncated exponential (or uniform if the drift is null) below the threshold level and exponential above it under suitably chosen system parameters and generally distributed inter-arrival times and workloads brought by customers. This result is proved under a mild limitation on arrival parameters using the so called BAR approach studied in \cite{BravDaiMiya2017,BravDaiMiya2024,Miya2017,Miya2024}. We also intuitively discuss about a diffusion process corresponding to the limit of the stationary distribution under scaling.
\end{abstract}



\section{Introduction}
\label{mysec:introduction}
\setnewcounter

We consider a single server queue which has a threshold to change its arrival process and service speed by its queue length. Namely, customers arrive subject to a renewal process as long as its queue length (including a customer being served) is less than threshold $\ell_{1}$, while it is changed to another renewal process otherwise. The service speed is changed similarly but in a slightly different way. That is, the service speed is constant when the queue length is not greater than the threshold, while it is another constant otherwise. We assume that customers are served in the FCFS manner, and the amounts of the work which are brought by them  are independent and identically distributed and independent of the arrival process. We refer to this model as a 2-level $GI/G/1$ queue.

This 2-level $GI/G/1$ queue is motivated by an energy saving problem to control the arrival process and service speed. Here, the system is an internet server, and its energy consumption depends on service speed. Then, the problem is to find an optimal threshold to minimize the total cost of the operating energy of the server and customer's waiting.

So far, we are interested in the performance of this 2-level $GI/G/1$ queue over long time, for example the mean queue length. However, it is hard to derive such a characteristic in an analytically tractable form in general, while, for the special case that all the distributions are exponential, its analytic expression is available, but the variability of the arrival times and service amounts are not well captured in this exponential case. Because of these facts, an approximation by a reflected diffusion is helpful, where a reflected diffusion means a one-dimensional diffusion reflected at the origin, in which the term ``diffusion'' is used for a one-dimensional continuous Markov process on the real line which has a finite and positive second moment at each time. 

So we study the diffusion approximation through the limit of the stationary distributions of the scaled queue length of this 2-level $GI/G/1$ queue in heavy traffic, where the scaling factor and threshold level are chosen so that the limit is well obtained as a probability distribution. It is expected that this limiting distribution corresponds to the stationary distribution of a reflected diffusion whose drift and variance are discontinuously changed at the level corresponding to the threshold (see \sectn{process-limit}). The stationary distribution of this reflected diffusion is recently studied in \cite{Miya2024b,Miya2024d}. Thus, we will indirectly consider a diffusion approximation for the 2-level $GI/G/1$ queue.

There is another motivation of this study. In general, it is hard to theoretically study diffusion approximation for state-dependent queues. For example, there are some theoretical studies about them, but they assume that the arrival and service processes are of exponential types, and their rates are smoothly changed (e.g., see \cite{MandPats1998,Yama1995}), while the arrival process and service speed discontinuously change in the present model. Thus, theoretical study for its diffusion approximation is challenging.

As is well known, the scaling limit of the stationary distribution is exponential for the standard $GI/G/1$ queue in heavy traffic (e.g., see \cite{King1961a}). Furthermore, \cite{Miya2024} shows that the limiting distribution is truncated exponential (or uniform for the null drift case) for the $GI/G/1$ queue with a finite waiting room. Both distributions agree with the stationary distributions of the corresponding reflected diffusions (see \cite{Harr2013}). Thus, one may conjecture that the limiting distribution for the 2-level $GI/G/1$ queue is truncated exponential or uniform below the threshold level and exponential above it. This conjecture can be directly verified if all the distributions are exponential (see \mypro{2LQ-exponential}), but this is exceptional. Under the generally distributed assumptions, we verify it, employing the BAR approach studied in \cite{BravDaiMiya2017,BravDaiMiya2024,Miya2017,Miya2024}. This verification, \mythr{2LQ-main}, requires a mild limitation on the parameters of the arrival process (see (\sect{sequence}.f) in \sectn{assumption}). Its proof needs some knowledge about Palm distributions and a truncation technique in computations, but does not require any deeper theory. However, the proof is complicated, so we split it into several components so that they are orthogonal as long as possible. This is the reason why we spend two sections for the proof.

In this paper, we only study the limit of the scaled stationary distribution, but, as our results are related to the diffusion approximation,  it is interesting to identify a reflected diffusion as a process limit for the 2-level $GI/G/1$ queue. This is a hard question because this diffusion discontinuously changes its drift and variance at the fixed state. We intuitively discuss what can be expected about this process limit in \sectn{process-limit}.

This paper is made up by six sections. We first introduce a general framework for a Markov process describing the 2-level $GI/G/1$ queue in \sectn{framework}, then provide notations and assumptions for a sequence of the 2-level $GI/G/1$ queues in heavy traffic in \sectn{assumption}. A main result, \mythr{2LQ-main}, and the exponentially distributed case, \mypro{2LQ-exponential}, are presented in \sectn{main}. In \sectn{2LQ-BAR}, we introduce the so called BAR approach, and apply it to the 2-level $GI/G/1$ queue. Using these results, \mythr{2LQ-main} is proved in \sectn{proof-main}. Finally, in \sectn{concluding}, we intuitively discuss two topics. In \sectn{process-limit}, we consider a process-limit for the 2-level $GI/G/1$ queue in heavy traffic. In \sectn{intuitive}, we intuitively verify the conjecture in \myrem{2LQ-main}.

\section{Framework for the 2-level $GI/G/1$ queue}
\label{mysec:framework}
\setnewcounter

We take the stochastic framework introduced in \cite{Miya2024} to study the 2-level $GI/G/1$ queue. Let $L(t)$, $R_{e}(t)$ and $R_{s}(t)$ be the queue length, the remaining arrive time and the remaining workload of a customer being served, respectively, at time $t$. Note that the remaining service time cannot be used for $R_{s}(t)$ here because the service speed may change in future. Define
\begin{align*}
  X(\cdot) = \{(L(t),R_{e}(t),R_{s}(t)); t \ge 0\},
\end{align*}
Then, it is easy to see that $X(\cdot)$ is a Markov process with state space $\dd{Z}_{+} \times \dd{R}_{+}^{2}$, and the 2-level $GI/G/1$ queue is described by $X(\cdot)$, where $\dd{Z}_{+}$ is the set of nonnegative integers and $\dd{R}_{+}=[0,\infty)$. It can be assumed that $X(t)$ is right-continuous and has left-limits at any point of time. Let $N_{e}(\cdot) \equiv \{N_{e}(t); t \ge 0\}$ and $N_{s}(\cdot) \equiv \{N_{s}(t); t \ge 0\}$ be the counting processes generated by the vanishing instants of $R_{e}(t)$ and $R_{s}(t)$, respectively. Note that those counting processes can be viewed as point processes which are integer-valued random measures on $(\dd{R}_{+},\sr{B}(\dd{R}_{+}))$ whose atoms are the counting instants of $N_{e}(\cdot)$ and $N_{s}(\cdot)$, where $\sr{B}(\dd{R}_{+})$ is the Borel field on $\dd{R}_{+}$. We denote those point processes by $N_{e}$ and $N_{s}$ respectively.

According to \cite{Miya2024}, we take stochastic basis $(\Omega,\sr{F},\dd{F},\dd{P})$ with time-shift operator semi-group $\Theta_{\bullet} \equiv \{\Theta_{t}; t \ge 0\}$ to jointly consider the Markov process $X(\cdot)$ and the counting processes $N_{e}(\cdot)$ and $N_{s}(\cdot)$, where $\dd{F} \equiv \{\sr{F}_{t}; t \ge 0\}$ is a filtration generated by $X(\cdot)$, that is, $\sr{F}_{t} = \sigma(\{X(u); 0 \le u \le t\})$, and $\Theta_{\bullet}$ is a operator semi-group (namely, $\Theta_{s+t} = \Theta_{s} \circ \Theta_{t}$ for $s,t \ge 0$) on $\Omega$ satisfying the following two conditions.
\begin{align*}
 \mbox{(\sect{framework}.a) Measurability:} \quad & A \in \sr{F}_{t} \; \mbox{ implies } \; \Theta_{s}^{-1} A \in \sr{F}_{s+t}, \qquad s,t \ge 0,\\
 \mbox{(\sect{framework}.b) Consistency:} \hspace{4ex} & X(t) \circ \Theta_{s}(\omega) = X(s+t)(\omega), \qquad s,t \ge 0, \omega \in \Omega,\\
 & N(t) \circ \Theta_{s}(\omega) = N(s+t)(\omega) - N(s)(\omega), \quad s,t \ge 0, \omega \in \Omega,\\
 & \mbox{for $N(\cdot) = N_{e}(\cdot)$ and $N(\cdot) = N_{s}(\cdot)$},
\end{align*}
where $W \circ \Theta_{s}(\omega) = W (\Theta_{s}(\omega))$ for random variable $W$, and $\Theta_{s}^{-1} A = \{\omega \in \Omega; \Theta_{s}(\omega) \in A\}$.

\section{A sequence of the 2-level $GI/G/1$ queues}
\label{mysec:sequence}
\setnewcounter

We consider the scaling limit of the stationary distribution of the 2-level $GI/G/1$ queue in heavy traffic. To this end, we introduce a sequence of this queueing models indexed by $r \in (0,1]$. Since this sequence can be countable, one may take $r=1/\sqrt{n}$ for positive integer $n$ so that the scaling is identical with the standard diffusion approximation, but we keep $r$ just for convenience. Let $a \wedge b = \min(a,b)$ and $a \vee b = \max(a,b)$ for $a,b \in \dd{R}$, which will be used throughout the paper.

\subsection{Assumptions}
\label{mysec:assumption}

For the $r$-th model, let $\ell^{(r)}_{1} \ge 1$ be the threshold level, and let $L^{(r)}(t)$ be the number of customers in the system at time $t$. We refer to this number as the queue length. Let $t^{(r)}_{e,n}$ be the $n$-th arrival time of a customer, and let $T^{(r)}_{s}(n)$ be the workload of the $n$-th arriving customer. We assume that $T_{e}^{(r)}(n) \equiv t^{(r)}_{e,n} - t^{(r)}_{e,n-1}$ is independent of $\{t^{(r)}_{e,n'}; 0 \le n' \le n-1\}$ and everything else for all $n \ge 1$, where $t^{(r)}_{e,0} = 0$. Thus, the inter-arrival times are independent, but their distributions depend on the queue length. Namely, $T_{e}^{(r)}(n)$ has the same distribution as $T^{(r)}_{e,1}$ if $L(t^{(r)}_{e,n-1}-) < \ell^{(r)}_{1}$ while its distribution agrees with that of $T_{e,2}^{(r)}$ if $L(t^{(r)}_{e,n-1}-) \ge \ell^{(r)}_{1}$, where $T^{(r)}_{e,1}$ and $T^{(r)}_{e,2}$ are independent random variables which may have different distributions. As for service, let $T^{(r)}_{s}(n)$ be the workload of the $n$-th arriving customers. We assume that $\{T^{(r)}_{s}(n); n \ge 1\}$ is a sequence of $i.i.d.$ random variables. We denote a random variable subject to the common distribution of $T^{(r)}_{s}(n)$ by $T^{(r)}_{s}$. Let $c^{(r)}_{1}$ and $c^{(r)}_{2}$ be the service speeds when the queue length is not greater than $\ell^{(r)}_{1}$ and otherwise, respectively. Namely, as long as $R^{(r)}_{s}(t)$ is differentiable at time $t$,
\begin{align*}
  \frac {d}{dt} R^{(r)}_{s}(t) = -c^{(r)}_{1} 1(1 \le L^{(r)}(t) \le \ell^{(r)}_{1}) -c^{(r)}_{2} 1(L^{(r)}(t) > \ell^{(r)}_{1}),  
\end{align*}
where $1(\cdot)$ is the indicator function of the statement ``$\cdot$''. These dynamics is depicted in the sample path of $L^{(r)}(t)$ in \myfig{2LQ-1}. 
\begin{figure}[h] 
   \centering
   \includegraphics[height=4.90cm]{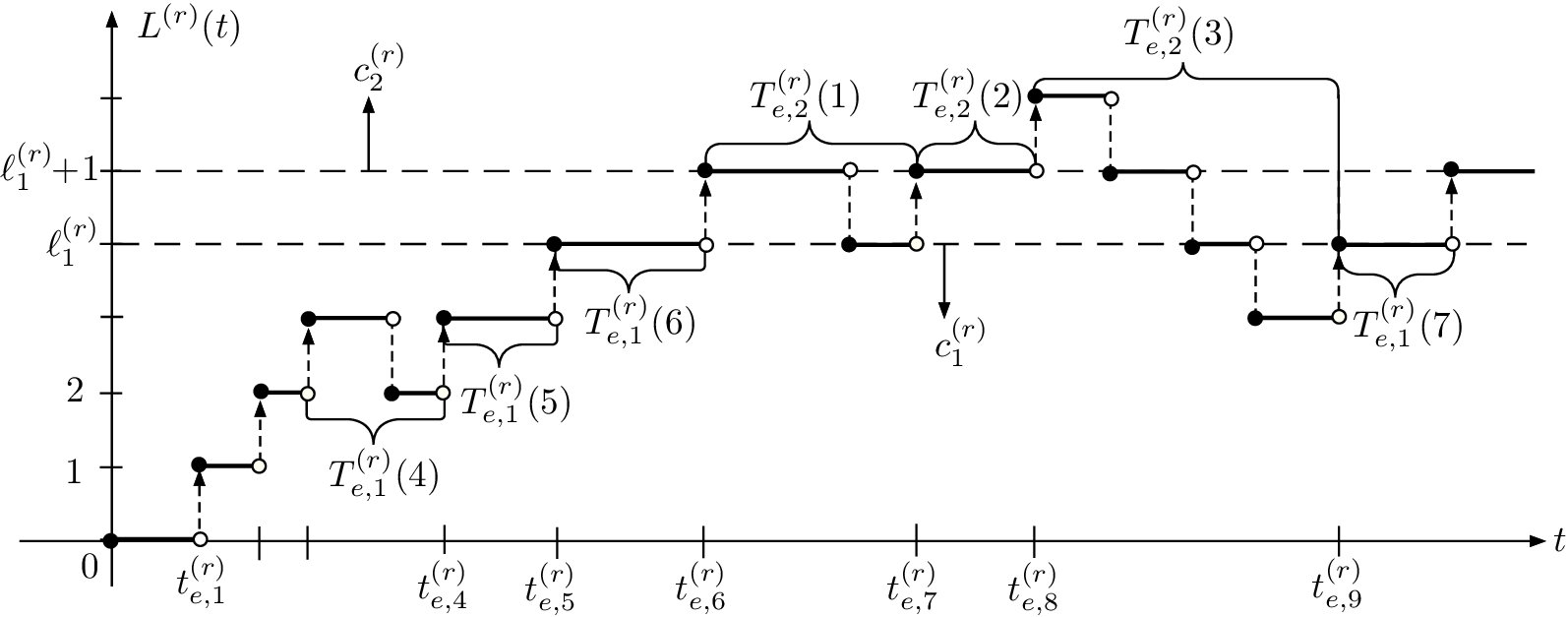} 
   \caption{Sample path of the queue length process of the $r$-th model}
   \label{myfig:2LQ-1}
\end{figure}

We will use the following notations.
\begin{align*}
 & m^{(r)}_{e,i} = \dd{E}(T^{(r)}_{e,i}), \qquad \lambda^{(r)}_{i} = 1/m^{(r)}_{e,i}, \qquad \sigma^{(r)}_{e,i} = (\dd{E}[(T^{(r)}_{e,i} - m^{(r)}_{e,i})^{2}]^{\frac 12}, \qquad i=1,2,\\
 & m^{(r)}_{s} = \dd{E}(T^{(r)}_{s}), \qquad \mu^{(r)} = 1/m^{(r)}_{s}, \qquad \sigma^{(r)}_{s} = (\dd{E}[(T^{(r)}_{s} - m^{(r)}_{s})^{2}])^{\frac 12},\\
 & \rho^{(r)}_{i} = \lambda^{(r)}_{i}/(c^{(r)}_{i}\mu^{(r)}), \qquad i = 1,2.
\end{align*}
For the $r$-th model, let $R^{(r)}_{e}(t)$ and $R^{(r)}_{s}(t)$ be the remaining arrival time and the remaining workload of a customer being served, respectively, at time $t$. Define $X^{(r)}(\cdot) \equiv \{X^{(r)}(t); t \ge 0\}$ by
\begin{align*}
  X^{(r)}(t) = (L^{(r)}(t), R^{(r)}_{e}(t), R^{(r)}_{s}(t)),
\end{align*}
and denote the counting processes of the arrivals and service completions by $N^{(r)}_{e}(\cdot)$ and $N^{(r)}_{s}(\cdot)$, respectively. Note that $X^{(r)}(\cdot)$ is a piecewise deterministic Markov process, so it is strong Markov (e.g., see \cite{Davi1993}). Obviously, $X^{(r)}(\cdot)$, $N^{(r)}_{e}(\cdot)$ and $N^{(r)}_{s}(\cdot)$ are consistent with the time-shift operator $\Theta_{\bullet}$. We assume the following assumptions.
\begin{mylist}{3}
\item [(\sect{sequence}.a)] $\{\dd{E}[(T^{(r)}_{e,i})^{3}]; r \in (0,1]\}$ for $i=1,2$ and $\{\dd{E}[(T^{(r)}_{s})^{3}]; r \in (0,1]\}$ are bounded.
\item [(\sect{sequence}.b)] For $i=1,2$, $m^{(r)}_{e,i}, \lambda^{(r)}_{i}, \sigma^{(r)}_{e,i}, c^{(r)}_{i}$ and $m^{(r)}_{s}, \mu^{(r)}, \sigma^{(r)}_{s}$ are finite and positive, and converge to $m_{e,i}, \lambda_{i}, \sigma_{e,i}, c_{i} > 0$ and $m_{s}, \mu, \sigma_{s} > 0$, respectively, as $r \downarrow 0$.
\item [(\sect{sequence}.c)] There are constants $b_{1} \in \dd{R}$ and $b_{2} > 0$ such that $c^{(r)}_{i} \mu^{(r)} - \lambda^{(r)}_{i} = r \mu^{(r)} b_{i} + o(r)$ as $r \downarrow 0$.

\item [(\sect{sequence}.d)] There is a constant $\ell_{1}$ such that $r \ell^{(r)}_{1} = \ell_{1}+ o(r)$ as $r \downarrow 0$.
\item [(\sect{sequence}.e)] $\rho^{(r)}_{2} < 1$.
\end{mylist}
Here, for functions $f, g: (0,1] \to (0,\infty)$, we write $f(r) = o(g(r))$ as $r \downarrow 0$ if \linebreak $\lim_{r \downarrow 0} f(r)/g(r) = 0$. Similarly, $f(r) = O(g(r))$ if $\limsup_{r \downarrow 0} f(r)/g(r) < \infty$. Note that (\sect{sequence}.a) implies that
\begin{mylist}{3}
\item [(\sect{sequence}.a')] $\{(T^{(r)}_{e,i})^{2}; r \in (0,1]\}$ for $i=1,2$ and $\{(T^{(r)}_{s})^{2}; r \in (0,1]\}$ are uniformly integrable.
\end{mylist}
Furthermore, define $\rho_{i} = \lambda_{i}/(c_{i}\mu)$, then (\sect{sequence}.b) and (\sect{sequence}.c) imply that $1 - \rho^{(r)}_{i} = r b_{i}/c_{i} + o(r)$, $\rho_{i} = 1$ and $\lambda_{i} = c_{i} \mu$. Hence, $\lambda_{1} = \lambda_{2}$ if and only if $c_{1} = c_{2}$.

Most of our arguments are valid under the assumptions (\sect{sequence}.a)--(\sect{sequence}.e), but we require the following condition for their final step to determine the limiting distribution of the scaled queue length in heavy traffic.
\begin{mylist}{3}
\item [(\sect{sequence}.f)] $\sigma_{e,1}^{2} = \sigma_{e,2}^{2}$ if $\lambda_{1} = \lambda_{2}$.
\end{mylist}

\begin{remark}\rm
\label{myrem:3f-2}
This is a mild limitation because no condition is needed for $\lambda_{1} \not= \lambda_{2}$ and it is automatically satisfied when $T_{e,1}$ and $T_{e,2}$ are identically distributed. This condition is likely not needed but hard to remove. See \myrem{2LQ-3f} for its details.
\pend
\end{remark}

We take countably many $r$'s for a limiting operation. For such a sequence, there exists a stochastic basis $(\Omega,\sr{F},\dd{F},\dd{P})$ with time-shift operator $\Theta_{\bullet}$ on which $X^{(r)}(\cdot)$ is a stationary process. In what follows, $\dd{P}$ is always this probability measure. By the condition (\sect{sequence}.e), the Markov process $X^{(r)}(\cdot)$ has the stationary distribution. Let $X^{(r)} \equiv (L^{(r)},R^{(r)}_{e},R^{(r)}_{s})$ be a random vector subject to the stationary distribution of $X^{(r)}(t)$. For convenience, let $R^{(r)} = (R^{(r)}_{e}, R^{(r)}_{s})$. For $i=1,2$, let $N^{(r)}_{e,i}$ be the stationary point process obtained from the renewal process whose inter-counting times have the same distribution as that of $T^{(r)}_{e,i}$. Similarly, let $N^{(r)}_{s}$ be the stationary point process for the service completion instants.

Define $\alpha^{(r)}_{e} = \dd{E}[N^{(r)}_{e}(1)]$ and $\alpha^{(r)}_{s} = \dd{E}[N^{(r)}_{s}(1)]$. At present, we cannot compute $\alpha^{(r)}_{e}$ and $\alpha^{(r)}_{s}$, but the following lemma is intuitively clear because $X^{(r)}(\cdot)$ is a stationary process by (\sect{sequence}.e). Since this lemma is a key for our analysis, we formally prove it in \myapp{alpha 1}.
\begin{lemma}
\label{mylem:alpha 1}
Assume the assumptions (\sect{sequence}.b) and (\sect{sequence}.e), then $\alpha^{(r)}_{e}$ is finite and positive, and $\alpha^{(r)}_{s} = \alpha^{(r)}_{e}$.
\end{lemma}

\subsection{Main results}
\label{mysec:main}

We now present a main result, which will be proved in \sectn{proof-main}. Recall that $L^{(r)}$ is a random variable subject to the stationary distribution of the $r$-th 2-level $GI/G/1$ queue. Define probability distribution $\nu^{(r)}$ under the assumptions (\sect{sequence}.a)--(\sect{sequence}.e) as
\begin{align*}
  \nu^{(r)}(B) = \dd{P}(rL^{(r)} \in B), \qquad B \in \sr{B}(\dd{R}_{+}).
\end{align*}
We refer to this $\nu^{(r)}$ as a scaled stationary distribution. Define finite measures $\nu^{(r)}_{i}$ and conditional probability distributions $\nu^{(r)}_{i|cd}$ on $(\dd{R}_{+}, \sr{B}(\dd{R}_{+}))$ for $i=1,2$ as
\begin{align*}
 & \nu^{(r)}_{1}(B) = \dd{P}(rL^{(r)} \in B; L^{(r)} \le \ell^{(r)}_{1}), \qquad \nu^{(r)}_{2}(B) = \dd{P}(rL^{(r)} \in B; \ell^{(r)}_{1} < L^{(r)}),\\
 & \nu^{(r)}_{1|cd}(B) = \dd{P}(rL^{(r)} \in B| L^{(r)} \le \ell^{(r)}_{1}), \qquad \nu^{(r)}_{2|cd}(B) = \dd{P}(rL^{(r)} \in B| \ell^{(r)}_{1} < L^{(r)}),
 \end{align*}
for $B \in \sr{B}(\dd{R}_{+})$. Obviously, $\nu^{(r)}(B) = \nu^{(r)}_{1}(B) + \nu^{(r)}_{2}(B)$.

\begin{theorem}
\label{mythr:2LQ-main}
Assume the assumptions (\sect{sequence}.a)--(\sect{sequence}.f). Let
\begin{align}
\label{eq:2LQ-beta 1}
  \sigma_{i}^{2} = \lambda_{i}^{2} \sigma_{e,i}^{2} + \mu^{2} \sigma_{s}^{2}, \qquad \beta_{i} = \frac {2 b_{i}} {c_{i} \sigma_{i}^{2}}, \qquad i =1, 2.
\end{align}
For each $r \in (0,1]$, define $\widetilde{\nu}_{i|cd}$ as the probability measures on $\sr{B}(\dd{R}_{+})$ which have the density functions $h_{i|cd}$ for $i=1,2$ such that
\begin{align}
\label{eq:2LQ-h 1}
  & h_{1|cd}(x) = 1(0 \le x \le \ell_{1}) \times \left\{
\begin{array}{ll}
  \ds 1/\ell_{1}, & b_{1} = 0, \\
 \ds \frac {e^{\beta_{1} \ell_{1}} \beta_{1}}{e^{\beta_{1} \ell_{1}} -1}  e^{-\beta_{1} x} , \quad & b_{1} \not= 0,
\end{array}
\right.\\
\label{eq:2LQ-h 2}
  & h_{2|cd}(x) = 1 (x > \ell_{1}) \beta_{2} e^{-\beta_{2} (x-\ell_{1})}.
\end{align}
\begin{itemize}
\item [(i)] For each $i=1,2$, if $\nu^{(r)}_{i}(\dd{R}_{+})$ is bounded away from $0$ as $r \downarrow 0$, that is,
\begin{align}
\label{eq:2LQ-phi-bound 1}
  \liminf_{r \downarrow 0} \varphi^{(r)}_{i}(0) > 0, \qquad i=1,2,
\end{align}
then $\nu^{(r)}_{i|cd}$ weakly converges to $\widetilde{\nu}_{i|cd}$, namely,
\begin{align}
\label{eq:2LQ-cd-limit 1}
  \lim_{r \downarrow 0} \nu_{i|cd}^{(r)}(B) = \widetilde{\nu}_{i|cd}(B), \qquad B \in \sr{B}(\dd{R}_{+}).
\end{align}

\item [(ii)] Define probability measure $\widetilde{\nu}$ on $(\dd{R}_{+},\sr{B}(\dd{R}_{+}))$ and constant $\widetilde{\alpha}_{e}$ as
\begin{align}
\label{eq:2LQ-nu-tilde}
 & \widetilde{\nu}(B) = \widetilde{\nu}_{1|cd}(B) \widetilde{A}_{1} + \widetilde{\nu}_{2|cd}(B) \widetilde{A}_{2}, \qquad B \in \sr{B}(\dd{R}_{+}),\\
\label{eq:2LQ-alpha-t-1}
 & \widetilde{\alpha}_{e} = \lambda_{1} \widetilde{A}_{1} + \lambda_{2} \widetilde{A}_{2},
\end{align}
where $\widetilde{A}_{i}$ for $i=1,2$ are given by
\begin{align}
\label{eq:2LQ-A 1}
  & \widetilde{A}_{1} = \begin{cases}
\ds \frac {2b_{2} \ell_{1}} {c_{2} \sigma_{1}^{2} + 2b_{2} \ell_{1}}, & b_{1}=0,\\
\ds \frac {c_{1}b_{2} (e^{\beta_{1} \ell_{1}} - 1)} {c_{2}b_{1} + c_{1}b_{2} (e^{\beta_{1} \ell_{1}} - 1)}, \; & b_{1} \not=0,
\end{cases}
 \qquad \widetilde{A}_{2} = 1 - \widetilde{A}_{1}.
\end{align}
If either $\lambda^{(r)}_{1} - \lambda^{(r)}_{2} = O(r)$, equivalently, $c^{(r)}_{1} - c^{(r)}_{2} = O(r)$, or $T^{(r)}_{s}$ is exponentially distributed for all $r \in (0,1]$, then
\begin{align}
\label{eq:2LQ-nu-limit 1}
 & \lim_{r \downarrow 0} \nu^{(r)}(B) = \widetilde{\nu}(B), \qquad B \in \sr{B}(\dd{R}_{+}),\\
\label{eq:2LQ-alpha-limit 1}
 & \lim_{r \downarrow 0} \alpha^{(r)}_{e} = \widetilde{\alpha}_{e}.
\end{align}

\item [(iii)] If $\lambda_{1} > \lambda_{2}$ ($\lambda_{1} < \lambda_{2}$), then $\nu^{(r)}_{1|cd}$ ($\nu^{(r)}_{2|cd}$) weakly converges to $\widetilde{\nu}_{1|cd}$ ($\widetilde{\nu}_{2|cd}$, respectively), as $r \downarrow 0$.

\end{itemize}
\end{theorem}

\begin{remark}
\label{myrem:2LQ-main}
Note that, under one of the conditions of (ii) of this theorem, we have
\begin{align}
\label{eq:2LQ-Ai}
   \lim_{r \downarrow 0} \nu^{(r)}_{i}(\dd{R}_{+}) = \widetilde{A}_{i}, \qquad i=1,2.
\end{align}
Conversely, if \eq{2LQ-Ai} holds, then we have the weak convergence \eq{2LQ-nu-limit 1} from (i) because $\widetilde{A}_{i} > 0$. Thus, \eq{2LQ-Ai} is equivalent to \eq{2LQ-nu-limit 1}. We conjecture that \eq{2LQ-nu-limit 1} holds without any condition in (ii). This is supported by the intuitive verification of \eq{2LQ-Ai}, which is detailed in \sectn{intuitive}.
\pend
\end{remark}

Let $S^{(r)}(n)$ be the sojourn time of the $n$-th arriving customer in service, then the traffic intensity $\rho^{(r)}$ can be defined under the stationary setting as
\begin{align}
\label{eq:2LQ-traffic 1}
  \rho^{(r)} = \alpha^{(r)}_{e} \lim_{n \to \infty} \frac 1n \sum_{k=1}^{n} S^{(r)}(k).
\end{align}
Hence, by Little's law, we have $\rho^{(r)} = \dd{P}(L^{(r)} > 0) = 1 - \dd{P}(L^{(r)}=0)$.
In diffusion applications of a single server queue, the mean queue length is estimated through $(1 - \rho^{(r)}) \dd{E}[L^{(r)}]$ as $r \downarrow 0$. This limit is computed in the following corollary, which is proved in \myapp{2LQ-mean-Q}.

\begin{corollary}
\label{mycor:2LQ-mean-Q}
Under the assumptions (\sect{sequence}.a)--(\sect{sequence}.f) and either one of the assumptions in (ii) of \mythr{2LQ-main}, 
\begin{align}
\label{eq:2LQ-mean-L 1}
 & \lim_{r \downarrow 0} (1 - \rho^{(r)}) \dd{E}[L^{(r)}] = \lim_{r \downarrow 0} \dd{P}(L^{(r)}=0) \dd{E}[L^{(r)}] \nonumber\\
  & \quad = \begin{cases}
 \ds \frac {\sigma_{1}^{2} (2b_{2} \ell_{1}(b_{2} \ell_{1} + c_{2} \sigma_{1}^{2}) + c_{2}^{2} \sigma_{1}^{2} \sigma_{2}^{2})} {2(2b_{2} \ell_{1} + c_{2} \sigma_{1}^{2})^{2}}, & b_{1}=0, \vspace{1ex}\\
 \ds \frac {c_{1}^{2} b_{2}^{2} \sigma_{1}^{2} (e^{\beta_{1} \ell_{1}} - 1) + 2 b_{1} b_{2} \ell_{1} (c_{2} b_{1} - c_{1} b_{2}) + c_{2}^{2} b_{1}^{2} \sigma_{2}^{2}} {2 b_{1} b_{2} (c_{2}b_{1} + c_{1}b_{2} (e^{\beta_{1} \ell_{1}} - 1))}, \; & b_{1} \not= 0.
\end{cases}
\end{align}
\end{corollary}

Let us consider \mythr{2LQ-main} when $T_{e,1}^{(r)}, T_{e,2}^{(r)}$ and $T_{s}^{(r)}$ are all exponentially distributed. We call this model the $r$-th 2-level $M/M/1$ queue. Note that (\sect{sequence}.f) is automatically satisfied in this exponential model.

\begin{proposition}
\label{mypro:2LQ-exponential}
Under the assumptions (\sect{sequence}.a)--(\sect{sequence}.e), the scaled stationary distribution $\nu^{(r)}_{\exp}$ of the $r$-th 2-level $M/M/1$ queue weakly converges to $\widetilde{\nu}_{\exp}$ as $r \downarrow 0$, where $\widetilde{\nu}_{\exp}$ is defined as $\widetilde{\nu}_{\exp}(B) = \widetilde{\nu}_{1|cd}(B) \widetilde{A}_{\exp,1} + \widetilde{\nu}_{2|cd}(B) \widetilde{A}_{\exp,2}$ for $B \in \sr{B}(\dd{R}_{+})$ in which
\begin{align}
\label{eq:2LQ-A-MM 1}
  & \widetilde{A}_{\exp,1} = \begin{cases}
\ds \frac {\beta_{2} \ell_{1}} {1 + \beta_{2} \ell_{1}}, & b_{1} =0,\smallskip\\
\ds \frac {\beta_{2} (e^{\beta_{1} \ell_{1}} - 1)} {\beta_{1} + \beta_{2} (e^{\beta_{1} \ell_{1}} - 1)}, \quad & b_{1} \not=0,   
\end{cases}
\end{align}
and $\widetilde{A}_{\exp,2} = 1 - \widetilde{A}_{\exp,1}$.
\end{proposition}
\begin{remark}
\label{myrem:2LQ-expenential}
Since $\sigma_{i}^{2} = 2$ and $\beta_{i} = b_{i}/c_{i}$ for $i=1,2$ in the exponential case, $\widetilde{\nu}_{\exp}$ and $\widetilde{A}_{\exp,i}$ agree with $\widetilde{\nu}$ and $\widetilde{A}_{\exp,i}$ in \mythr{2LQ-main}, respectively, for $i=1,2$.
\pend
\end{remark}

Of course, \mypro{2LQ-exponential} is the special case of (ii) of \mythr{2LQ-main}, so no proof is needed. However, for a double check, we give its direct proof in \myapp{2LQ-exponential-proof}, which is helpful to convince (ii) of \mythr{2LQ-main}. This proof is not obvious because the exponential model requires a Markov chain with two-dimensional state space to correctly describe the case that an arrival find the queue length less than $\ell_{1}$ when the previous arrival found the queue length greater than $\ell_{1}$ but multiple departures occurred (see the arrival at time $t^{(r)}_{e}(8)$ in \myfig{2LQ-1}). This Markov chain will be detailed in \myapp{2LQ-exponential-proof}.

\section{BAR's for the 2-level $GI/G/1$ queue}
\label{mysec:2LQ-BAR}
\setnewcounter

In this section, we prepare auxiliary notations and results for proving \mythr{2LQ-main}. A major difficulty for this proof comes from the fact that $L^{(r)}(\cdot)$ is not a Markov chain, and we have to work with $X^{(r)}(\cdot)$. Instead of directly computing the stationary distribution, we study it through a stationary equation, called a basic adjoint relationship, BAR in short, using Palm distributions, which will be defined below. This method is called a BAR approach, and has been used in \cite{BravDaiMiya2017,BravDaiMiya2024,Miya2017}. Recently, this approach is further studied in \cite{Miya2024}.

Throughout this section, we assume that (\sect{sequence}.a)--(\sect{sequence}.e) hold unless stated otherwise, but (\sect{sequence}.f) is not assumed. The proof of \mythr{2LQ-main} will be done in \sectn{proof-main}.

\subsection{Palm distribution and BAR}
\label{mysec:Palm-BAR}

Recall that we are working on stochastic basis $(\Omega,\sr{F},\dd{F},\dd{P})$ with time-shift operator semi-group $\Theta_{\bullet}$. Furthermore, by the stability assumption (\sect{sequence}.e), we can assume that this framework is stationary, that is,
\begin{align*}
  \dd{P}(\Theta_{t}^{-1} A) = \dd{P}(A), \qquad A \in \sr{F}.
\end{align*}
Then, for each $r \in (0,1]$, $X^{(r)}(\cdot)$ and the increments of $N^{(r)}_{e}$ and $N^{(r)}_{s}$ are jointly stationary. Since $\alpha^{(r)}_{e}$ and $\alpha^{(r)}_{s}$ are finite and positive by \mylem{alpha 1}, we can define probability distributions $\dd{P}^{(r)}_{e}$ and $\dd{P}^{(r)}_{s}$ on $(\Omega,\sr{F})$ by
\begin{align}
\label{eq:Palm 1}
  \dd{P}^{(r)}_{v}(A) = (\alpha^{(r)}_{v})^{-1} \dd{E}\left[ \int_{(0,1]} 1_{\Theta_{u}^{-1} A} N^{(r)}_{v}(du) \right], \qquad A \in \sr{F}, v=e,s,
\end{align}
where $1_{A}$ is the indicator function of the event $A \in \sr{F}$. For $v=e,s$, $\dd{P}^{(r)}_{v}$ is called a Palm distribution concerning the counting process $N^{(r)}_{v}$.

The Palm distributions play a major role for the derivation of the stationary equations in our analysis (e.g., see \eq{BAR-1}). Intuitively, $\dd{P}^{(r)}_{v}$ is considered as the conditional probability distribution given the event that $N^{(r)}_{v}$ counts arrivals for $v=e$ (service completion for $v=s$) at time $0$. This event occurs with probability $0$ under the stationary framework, so such an event is often handled as a boundary condition in the stationary equation  (e.g, see \cite{Davi1993}). On the other hand, Palm distributions are defined through time averages under the stationary framework. This enables us to derive the stationary equation without any boundary condition, where the boundary conditions are handled by Palm distributions in the stationary equation. This is quite convenient because the boundary conditions may be only asymptotically satisfied. This is exactly our case.

We now derive a BAR for $X^{(r)}(\cdot)$ using the framework proposed in \cite{Miya2024}. For this, we introduce notations for describing the dynamics of $X^{(r)}(\cdot)$. Let $S$ be its state space, that is, $S = \dd{Z}_{+} \times \dd{R}_{+}^{2}$, and let $C_{b}^{p}(S)$ be the set of all bounded and continuous functions from $S$ to $\dd{R}_{+}$ which are partially differentiable from the right concerning the second and third variables and in which $\frac {\partial^{+}}{\partial x_{i}}f$ is the $i$-th partial derivative from the right. Define operator $\sr{H}$ on $C_{b}^{p}(S)$ which maps $f \in C_{b}^{p}(S)$ to $\sr{H} f \in C_{b}^{p}(S)$ by
\begin{align*}
 & \sr{H} f(\vc{x}) = - \frac {\partial^{+}}{\partial x_{2}} f(\vc{x}) - (c_{1} 1(1 \le x_{1} \le \ell_{1}) + c_{2} 1(x_{1} > \ell_{1})) \frac {\partial^{+}}{\partial x_{3}} f(\vc{x}), \qquad \vc{x} \in S,
\end{align*}
where $\vc{x} = (x_{1},x_{2},x_{3}) \in \dd{R}_{+}^{2}$.

In what follows, we often use a difference operator $\Delta$, which is defined as $\Delta g(t) = g(t) - g(t-)$ for function $g$ from $\dd{R}_{+}$ to $\dd{R}$ which is right continuous with the left-limit. For $t \ge 0$, define $\dd{F}$-adapted process $X^{(r)}_{1}(\cdot)$ and $L^{(r)}_{1}(\cdot)$ as 
\begin{align*}
  & X^{(r)}_{1}(t) = 
\begin{cases}
X^{(r)}(t-), & \Delta N^{(r)}_{e}(t) = 0,\\
 (L^{(r)}(t-) + 1,  T^{(r)}_{e,1}(n_{e,1}(t)), R^{(r)}_{s}(t-)), \quad & \Delta N^{(r)}_{e}(t) = 1, L^{(r)}(t-) < \ell^{(r)}_{1},\\
 (L^{(r)}(t-) + 1,  T^{(r)}_{e,2}(n_{e,2}(t)), R^{(r)}_{s}(t-)), \quad & \Delta N^{(r)}_{e}(t) = 1, L^{(r)}(t-) \ge \ell^{(r)}_{1},
\end{cases}
\\
  & L^{(r)}_{1}(t) = L^{(r)}(t-) + 1(\Delta N^{(r)}_{e}(t) = 1), \qquad R^{(r)}_{1}(t) = (R^{(r)}_{e}(t), R^{(r)}_{s}(t-)),
\end{align*}
where $n_{e,1}(t) = \sum_{n=1}^{N^{(r)}_{e}(t)} 1(L^{(r)}(t^{(r)}_{e,n}-) < \ell^{(r)}_{1})$ and $n_{e,2}(t) = \sum_{n=1}^{N^{(r)}_{e}(t)} 1(L^{(r)}(t^{(r)}_{e,n}-) \ge \ell^{(r)}_{1})$. We call $X^{(r)}_{1}(t)$ an intermediate state. Note that
\begin{align*}
  X^{(r)}(t) = (L^{(r)}_{1}(t) -1, R^{(r)}_{e}(t), T^{(r)}_{s}(N^{(r)}_{s}(t)))
\end{align*}
when $\Delta N^{(r)}_{s}(t) = 1$, and $X^{(r)}(t) = X^{(r)}_{1}(t)$ when $\Delta N^{(r)}_{s}(t) = 0$.

In these definitions, $X^{(r)}(0-)$ is required, but formally not defined. However, this does not cause any problem. Please see (i) of \myrem{BAR-1} for its details.
For convenience, we define another difference operators $\Delta_{e}$ and $\Delta_{s}$ as
\begin{align*}
 & \Delta_{e} f (X^{(r)})(t) = f(X^{(r)}_{1}(t)) - f(X^{(r)}(t-)),\\
 & \Delta_{s} f (X^{(r)})(t) = f(X^{(r)}(t)) - f(X^{(r)}_{1}(t)).
\end{align*}
Further, let $f(X)(t) = f(X(t))$. Then, from the time evolution of $X^{(r)}(t)$, we have
\begin{align}
\label{eq:time-evolution-1}
  f(X^{(r)}(t)) = f(X^{(r)}(0)) & + \int_{0}^{t} \sr{H} f(X^{(r)}(u)) du  \nonumber\\
  & + \sum_{u \in (0,t]} [\Delta_{e} f(X^{(r)})(u) + \Delta_{s} f(X^{(r)})(u)],
\end{align}
where the summation is well defined because $\Delta_{e} f(X^{(r)})(u)$ and $\Delta_{s} f(X^{(r)})(u)$ vanish except for finitely many $u \in (0,t]$. Take the expectation of \eq{time-evolution-1} for $t=1$ by $\dd{P}$, then the stationarity of $X(\cdot)$ and the definition of the Palm distribution yield
\begin{align*}
 & \dd{E}\left[\int_{0}^{1} \sr{H} f(X^{(r)}(u)) du\right] = \int_{0}^{1} \dd{E}\left[\sr{H} f(X^{(r)}(u))\right] du = \dd{E}\left[\sr{H} f(X^{(r)}(0))\right],\\
 & \dd{E}\left[\sum_{u \in (0,1]} \Delta_{v} f(X^{(r)})(u)]\right] = \alpha^{(r)}_{v} \dd{E}^{(r)}_{v}\left[\Delta_{v} f(X^{(r)})(0) \right], \qquad v= e,s.
\end{align*}
Hence, we have a stationary equation:
\begin{align}
\label{eq:BAR-1}
  & \dd{E}\left[ \sr{H}f(X^{(r)}(0)) \right]  \nonumber\\
  & \qquad + \alpha^{(r)}_{e} \dd{E}^{(r)}_{e}\left[\Delta_{e} f(X^{(r)})(0) \right] + \alpha^{(r)}_{s} \dd{E}^{(r)}_{s}\left[\Delta_{s} f(X^{(r)})(0) \right] = 0, \qquad f \in C_{b}^{p},
\end{align}

 The formula \eq{BAR-1} is a special case of the rate conservation law (e.g., see \cite{Miya1994}), and referred as a basic adjoint relationship, BAR in short, then we call analysis based on this BAR as a BAR approach. A major difficulty in this BAR approach is to evaluate the expectations under the Palm distributions (the last two terms in \eq{BAR-1}). Their evaluations will be done for the 2-level $GI/G/1$ queue in \sectn{expansion}.
 
\begin{remark}
\label{myrem:BAR-1}
We give two remarks on Palm distributions. (i) $X^{(r)}(0-)$ is not defined under $\dd{P}$, but its distribution is well defined under the Palm distributions $\dd{P}^{(r)}_{e}$ and $\dd{P}^{(r)}_{e}$ because the definition \eq{Palm 1} only uses $X^{(r)}(t)$ for $t > 0$ under $\dd{P}$. So, $X^{(r)}(0-)$ is a random variable which can be defined under $\dd{P}^{(r)}_{e}$ and $\dd{P}^{(r)}_{s}$. Hence, $\Delta_{e} f(X^{(r)})(0)$ and $\Delta_{s} f(X^{(r)})(0)$  in \eq{BAR-1} are well defined under the Palm distributions.\\
(ii) Under the Palm distributions, $X^{(r)}_{1}(0)$ is an intermediate state between just after when $N^{(r)}_{e}$ is counted at time $0$ and just before when $N^{(r)}_{s}$ is counted at time $0$. This means that we take a so called arrival first framework. As for the intermediate state, see \cite{Miya2024} for its details. We consider $L^{(r)}(0-)$ under $\dd{P}^{(r)}_{e}$, while $L^{(r)}(0)$ under $\dd{P}^{(r)}_{s}$. This prevents to directly refer to the intermediate state $X^{(r)}_{1}(0)$ in our arguments. Of course, the intermediate state is not needed if $N^{(r)}_{e}$ and $N^{(r)}_{s}$ do not have a simultaneous count, but such restriction is avoided by introducing $X^{(r)}_{1}(0)$.
\pend
\end{remark}

We note basic facts on $\alpha^{(r)}_{e}$ and $L^{(r)}(\cdot)$ under the Palm distributions. 
\begin{lemma}
\label{mylem:2LQ-basic-1}
Under conditions (\sect{sequence}.b) and (\sect{sequence}.e), 
\begin{align}
\label{eq:2LQ-balance-1}
 & \dd{P}^{(r)}_{e}[(L^{(r)}(0-) = k] = \dd{P}^{(r)}_{s}[L^{(r)}(0) = k], \qquad k \ge 0,\\
\label{eq:2LQ-alpha 1}
 & \alpha^{(r)}_{e} = \frac {\lambda^{(r)}_{1} \lambda^{(r)}_{2}} {\lambda^{(r)}_{2} \dd{P}^{(r)}_{e}(L^{(r)}(0-) < \ell^{(r)}_{1}) + \lambda^{(r)}_{1} \dd{P}^{(r)}_{e}(L^{(r)}(0-) \ge \ell^{(r)}_{1})},\\
\label{eq:2LQ-alpha-2}
 & \alpha^{(r)}_{e} = c^{(r)}_{1} \mu^{(r)} \dd{P}(1 \le L^{(r)} \le \ell^{(r)}_{1}) + c^{(r)}_{2} \mu^{(r)} \dd{P}(L^{(r)} > \ell^{(r)}_{1}).
\end{align}
\end{lemma}

\begin{proof}
For each integer $k \ge 0$, let $f(z,\vc{x}) = z \wedge (k+1)$ for $(z,\vc{x}) \in S$, then obviously $f \in C_{b}^{p}(S)$. Since $\Delta_{e} f(X^{(r)})(u) = 1(L^{(r)}(u-) \le k, N^{(r)}_{e}(u) - N^{(r)}_{e}(u-) = 1)$ and $\Delta_{s} f(X^{(r)})(u) = - 1(L^{(r)}(u) \le k, N^{(r)}_{s}(u) - N^{(r)}_{s}(u-) = 1)$, we have
\begin{align*}
 & \alpha^{(r)}_{e} \dd{E}^{(r)}_{e}\left[\Delta_{e} f(X^{(r)})(0) \right] + \alpha^{(r)}_{s} \dd{E}^{(r)}_{s}\left[\Delta_{s} f(X^{(r)})(0) \right] \\
 & \quad =  \alpha^{(r)}_{e} \dd{P}^{(r)}_{e}\left[L^{(r)}(0-) \le k\right] - \alpha^{(r)}_{s} \dd{P}^{(r)}_{s}\left[L^{(r)}(0) \le k\right].
\end{align*}
Applying this to \eq{BAR-1} yields
\begin{align*}
  \alpha^{(r)}_{e} \dd{P}^{(r)}_{e}\left[L^{(r)}(0-) \le k\right] - \alpha^{(r)}_{s} \dd{P}^{(r)}_{s}\left[L^{(r)}(0) \le k\right] = 0, \qquad k \ge 0.
\end{align*}
Since $\alpha^{(r)}_{e} = \alpha^{(r)}_{s}$ by \mylem{alpha 1}, the above formula yields \eq{2LQ-balance-1}. We next let $f(z,\vc{x}) = x_{1}$, and apply \eq{BAR-1}, then
\begin{align*}
   - 1 + \alpha^{(r)}_{e} \dd{E}^{(r)}_{e}\left[ T^{(r)}_{e,1} 1(L^{(r)}(0-) <\ell^{(r)}_{1}) + T^{(r)}_{e,2} 1(L^{(r)}(0-) \ge \ell^{(r)}_{1}) \right] = 0
\end{align*}
because $(R^{(r)}_{e})'(t) = -1$. This proves \eq{2LQ-alpha 1}. Similarly, let $f(z,\vc{x}) = x_{2}$, then 
\begin{align}
\label{eq:2LQ-Rs-dif 1}
  (R^{(r)}_{s})'(t) = - c_{1}^{(r)} 1(1 \le L^{(r)}(0) \le \ell^{(r)}_{1}) - c_{2}^{(r)} 1(L^{(r)}(0) > \ell^{(r)}_{1}),
\end{align}
and $\dd{E}^{(r)}_{s}[\Delta_{s} R^{(r)}_{s}(0)] = \dd{E}[T^{(r)}_{s}]$. Hence, \eq{BAR-1} yields \eq{2LQ-alpha-2}.
\end{proof}

Since it follows from \eq{2LQ-alpha 1} and (\sect{sequence}.b) that
\begin{align}
\label{eq:alpha-r-bounds}
  0 < \lambda^{(r)}_{1} \wedge \lambda^{(r)}_{2} \le \alpha^{(r)}_{e} \le \lambda^{(r)}_{1} \vee \lambda^{(r)}_{2} < \infty, \qquad r \in (0,1],
\end{align}
we have the following corollary.
\begin{corollary}
\label{mycor:alpha 2}
Under conditions (\sect{sequence}.b) and (\sect{sequence}.e), $\alpha^{(r)}_{e}$ is uniformly bounded away from $0$ and uniformly finite for all $r \in (0,1]$.
\end{corollary}

To prove \mythr{2LQ-main}, we will fully use the BAR \eq{BAR-1} replacing $\alpha^{(r)}_{s}$ by $\alpha^{(r)}_{e}$ because $\alpha^{(r)}_{s} = \alpha^{(r)}_{e}$ by \mylem{alpha 1}. We first apply an exponential type of a test function to this BAR, and derive various properties on it under the stationary distribution $\dd{P}$ and the Palm distributions $\dd{P}^{(r)}_{e}$ and $\dd{P}^{(r)}_{s}$ in \sectn{asymptotic}, then \mythr{2LQ-main} is proved in \sectn{proof-main}.

\subsection{Asymptotic BAR for the pre-limit process $X^{(r)}(\cdot)$}
\label{mysec:asymptotic}

We will prove \mythr{2LQ-main} using the moment generating functions $\varphi^{(r)}_{i}(\theta_{i})$ and $\varphi^{(r)}_{i|cd}(\theta_{i})$ of $\nu^{(r)}_{i}$ and $\nu^{(r)}_{i|cd}$, respectively, for $i=1,2$. Namely, for $\theta_{1} \in \dd{R}$ and $\theta_{2} \le 0$,
\begin{align*}
 & \varphi^{(r)}_{1}(\theta_{1}) = \dd{E}[e^{\theta_{1} rL^{(r)}}1(L^{(r)} \le \ell^{(r)}_{1})], \qquad \varphi^{(r)}_{2}(\theta_{2}) = \dd{E}[e^{\theta_{2} rL^{(r)}}1(\ell^{(r)}_{1} < L^{(r)})],\\
 & \varphi^{(r)}_{1|cd}(\theta_{1}) = \dd{E}[e^{\theta_{1} rL^{(r)}}|L^{(r)} \le \ell^{(r)}_{1}], \qquad \varphi^{(r)}_{2|cd}(\theta_{2}) = \dd{E}[e^{\theta_{2} rL^{(r)}}|\ell^{(r)}_{1} < L^{(r)}].
\end{align*}
To this end, we introduce test functions for the BAR \eq{BAR-1}. We first define functions $w^{(r)}_{i}$ and $g^{(r)}_{i,\theta}$ for $i=1,2$ and $\theta \in \dd{R}$ as
\begin{align*}
 & w^{(r)}_{1}(z) = 1(z \le \ell^{(r)}_{1}), \qquad w^{(r)}_{2}(z) = 1(z > \ell^{(r)}_{1}), \qquad z \in \dd{Z}_{+}, \\
 & g^{(r)}_{i,\theta}(x_{1},x_{2}) = \exp\left( - \eta^{(r)}_{i} (\theta) (x_{1} \wedge 1/r) - \zeta^{(r)}(\theta) (x_{2} \wedge 1/r)\right), \qquad (x_{1},x_{2}) \in \dd{R}_{+}^{2},
\end{align*}
where $\eta^{(r)}_{i}(\theta), \zeta^{(r)}(\theta)$ are the unique solutions of the following equations for each $\theta \in \dd{R}$. 
\begin{align}
\label{eq:2LQ-boundary 1}
  e^{\theta} \dd{E}\left[e^{-\eta^{(r)}_{i}(\theta) (T^{(r)}_{e,i} \wedge 1/r)}\right] = 1, \quad i=1,2, \quad e^{-\theta} \dd{E}\left[e^{-\zeta^{(r)}(\theta) (T^{(r)}_{s} \wedge 1/r)}\right] = 1,
\end{align}
because $\eta^{(r)}_{i}(\theta)$ is the inverse function of the Laplace transform of the finite positive random variable $T^{(r)}_{e,i} \wedge 1/r$ taking the value $e^{-\theta}$ and $\zeta^{(r)}(\theta)$ is similarly determined. Then, for each $\vc{\theta} \equiv (\theta_{1},\theta_{2}) \in \dd{R}^{2}$, define test function $f^{(r)}_{\vc{\theta}}$ for the BAR as
\begin{align}
\label{eq:f-theta-r}
 & f^{(r)}_{\vc{\theta}}(z,x_{1},x_{2}) = \sum_{i=1,2} e^{\theta_{i} z} g^{(r)}_{i,\theta_{i}}(x_{1},x_{2}) w^{(r)}_{i}(z), \qquad (z,x_{1},x_{2}) \in S.
\end{align}

Functions $\eta^{(r)}_{i}(\theta)$ and $\zeta^{(r)}(\theta)$ play key roles in our analysis. Consider their properties. First, it is easy to see that they are infinitely many differentiable functions of $\theta$. Hence, twice differentiating the equations in \eq{2LQ-boundary 1}, the following lemma can be obtained by their Taylor expansions around the origin. It is formally proved in \cite[Lemma 5.8]{BravDaiMiya2024}.
\begin{lemma}
\label{mylem:2LQ-eta-zeta 1}
Under the assumptions (\sect{sequence}.a)--(\sect{sequence}.d), we have the following expansions as $r \downarrow 0$ for each fixed $\theta \in \dd{R}$.
\begin{align}
\label{eq:2LQ-eta 1}
 & \eta_{i}^{(r)}(r\theta) = \lambda^{(r)}_{i} r \theta + \frac 12 (\lambda^{(r)}_{i})^{3} (\sigma^{(r)}_{e,i})^{2} r^{2} \theta^{2} + o((\theta r)^{2}), \qquad i = 1,2,\\
\label{eq:2LQ-zeta 1}
 & \zeta^{(r)}(r\theta) = - \mu^{(r)} r \theta + \frac 12 (\mu^{(r)})^{3} (\sigma^{(r)}_{s})^{2} r^{2} \theta^{2} + o((\theta r)^{2}).
\end{align}
Furthermore, for $i=1,2$, there are constant $d_{e,i}, d_{s}, a > 0$ such that, for $r \in (0,1]$ and $\theta \in \dd{R}$ satisfying $r|\theta| < a$,
\begin{align}
\label{eq:2LQ-eta-b}
 & \left|\eta^{(r)}_{i}(r\theta) (u_{1} \wedge 1/r) \right| \le |r \theta| d_{e,i} (u_{1} \wedge 1/r), \qquad i=1,2,\\
\label{eq:2LQ-zeta-b}
 & \left|\zeta^{(r)}(r\theta) (u_{2} \wedge 1/r)\right| \le |r \theta| d_{s} (u_{2} \wedge 1/r).
\end{align}
\end{lemma}

\begin{remark}
\label{myrem:2LQ-eta-zeta 1}
In \eq{2LQ-eta 1} and \eq{2LQ-zeta 1}, $o((\theta r)^{2})$ can be replaced by $o(r^{2})$ because $\theta$ is fixed, but $o((\theta r)^{2})$ is convenient for our computation because we need to divide both sides of them by $\theta$ in some cases. We will take this convention throughout the paper as long as it is possibe.
\pend
\end{remark}

Recall that (\sect{sequence}.a)--(\sect{sequence}.e) are always assumed in \sectn{2LQ-BAR} if particular assumptions among them are specified. By \mylem{2LQ-eta-zeta 1}, $g^{(r)}_{i,r\theta_{i}}(R^{(r)})$ is bounded by $e^{|\theta_{i}|(d_{e,i} + d_{s})}$ for sufficiently small $r$ for each $\theta_{1} \in \dd{R}$ and $\theta_{2} \le 0$, so we have the following facts.

\begin{corollary}
\label{mycor:2LQ-limit 1}
Let $R^{(r)} = (R^{(r)}_{e},R^{(r)}_{s})$, then $g^{(r)}_{i,r\theta_{i}}(R^{(r)})$ converges to $1$ as $r \downarrow 0$, and
\begin{align}
\label{eq:2LQ-limit 1}
  \lim_{r \downarrow 0} \dd{E}\left[g^{(r)}_{i,r\theta_{i}}(R^{(r)})\right] = 1, \qquad \lim_{r \downarrow 0} \dd{E}_{u}\left[g^{(r)}_{i,r\theta_{i}}(R^{(r)})\right] = 1, \quad u = e, s.
\end{align}
\end{corollary}

We next introduce auxiliary functions $\psi^{(r)}_{i}(\theta_{i})$ for $i=1,2$ as
\begin{align}
\label{eq:phi-theta-r}
 & \psi^{(r)}_{i}(\theta_{i}) = \dd{E}\left[e^{r\theta_{i} L^{(r)}}g^{(r)}_{i,r \theta_{i}}(R^{(r)}) w^{(r)}_{i}(L^{(r)}) \right], \qquad \theta_{1} \in \dd{R}, \theta_{2} \le 0.
\end{align}
Then, by \mylem{2LQ-eta-zeta 1} and \mycor{2LQ-limit 1}, we have
\begin{align}
\label{eq:2LQ-finite 1}
  \psi^{(r)}_{i}(\theta_{i}) < \infty, \quad \dd{E}_{u}\left[f^{(r)}_{r\vc{\theta}}(L^{(r)},R^{(r)})\right] < \infty, \quad u=e,s, \quad \mbox{uniformly in $r \in (0,1]$},
\end{align}
because $r \theta_{1} L^{(r)}1(L^{(r)} \le \ell^{(r)}_{1}) \le \theta_{1} (\ell + o(r))$ for $\theta_{1} > 0$. These finiteness guarantees our computations for $\theta_{1} \in \dd{R}$ and $\theta_{2} \le 0$. We also note the following fact.
\begin{corollary}
\label{mycor:2LQ-psi-phi-diff}
For $\theta_{1} \in \dd{R}$ and $\theta_{2} \le 0$
\begin{align}
\label{eq:2LQ-psi-phi-diff 2}
  \lim_{r \downarrow 0} \left|\psi^{(r)}_{i}(\theta_{i}) - \varphi^{(r)}_{i}(\theta_{i})\right| = 0, \quad i=1,2.
\end{align}
\end{corollary}
\begin{proof}
By \mycor{2LQ-limit 1}, we have
\begin{align*}
  \left|\psi^{(r)}_{i}(\theta_{i}) - \varphi^{(r)}_{i}(\theta_{i})\right| & = \left|\dd{E}[e^{r\theta_{i} L^{(r)}}(g_{1,\theta_{i}}(R^{(r)}_{e}) - 1)1( L^{(r)} \le \ell^{(r)}_{1})]\right| \nonumber\\
  & \le e^{|\theta_{1}| r \ell^{(r)}_{1}} \left|\dd{E}[g_{i,\theta_{i}}(R^{(r)}_{e}) - 1]\right| \to 0, \qquad r \downarrow 0.
\end{align*}
This proves \eq{2LQ-psi-phi-diff 2}.
\end{proof}

Based on these facts, we compute $\dd{E}\left[ \sr{H}f(X^{(r)}(0))\right]$ in \eq{BAR-1} for $f^{(r)}_{r\vc{\theta}}$. Then, 
\begin{align}
\label{eq:2LQ-H 1}
   & \dd{E}\left[ \sr{H}f^{(r)}_{r\vc{\theta}}(X^{(r)}(0))\right] = \sum_{i=1,2} \eta_{i}^{(r)}(r\theta_{i}) \psi^{(r)}_{i}(\theta_{i}) \nonumber\\
 & \qquad + \sum_{i=1,2} c^{(r)}_{i} \zeta^{(r)} (r\theta_{i}) \psi^{(r)}_{i}(\theta_{i}) - c^{(r)}_{1} \zeta^{(r)}(r\theta_{1}) \dd{E}\left[1(L^{(r)}=0) g^{(r)}_{1,r \theta_{1}}(R^{(r)}) )\right].
\end{align}
Since $\lambda^{(r)}_{i} - c^{(r)}_{i} \mu^{(r)} = - b_{i} \mu^{(r)} r + o(r)$, it follows from \mylem{2LQ-eta-zeta 1}, $\beta_{i} = 2b_{i}/(c_{i} \sigma_{i}^{2})$ and $\lambda_{i} = c_{i} \mu$ that, each fixed $\theta_{i} \in \dd{R}$,
\begin{align*}
 \eta^{(r)}_{i}(r\theta_{i}) + c^{(r)}_{i} \zeta^{(r)}(r\theta_{i}) & = - b_{i} r^{2} \mu \theta_{i} + \frac 12 c_{i} (\lambda_{i}^{2} \sigma_{e,i}^{2} + \mu^{2} \sigma_{s}^{2}) r^{2} \mu \theta_{i}^{2} + o((\theta_{i} r)^{2})\\
  & = - \frac 12 c_{i} \sigma_{i}^{2} (\beta_{i} - \theta_{i}) \mu \theta_{i} r^{2} + o((\theta_{i} r)^{2}), \qquad r \downarrow 0.
\end{align*}
Hence, from \eq{2LQ-H 1}, we have
\begin{align}
\label{eq:2LQ-H 2}
 & \dd{E}\left[ \sr{H}f^{(r)}_{r\vc{\theta}}(X^{(r)}(0))\right] = \frac 12 \sum_{i=1,2} c_{i} \sigma_{i}^{2} \left(- \beta_{i} + \theta_{i} \right) \mu \theta_{i} r^{2} \psi^{(r)}_{i}(\theta_{i}) \nonumber\\
 & \quad + c^{(r)}_{1} \left(1 - \frac 12 \mu^{2} \sigma_{s}^{2} r \theta_{1}\right) \mu r \theta_{1} \dd{E}\left[1(L^{(r)}=0) g^{(r)}_{1,r \theta_{1}}(R^{(r)})\right] + o((\theta_{1} r)^{2}) + o((\theta_{2} r)^{2}).
\end{align}

We next compute the Palm expectation terms in \eq{BAR-1}. In what follows, we use the following abbreviated notations.
\begin{align}
\label{eq:2LQ-T-simple}
 & \widehat{T}^{(r)}_{e,i} \equiv T^{(r)}_{e,i} \wedge 1/r, \quad i=1,2, \qquad \widehat{T}^{(r)}_{s} \equiv T^{(r)}_{s} \wedge 1/r,\\
\label{eq:2LQ-R-simple}
 & \widehat{R}^{(r)}_{e} \equiv R^{(r)}_{e}(0) \wedge 1/r, \qquad \widehat{R}^{(r)}_{s} \equiv R^{(r)}_{s}(0) \wedge 1/r, \qquad \widehat{R}^{(r)}_{s-} \equiv R^{(r)}_{s}(0-) \wedge 1/r.
\end{align}
Since $g^{(r)}_{i,\theta_{i}}(R^{(r)}(t))$ changes at $t=0$ only when $L^{(r)}(0-) = \ell^{(r)}_{1}$ under $\dd{P}^{(r)}_{e}$ or $L^{(r)}(0) = \ell^{(r)}_{1}$ under $\dd{P}^{(r)}_{s}$, \eq{2LQ-boundary 1}  implies that
\begin{align*}
 & \dd{E}^{(r)}_{e}\left[ \Delta_{e} f^{(r)}_{r\vc{\theta}}(X^{(r)})(0) 1(L^{(r)}(0-) \not= \ell^{(r)}_{1})\right] = 0,\\
 & \dd{E}^{(r)}_{s}\left[\Delta_{s} f^{(r)}_{r\vc{\theta}}(X^{(r)})(0) 1(L^{(r)}(0) \not= \ell^{(r)}_{1})\right] = 0.
\end{align*}
Hence, from the definitions of $X^{(r)}_{1}(0)$ and $\Delta_{e} f^{(r)}_{r\vc{\theta}}(X^{(r)})(0)$, we have
\begin{align}
\label{eq:2LQ-df 1}
 & \dd{E}^{(r)}_{e}[\Delta_{e} f^{(r)}_{r\vc{\theta}}(X^{(r)})(0)] = \dd{E}^{(r)}_{e}[(f^{(r)}_{r\vc{\theta}}(X^{(r)}_{1}(0)) - f^{(r)}_{r\vc{\theta}}(X^{(r)}(0-))) 1(L^{(r)}(0-) = \ell^{(r)}_{1})] \nonumber\\
  & \quad = \dd{E}^{(r)}_{e} \Big[e^{r\theta_{2} (\ell^{(r)}_{1} + 1)} g^{(r)}_{2,r\theta_{2}}(R^{(r)}_{1}(0))1(L^{(r)}_{1}(0) = \ell^{(r)}_{1}+1)  \nonumber\\
  & \hspace{15ex} - e^{r\theta_{1} \ell^{(r)}_{1}} g^{(r)}_{1,r\theta_{1}}(R^{(r)}(0-))) 1(L^{(r)}(0-) = \ell^{(r)}_{1}) \Big] \nonumber\\
 & \quad= e^{r\theta_{2} \ell^{(r)}_{1}} \dd{E}^{(r)}_{e}[e^{-\zeta^{(r)}(r\theta_{2}) \widehat{R}^{(r)}_{s-}} 1(L^{(r)}(0-) = \ell^{(r)}_{1})]  \nonumber\\
  & \hspace{15ex} - e^{r\theta_{1} \ell^{(r)}_{1}} \dd{E}^{(r)}_{e}[e^{-\zeta^{(r)}(r\theta_{1}) \widehat{R}^{(r)}_{s-}} 1(L^{(r)}(0-) = \ell^{(r)}_{1}))], \nonumber\\
  & \quad = (e^{r\theta_{2} \ell^{(r)}_{1}} - e^{r\theta_{1}\ell^{(r)}_{1}}) \dd{E}^{(r)}_{e}[e^{-\zeta^{(r)}(r\theta_{2}) \widehat{R}^{(r)}_{s-}} 1(L^{(r)}(0-) = \ell^{(r)}_{1})]  \nonumber\\
  & \qquad + e^{r\theta_{1} \ell^{(r)}_{1}} \dd{E}^{(r)}_{e}[(e^{-\zeta^{(r)}(r\theta_{2}) \widehat{R}^{(r)}_{s-}} - e^{-\zeta^{(r)}(r\theta_{1}) \widehat{R}^{(r)}_{s-}}) 1(L^{(r)}(0-) = \ell^{(r)}_{1}))],
\end{align}
where the third equality is obtained by applying \eq{2LQ-boundary 1}. Similarly, we have
\begin{align}
\label{eq:2LQ-df 2}
 & \dd{E}^{(r)}_{s}[\Delta_{s} f^{(r)}_{r\vc{\theta}}(X^{(r)})(0)]  = \dd{E}^{(r)}_{s}[(f^{(r)}_{r\vc{\theta}}(X^{(r)}(0)) - f^{(r)}_{r\vc{\theta}}(X^{(r)}_{1}(0))) 1(L^{(r)}(0) = \ell^{(r)}_{1})]\nonumber\\
  & \quad = \dd{E}^{(r)}_{s} \Big[e^{r\theta_{1} \ell^{(r)}_{1}} g^{(r)}_{1,r\theta_{1}}(R^{(r)}(0))1(L^{(r)}(0) = \ell^{(r)}_{1}) \nonumber\\
  & \qquad - e^{r\theta_{2} (\ell^{(r)}_{1}+1)} g^{(r)}_{2,r\theta_{2}}(R^{(r)}_{1}(0))) 1(L^{(r)}_{1}(0-) = \ell^{(r)}_{1}+1) \Big] \nonumber\\
 & \quad = e^{r\theta_{1} (\ell^{(r)}_{1}+1)} \dd{E}^{(r)}_{s}[e^{-\eta^{(r)}_{1}(r\theta_{1}) \widehat{R}^{(r)}_{e}} 1(L^{(r)}(0) = \ell^{(r)}_{1})] \nonumber\\
 & \qquad - e^{r\theta_{2} (\ell^{(r)}_{1}+1)} \dd{E}^{(r)}_{s}[e^{-\eta^{(r)}_{2}(r\theta_{2}) \widehat{R}^{(r)}_{e}} 1(L^{(r)}(0) = \ell^{(r)}_{1})]  \nonumber\\
 & \quad = (e^{r\theta_{1} (\ell^{(r)}_{1}+1)} - e^{r\theta_{2} (\ell^{(r)}_{1}+1)}) \dd{E}^{(r)}_{s}[e^{-\eta^{(r)}_{1}(r\theta_{1}) \widehat{R}^{(r)}_{e}} 1(L^{(r)}(0) = \ell^{(r)}_{1})]  \nonumber\\
  & \qquad + e^{r\theta_{2} (\ell^{(r)}_{1}+1)} \dd{E}^{(r)}_{s}[(e^{-\eta^{(r)}_{1}(r\theta_{1}) \widehat{R}^{(r)}_{e}} - e^{-\eta^{(r)}_{2}(r\theta_{2}) \widehat{R}^{(r)}_{e}}) 1(L^{(r)}(0) = \ell^{(r)}_{1})].
\end{align}

Define $D^{(r)}(\theta_{1},\theta_{2}) = \alpha^{(r)}_{e} \left( \dd{E}^{(r)}_{e}\left[\Delta_{e} f^{(r)}_{r\vc{\theta}}(X^{(r)})(0) \right] + \dd{E}^{(r)}_{s}\left[\Delta_{s} f^{(r)}_{r\vc{\theta}}(X^{(r)})(0) \right] \right)$. Then, by \eq{2LQ-df 1} and \eq{2LQ-df 2},
\begin{align}
\label{eq:2LQ-D 1}
   & D^{(r)}(\theta_{1},\theta_{2}) = \alpha^{(r)}_{e} (e^{r\theta_{2} \ell^{(r)}_{1}} - e^{r\theta_{1} \ell^{(r)}_{1}}) \dd{E}^{(r)}_{e}[e^{-\zeta^{(r)}(r\theta_{2}) \widehat{R}^{(r)}_{s-})} 1(L^{(r)}(0-) = \ell^{(r)}_{1})]  \nonumber\\
  & \quad + \alpha^{(r)}_{e} e^{r\theta_{1} \ell^{(r)}_{1}} \dd{E}^{(r)}_{e}[(e^{-\zeta^{(r)}(r\theta_{2}) \widehat{R}^{(r)}_{s-}} - e^{-\zeta^{(r)}(r\theta_{1}) \widehat{R}^{(r)}_{s-}}) 1(L^{(r)}(0-) = \ell^{(r)}_{1})] \nonumber\\
  & \quad + \alpha^{(r)}_{e} (e^{r\theta_{1} (\ell^{(r)}_{1}+1)} - e^{r\theta_{2} (\ell^{(r)}_{1} + 1)}) \dd{E}^{(r)}_{s}[e^{-\eta^{(r)}_{1}(r\theta_{1}) \widehat{R}^{(r)}_{e}} 1(L^{(r)}(0) = \ell^{(r)}_{1})]  \nonumber\\
  & \quad + \alpha^{(r)}_{e} e^{r\theta_{2} (\ell^{(r)}_{1} + 1)} \dd{E}^{(r)}_{s}[(e^{-\eta^{(r)}_{1}(r\theta_{1}) \widehat{R}^{(r)}_{e}} - e^{-\eta^{(r)}_{2}(r\theta_{2}) \widehat{R}^{(r)}_{e}}) 1(L^{(r)}(0) = \ell^{(r)}_{1})].
\end{align}

Thus, applying \eq{2LQ-H 2}, \eq{2LQ-df 1} and \eq{2LQ-df 2} to \eq{BAR-1}, we have a pre-limit BAR for the 2-level $GI/G/1$ queue.

\begin{lemma}
\label{mylem:2LQ-BAR1}
For each fixed $\theta_{1} \in \dd{R}$ and $\theta_{2} \le 0$, we have, as $r \downarrow 0$,
\begin{align}
\label{eq:2LQ-BAR1}
 & \frac 12 \sum_{i=1,2} c_{i} \sigma_{i}^{2} \big((\beta_{i} - \theta_{i}) \mu \theta_{i} r^{2} + o((\theta_{i} r)^{2})\big) \psi^{(r)}_{i}(\theta_{i}) \nonumber\\
 & \quad - c_{1} \left(\mu r\theta_{1} - \mu^{3} \sigma_{s}^{2} r^{2} \theta_{1}^{2}/2 + o((\theta_{1} r)^{2})\right)  \dd{E}[ 1(L^{(r)} = 0) g^{(r)}_{1, r \theta_{1}}(R^{(r)})] = D^{(r)}(\theta_{1},\theta_{2}).
\end{align}
\end{lemma}

This BAR is called an asymptotic BAR, and a starting point of our proof of \mythr{2LQ-main}.

\subsection{Expansion of the asymptotic BAR}
\label{mysec:expansion}

We aim to derive the limits of $\psi^{(r)}_{1}(\theta_{1})$ and $\psi^{(r)}_{2}(\theta_{2})$ from the BAR \eq{2LQ-BAR1}. For this, we will divide it for either $\theta_{2} = 0$ or $\theta_{1} = 0$ by $r^{2}$, and take its limit as $r \downarrow 0$, expanding the exponential functions in \eq{2LQ-BAR1}. We perform this computations in three steps. Recall that (\sect{sequence}.a)--(\sect{sequence}.e) are assumed but (\sect{sequence}.f) is not assumed throughout \sectn{2LQ-BAR}. \medskip

\noindent{({\bf Step 1})} $\;$ In this step, we show that the moments of $\widehat{T}^{(r)}_{v}$ for $v=(e,1), (e,2), s$ and $\widehat{R}^{(r)}_{v'}$ for $v'=e,s$ are well behaved as $r \downarrow 0$. This is also examined for $\widehat{R}^{(r)}_{s-} 1(L^{(r)} = \ell^{(r)}_{1})$ and $\widehat{R}^{(r)}_{e} 1(L^{(r)} = \ell^{(r)}_{0-})$ under $\dd{P}^{(r)}_{e}$ and $\dd{P}^{(r)}_{s}$, respectively. These are basis for the expansion of the BAR. We then introduce the set of key quantities, $\Delta^{(r)}_{i}$ for $i=1,2$.

\begin{lemma}
\label{mylem:2LQ-basic-2}
For $k =0,1,2$ and $r \in (0,1]$,
\begin{align}
\label{eq:2LQ-Te 1}
 & (k+1) \dd{E}[(R^{(r)}_{e})^{k}] \le \alpha^{(r)}_{e} \max_{i=1,2} \dd{E}[(T^{(r)}_{e,i})^{k+1}], \\
\label{eq:2LQ-Ts 1}
 & (k+1) (\min_{i=1,2} c^{(r)}_{i}) \dd{E}[(R^{(r)}_{s})^{k}] \le \alpha^{(r)}_{e} \dd{E}[(T^{(r)}_{s})^{k+1}] + (k+1) (\min_{i=1,2} c^{(r)}_{i}) \dd{E}[T_{s}^{k}],
\end{align}
and, for $i=1,2$, as $r \downarrow 0$
\begin{align}
\label{eq:2LQ-Tu 1}
 & |\dd{E}[(\widehat{T}^{(r)}_{u})^{k}] - \dd{E}[(T^{(r)}_{u})^{k}]| = \dd{E}[(T^{(r)}_{u})^{k} 1(T^{(r)}_{u} > 1/r)] = o(r^{3-k}), \quad u = (e,i), s,\\
\label{eq:2LQ-Rv 1}
 & |\dd{E}[(\widehat{R}^{(r)}_{v})^{k}] - \dd{E}[(R^{(r)}_{v})^{k}]| = \dd{E}[(R^{(r)}_{v})^{k} 1(R^{(r)}_{v} > 1/r)] = o(r^{2-k}), \qquad v = e, s.
\end{align}
In particular, $\{(R^{(r)}_{e})^{k}; r \in (0,1]\}$ and $\{(R^{(r)}_{s})^{k} 1(L^{(r)}(0) \ge 1); r \in (0,1]\}$ are uniformly integrable for $k=1,2$ by (\sect{sequence}.a).
\end{lemma}
This lemma is easily proved using the BAR \eq{BAR-1}, but its proof is lengthy, so we defer it to \myapp{2LQ-basic-2}. We prepare one more lemma and its corollary, which are helpful for the expansion of the BAR.

\begin{lemma}
\label{mylem:2LQ-R1-1}
For $k=1,2,3$ and $r \in (0,1]$,
\begin{align}
\label{eq:2LQ-Rs 1}
 & c^{(r)}_{1} k \dd{E}[(R^{(r)}_{s})^{k-1} (1 - 1(R^{(r)}_{s} \ge 1/r)) 1(1\le L^{(r)} \le \ell^{(r)}_{1})] \nonumber\\
 & \quad = \alpha^{(r)}_{e} \Big( - \dd{E}^{(r)}_{e}[(\widehat{R}^{(r)}_{s-})^{k} 1(L^{(r)}(0-) = \ell^{(r)}_{1})] + \dd{E}[(\widehat{T}^{(r)}_{s})^{k}] \dd{P}^{(r)}_{s}[L^{(r)}(0) \le \ell^{(r)}_{1}]\Big),\\
\label{eq:2LQ-Rs 2}
 & c^{(r)}_{2} k \dd{E}[(R^{(r)}_{s})^{k-1} (1 - 1(R^{(r)}_{s} \ge 1/r)) 1( L^{(r)} > \ell^{(r)}_{1})]\nonumber\\
 & \quad = \alpha^{(r)}_{e} \Big( \dd{E}^{(r)}_{e}[(\widehat{R}^{(r)}_{s-})^{k} 1(L^{(r)}(0-) = \ell^{(r)}_{1})] + \dd{E}[(\widehat{T}^{(r)}_{s})^{k}] \dd{P}^{(r)}_{s}[L^{(r)}(0) > \ell^{(r)}_{1}]\Big),\\
\label{eq:2LQ-Re 1}
 & k \dd{E}[(R^{(r)}_{e})^{k-1} (1 - 1(R^{(r)}_{e} \ge 1/r)) 1(L^{(r)} \le \ell^{(r)}_{1})]  \nonumber\\
 & \quad = \alpha^{(r)}_{e}\Big( \dd{E}_{s}[(\widehat{R}^{(r)}_{e})^{k} 1(L^{(r)}(0) = \ell^{(r)}_{1})] + \dd{E}[(\widehat{T}^{(r)}_{e,1})^{k}] \dd{P}^{(r)}_{e}[L^{(r)}(0-) < \ell^{(r)}_{1}]\Big),\\
\label{eq:2LQ-Re 2}
 & k \dd{E}[(R^{(r)}_{e})^{k-1} (1 - 1(R^{(r)}_{e} \ge 1/r)) 1(L^{(r)} > \ell^{(r)}_{1})] \nonumber\\
 & \quad = \alpha^{(r)}_{e}\Big(- \dd{E}_{s}[(\widehat{R}^{(r)}_{e})^{k} 1(L^{(r)}(0) = \ell^{(r)}_{1})] + \dd{E}[(\widehat{T}^{(r)}_{e,2})^{k}] \dd{P}^{(r)}_{e}[L^{(r)}(0-) \ge \ell^{(r)}_{1}] \Big).
\end{align}
\end{lemma}
\begin{proof}
For $k \ge 1$, we first prove \eq{2LQ-Rs 1}. We apply $f(z,\vc{x}) = (x_{2} \wedge 1/r)^{k} 1(z \le \ell^{(r)}_{1})$ to \eq{BAR-1}, then we have \eq{2LQ-Rs 1} because
\begin{align*}
   \frac d{dt} f(X^{(r)}(t)) = k (R^{(r)}_{s}(t))^{k-1} 1(R^{(r)}_{s} < 1/r, 1 \le L^{(r)}(t) \le \ell^{(r)}_{1}).
\end{align*}
Similarly, \eq{2LQ-Rs 2}, \eq{2LQ-Re 1} and \eq{2LQ-Re 2} are obtained from \eq{BAR-1} by applying $f(z,\vc{x}) = (x_{2} \wedge 1/r)^{k} 1(z > \ell^{(r)}_{1})$, $f(z,\vc{x}) = (x_{1} \wedge 1/r)^{k} 1(z \le \ell^{(r)}_{1})$ and $f(z,\vc{x}) = (x_{1} \wedge 1/r)^{k} 1(z > \ell^{(r)}_{1})$, respectively.
\end{proof}

The following corollary is immediate from \eq{2LQ-Rs 1}, \eq{2LQ-Re 2}, (\sect{sequence}.a) and \mylem{2LQ-basic-2}.

\begin{corollary}
\label{mycor:2LQ-Re-Rs 1}
For each $k=1,2, 3$, $\dd{E}_{e}[(\widehat{R}^{(r)}_{s-})^{k} 1(L^{(r)}(0-) = \ell^{(r)}_{1})]$ and $\dd{E}_{s}[(\widehat{R}^{(r)}_{e})^{k} 1(L^{(r)}(0) = \ell^{(r)}_{1})]$ are uniformly bounded for $r \in (0,1]$.
\end{corollary}

Note that the limits of $\psi^{(r)}_{1}(\theta_{1})$ and $\psi^{(r)}_{2}(\theta_{2})$ should be obtained from \eq{2LQ-BAR1}. However, we will see that it requires the relation between the limits of $\varphi^{(r)}_{1}(0)$ and $\varphi^{(r)}_{2}(0)$, which is not easy to get from \eq{2LQ-BAR1}. To consider this problem, we introduce the following quantities. Define, for $i=1,2$,
\begin{align}
\label{eq:2LD-Di}
 \Delta^{(r)}_{i} & = \mu^{(r)} \dd{E}^{(r)}_{e}[\widehat{R}^{(r)}_{s-} 1(L^{(r)}(0-) = \ell^{(r)}_{1})]  \nonumber\\
 & \quad + \lambda^{(r)}_{i} \dd{E}^{(r)}_{s}[\widehat{R}^{(r)}_{e} 1(L^{(r)}(0) = \ell^{(r)}_{1})] - \dd{P}^{(r)}_{e}[L^{(r)}(0-) = \ell^{(r)}_{1}].
\end{align}
Then, $\Delta^{(r)}_{i}$ is well defined and finite by \mycor{2LQ-Re-Rs 1}, and we have

\begin{lemma}
\label{mylem:2LQ-R2-1}
\begin{align}
\label{eq:2LQ-R1-1}
 & \alpha^{(r)}_{e} \Delta^{(r)}_{1} = - rb_{1} \mu  \dd{P}[L^{(r)} \le \ell^{(r)}_{1}]  + c^{(r)}_{1}\mu^{(r)} \dd{P}[L^{(r)} = 0] + o(r),\\
\label{eq:2LQ-R2-1}
 & \alpha^{(r)}_{e} \Delta^{(r)}_{2} = r\mu b_{2} \dd{P}[L^{(r)} > \ell^{(r)}_{1}] + o(r).
\end{align}
\end{lemma}

\begin{proof}
For $k=1$, multiply both sides of \eq{2LQ-Rs 1} and \eq{2LQ-Re 1} by $\mu^{(r)}$ and $\lambda^{(r)}_{1}$, respectively, and take their difference, then, noting that $c^{(r)}_{1} \mu^{(r)} - \lambda^{(r)}_{1} = r \mu^{(r)} b_{1} + o(r)$, we have
\begin{align*}
 & (r \mu^{(r)} b_{1} + o(r)) \dd{P}[L^{(r)} \le \ell^{(r)}_{1}] - c^{(r)}_{1}\mu^{(r)} \dd{P}[L^{(r)} = 0]  \nonumber\\
 & \quad = \alpha^{(r)}_{e} \Big(\dd{P}^{(r)}_{e}[L^{(r)}(0-) = \ell^{(r)}_{1}]  \nonumber\\
 & \hspace{10ex} - \mu^{(r)} \dd{E}^{(r)}_{e}[\widehat{R}^{(r)}_{s-} 1(L^{(r)}(0-) = \ell^{(r)}_{1})] - \lambda^{(r)}_{1} \dd{E}_{s}^{(r)}[\widehat{R}^{(r)}_{e}1(L^{(r)}(0) = \ell^{(r)}_{1}]\Big) + o(r),
\end{align*}
because $\dd{P}^{(r)}_{s}[L^{(r)}(0) \le \ell^{(r)}_{1}] - \dd{P}^{(r)}_{e}[L^{(r)}(0-) < \ell^{(r)}_{1}] = \dd{P}^{(r)}_{e}[L^{(r)}(0-) = \ell^{(r)}_{1}]$ by \mylem{2LQ-basic-1}.
This proves \eq{2LQ-R1-1}. Similarly, \eq{2LQ-R2-1} is obtained from \eq{2LQ-Rs 2} and \eq{2LQ-Re 2},
\end{proof}

We will use \mylem{2LQ-R2-1} to determine the limits of $\dd{P}[L^{(r)} \le \ell^{(r)}_{1}]$ and $\dd{P}[L^{(r)} > \ell^{(r)}_{1}]$, equivalently, $\varphi^{(r)}_{1}(0)$ and $\varphi^{(r)}_{2}(0)$, as $r \downarrow 0$ under some conditions.\medskip

\noindent{({\bf Step 2})} $\;$ In this step, we expand $D^{(r)}(\theta,0)$ and $D^{(r)}(0,\theta)$ concerning $r$ as $r \downarrow 0$. We will see that $\Delta^{(r)}_{i}$'s describe the expansion terms of $D^{(r)}(\theta,0)$ of order $r$. We here will introduce $\sr{E}^{(r)}_{i}$'s for describing their expansion terms of order $r^{2}$. 

We start with the following elementary expansions. Recall \myrem{2LQ-eta-zeta 1} for our convention of the order terms.
\begin{lemma}
\label{mylem:2LQ-asymp 1}
For each fixed $\theta_{1}$ and $\theta_{2}$, as $r \downarrow 0$,
\begin{align}
\label{eq:2LQ-exp 1}
 & e^{r \theta_{i} \ell^{(r)}_{1}} = e^{\theta_{i} \ell_{1}} (1 + o(\theta_{i} r)), \quad e^{r\theta_{i} (\ell^{(r)}_{1}+1)} = e^{\theta_{i} \ell_{1}} (1 + \theta_{i}r + o(\theta_{i} r)), \quad i=1,2,\\
\label{eq:2LQ-exp 2}
 & e^{r\theta_{1}\ell^{(r)}_{1}} - e^{r\theta_{2} \ell^{(r)}_{1}} = e^{\theta_{1} \ell_{1}} - e^{\theta_{2} \ell_{1}} + e^{(\theta_{1} \vee \theta_{2}) \ell_{1}} ( o(\theta_{1} r) + o(\theta_{2} r)),\\
\label{eq:2LQ-exp 3}
 & e^{r\theta_{1}(\ell^{(r)}_{1}+1)} - e^{r\theta_{2} (\ell^{(r)}_{1}+1)}  \nonumber\\
 & \quad = e^{\theta_{1}\ell_{1}} - e^{\theta_{2} \ell_{1}} + r (e^{\theta_{1} \ell_{1}} \theta_{1} - e^{\theta_{2} \ell_{1}} \theta_{2}) + e^{(\theta_{1} \vee \theta_{2}) \ell_{1}} ( o(\theta_{1} r) + o(\theta_{2} r)).
\end{align}
\end{lemma}
\begin{lemma}
\label{mylem:2LQ-asymp 2}
Let $S_{n}(x) = \sum_{k=1}^{n} x^{k}$ for $x \in \dd{R}$ and integer $n \ge 1$, then, for each fixed $\theta_{1}$ and $\theta_{2}$, as $r \downarrow 0$, for $i=1,2$,
\begin{align}
\label{eq:2LQ-exp-eta 1}
 & e^{-\eta^{(r)}_{i}(r\theta_{i}) \widehat{R}^{(r)}_{e}} = 1 - r\theta_{i} (\lambda^{(r)}_{i} \widehat{R}^{(r)}_{e}) - 2^{-1} (r\theta_{i})^{2} [ (\lambda^{(r)}_{i} \sigma^{(r)}_{e,i})^{2} (\lambda^{(r)}_{i} \widehat{R}^{(r)}_{e}) - (\lambda^{(r)}_{i} \widehat{R}^{(r)}_{e})^{2} ] \nonumber\\
 & \qquad  + S_{3}(\widehat{R}^{(r)}_{e}) o((r\theta)^{2}),\\
\label{eq:2LQ-exp-zeta 1}
 & e^{-\zeta^{(r)}(r\theta_{i}) \widehat{R}^{(r)}_{s-}} = 1 + r\theta_{i} (\mu^{(r)} \widehat{R}^{(r)}_{s-}) - 2^{-1} (r\theta_{i})^{2} [ (\mu^{(r)} \sigma^{(r)}_{s})^{2} (\mu^{(r)} \widehat{R}^{(r)}_{s-}) - (\mu^{(r)} \widehat{R}^{(r)}_{s-})^{2} ] \nonumber\\
 & \qquad + S_{3}(\widehat{R}^{(r)}_{s-}) o((r\theta)^{2}),\\
\label{eq:2LQ-exp-eta 2}
 & e^{-\eta^{(r)}_{1}(r\theta_{1}) \widehat{R}^{(r)}_{e}} - e^{-\eta^{(r)}_{2}(r\theta_{2}) \widehat{R}^{(r)}_{e}} \nonumber\\
 & \quad = r (\theta_{2} \lambda^{(r)}_{2} - \theta_{1} \lambda^{(r)}_{1}) \widehat{R}^{(r)}_{e} +  2^{-1} r^{2} [\theta_{2}^{2} (\lambda^{(r)}_{2} \sigma^{(r)}_{e,2})^{2} \lambda^{(r)}_{2} - \theta_{1}^{2} (\lambda^{(r)}_{1} \sigma^{(r)}_{e,1})^{2} \lambda^{(r)}_{1}] \widehat{R}^{(r)}_{e} \nonumber\\
 & \qquad - 2^{-1} r^{2} [\theta_{2}^{2} (\lambda^{(r)}_{2})^{2} - \theta_{1}^{2} (\lambda^{(r)}_{1})^{2}] (\widehat{R}^{(r)}_{e})^{2} + S_{3}(\widehat{R}^{(r)}_{e}) o((r\theta)^{2}),\\
\label{eq:2LQ-exp-zeta 2}
 & e^{-\zeta^{(r)}(r\theta_{2}) \widehat{R}^{(r)}_{s-}} - e^{-\zeta^{(r)}(r\theta_{1}) \widehat{R}^{(r)}_{s-}} \nonumber\\
 & \quad = r (\theta_{2} - \theta_{1}) (\mu^{(r)} \widehat{R}^{(r)}_{s-}) - 2^{-1} r^{2} (\theta_{2}^{2} - \theta_{1}^{2}) [(\mu^{(r)} \sigma^{(r)}_{s})^{2} (\mu^{(r)} \widehat{R}^{(r)}_{s-}) - (\mu^{(r)} \widehat{R}^{(r)}_{s-})^{2}] \nonumber\\
 & \qquad + S_{3}(\widehat{R}^{(r)}_{s-}) o((r\theta)^{2}).
\end{align}
\end{lemma}

These lemmas are proved in \myapp{2LQ-asymp}. We first compute $D^{(r)}(\theta,0)$.
\begin{align}
\label{eq:2LQ-D10-1}
   D^{(r)}(\theta,0) & = \alpha^{(r)}_{e} (1 - e^{r\theta \ell^{(r)}_{1}}) \dd{P}^{(r)}_{e}[L^{(r)}(0-) = \ell^{(r)}_{1}]  \nonumber\\
  & \qquad + \alpha^{(r)}_{e} e^{r\theta \ell^{(r)}_{1}} \dd{E}^{(r)}_{e}[(1 - e^{-\zeta^{(r)}(r\theta) \widehat{R}^{(r)}_{s-}}) 1(L^{(r)}(0-) = \ell^{(r)}_{1})] \nonumber\\
  & \qquad + \alpha^{(r)}_{e} (e^{r\theta (\ell^{(r)}_{1}+1)} - 1) \dd{E}^{(r)}_{s}[e^{-\eta^{(r)}_{1}(r\theta) \widehat{R}^{(r)}_{e}} 1(L^{(r)}(0) = \ell^{(r)}_{1})]  \nonumber\\
  & \qquad + \alpha^{(r)}_{e} \dd{E}^{(r)}_{s}[(e^{-\eta^{(r)}_{1}(r\theta) \widehat{R}^{(r)}_{e}} - 1) 1(L^{(r)}(0) = \ell^{(r)}_{1})] \nonumber\\
  & = - \alpha^{(r)}_{e} e^{r\theta \ell^{(r)}_{1}} \dd{E}^{(r)}_{e}[ e^{-\zeta^{(r)}(r\theta) \widehat{R}^{(r)}_{s-}} 1(L^{(r)}(0-) = \ell^{(r)}_{1})] \nonumber\\
  & \qquad + \alpha^{(r)}_{e} e^{r\theta \ell^{(r)}_{1}} e^{r \theta} \dd{E}^{(r)}_{s}[e^{-\eta^{(r)}_{1}(r\theta) \widehat{R}^{(r)}_{e}} 1(L^{(r)}(0) = \ell^{(r)}_{1})].
\end{align}

Define $\sr{E}^{(r)}_{1}$ as
\begin{align}
\label{eq:2LQ-E1}
   \sr{E}^{(r)}_{1} & = 2^{-1}\dd{E}^{(r)}_{e}[ \{(\mu^{(r)} \sigma^{(r)}_{s})^{2} (\mu^{(r)} \widehat{R}^{(r)}_{s-}) - (\mu^{(r)} \widehat{R}^{(r)}_{s-})^{2} \}1(L^{(r)}(0-) = \ell^{(r)}_{1})] \nonumber\\
 & \quad - 2^{-1} \dd{E}^{(r)}_{s}[  \{ ((\lambda^{(r)}_{1} \sigma^{(r)}_{e,1})^{2} + 2) (\lambda^{(r)}_{1} \widehat{R}^{(r)}_{e}) - (\lambda^{(r)}_{1} \widehat{R}^{(r)}_{e})^{2} \} 1(L^{(r)}(0) = \ell^{(r)}_{1})].
\end{align}
Applying \eq{2LQ-exp-zeta 1} and
\begin{align*}
 & e^{r \theta}  e^{-\eta^{(r)}_{1}(r\theta) \widehat{R}^{(r)}_{e}} = (1 + r \theta + o(r\theta)) \{1 - r\theta (\lambda^{(r)}_{1} \widehat{R}^{(r)}_{e}) \nonumber\\
 & \qquad - 2^{-1} (r\theta)^{2} [ (\lambda^{(r)}_{1} \sigma^{(r)}_{e,1})^{2} (\lambda^{(r)}_{1} \widehat{R}^{(r)}_{e}) - (\lambda^{(r)}_{1} \widehat{R}^{(r)}_{e})^{2} ] + o((r\theta)^{2})\} \nonumber\\
 & \quad = 1 - r\theta (-1 + \lambda^{(r)}_{1} \widehat{R}^{(r)}_{e}) - 2^{-1} (r\theta)^{2} [ ((\lambda^{(r)}_{1} \sigma^{(r)}_{e,1})^{2} + 2) (\lambda^{(r)}_{1} \widehat{R}^{(r)}_{e}) - (\lambda^{(r)}_{1} \widehat{R}^{(r)}_{e})^{2} ] \nonumber\\
 & \qquad + S_{3}(\widehat{R}^{(r)}_{e}) o((r\theta)^{2}),
\end{align*}
to \eq{2LQ-D10-1} and recalling the definition of $\Delta^{(r)}_{1}$, we have
\begin{align}
\label{eq:2LQ-D10-2}
 & D^{(r)}(\theta,0) = \alpha^{(r)}_{e} e^{\theta \ell_{1}} (- r\theta \Delta^{(r)}_{1} + (r \theta)^{2} \sr{E}^{(r)}_{1})  + o((r\theta)^{2}),
\end{align}
because $\dd{E}^{(r)}_{s}[S_{3}(\widehat{R}^{(r)}_{e}) 1(L^{(r)}(0) = \ell^{(r)}_{1})]$ is uniformly bounded for $r \in (0,1]$ by \mycor{2LQ-Re-Rs 1}.

Similarly, define $\sr{E}^{(r)}_{2}$ as
\begin{align}
\label{eq:2LQ-E2}
  \sr{E}^{(r)}_{2} & = - 2^{-1} \dd{E}^{(r)}_{e} [ (\mu^{(r)} \sigma^{(r)}_{s})^{2} (\mu^{(r)} \widehat{R}^{(r)}_{s-}) - (\mu^{(r)} \widehat{R}^{(r)}_{s-})^{2} ] 1(L^{(r)}(0-) = \ell^{(r)}_{1})] \nonumber\\
  & \quad + 2^{-1} \dd{E}^{(r)}_{s}[\{ (\lambda^{(r)}_{2} (\sigma^{(r)}_{e,2})^{2} + 2) (\lambda^{(r)}_{2} \widehat{R}^{(r)}_{e}) - (\lambda^{(r)}_{2} \widehat{R}^{(r)}_{e})^{2} \} 1(L^{(r)}(0) = \ell^{(r)}_{1})].
\end{align}
Then, $D^{(r)}(0,\theta)$ is computed from \eq{2LQ-D 1} as
\begin{align}
\label{eq:2LQ-D01-2}
  & D^{(r)}(0,\theta) = \alpha^{(r)}_{e} (e^{r\theta \ell^{(r)}_{1}} - 1) \dd{E}^{(r)}_{e}[e^{-\zeta^{(r)}(r\theta) \widehat{R}^{(r)}_{s})} 1(L^{(r)}(0-) = \ell^{(r)}_{1})]  \nonumber\\
  & \quad + \alpha^{(r)}_{e} \dd{E}^{(r)}_{e}[(e^{-\zeta^{(r)}(r\theta) \widehat{R}^{(r)}_{s}} - 1) 1(L^{(r)}(0-) = \ell^{(r)}_{1})] + \alpha^{(r)}_{e} (1 - e^{r\theta (\ell^{(r)}_{1} + 1)}) \dd{P}^{(r)}_{s}[L^{(r)}(0) = \ell^{(r)}_{1}]  \nonumber\\
  & \quad + \alpha^{(r)}_{e} e^{r\theta (\ell^{(r)}_{1} + 1)} \dd{E}^{(r)}_{s}[(1 - e^{-\eta^{(r)}_{2}(r\theta) \widehat{R}^{(r)}_{e}}) 1(L^{(r)}(0) = \ell^{(r)}_{1})] \nonumber\\
  & = \alpha^{(r)}_{e} e^{r\theta \ell^{(r)}_{1}} \big\{\dd{E}^{(r)}_{e}[e^{-\zeta^{(r)}(r\theta) \widehat{R}^{(r)}_{s}} 1(L^{(r)}(0-) = \ell^{(r)}_{1})]  - e^{r\theta}\dd{E}^{(r)}_{s}[e^{-\eta^{(r)}_{2}(r\theta) \widehat{R}^{(r)}_{e}}1(L^{(r)}(0) = \ell^{(r)}_{1})] \big\} \nonumber\\
  & = \alpha^{(r)}_{e} e^{\theta \ell_{1}} \big\{\dd{E}^{(r)}_{e}[1 + r\theta \mu^{(r)} \widehat{R}^{(r)}_{s-} \nonumber\\
  & \qquad - 2^{-1} (r \theta)^{2} [ (\mu^{(r)} \sigma^{(r)}_{s})^{2} (\mu^{(r)} \widehat{R}^{(r)}_{s-}) - (\mu^{(r)} \widehat{R}^{(r)}_{s-})^{2} ] 1(L^{(r)}(0-) = \ell^{(r)}_{1})] \nonumber\\
  & \quad - (1 + r \theta + o(r\theta)) \dd{E}^{(r)}_{s}[1 - r\theta (\lambda^{(r)}_{2} \widehat{R}^{(r)}_{e}) \nonumber\\
 & \qquad - 2^{-1} (r\theta)^{2} [ (\lambda^{(r)}_{2} \sigma^{(r)}_{e,2})^{2} (\lambda^{(r)}_{2} \widehat{R}^{(r)}_{e}) - (\lambda^{(r)}_{2} \widehat{R}^{(r)}_{e})^{2} ] 1(L^{(r)}(0) = \ell^{(r)}_{1})]\big\}  + o((r\theta)^{2}) \nonumber\\
 & = \alpha^{(r)}_{e} e^{\theta \ell_{1}} (r \theta \Delta^{(r)}_{2} + (r \theta)^{2} \sr{E}^{(r)}_{2}) + o((r\theta)^{2}).
\end{align}

\noindent{({\bf Step 3})} $\;$ In this step, we derive asymptotic expansions for $\psi^{(r)}_{i}(\theta_{i})$ for $i=1,2$. First, note that \eq{2LQ-R1-1} and \eq{2LQ-R2-1} can be written as
\begin{align}
\label{eq:2LQ-R1-2}
 & \alpha^{(r)}_{e} \Delta^{(r)}_{1} = - rb_{1} \mu \varphi^{(r)}_{1}(0)  + c^{(r)}_{1}\mu^{(r)} \dd{P}[L^{(r)} = 0] + o(r),\\
\label{eq:2LQ-R2-2}
 & \alpha^{(r)}_{e} \Delta^{(r)}_{2} = r\mu b_{2} \varphi^{(r)}_{2}(0) + o(r).
\end{align}
Then, substituting \eq{2LQ-D10-2} into the asymptotic BAR \eq{2LQ-BAR1} with $\theta_{2} = 0$ and applying \eq{2LQ-R1-2}, we have
\begin{align}
\label{eq:2LQ-BAR 4a}
 & \frac 12 r c_{1} \mu  \sigma_{1}^{2} \big(\beta_{1} - \theta_{1}\big) \psi^{(r)}_{1}(\theta_{1}) =  c_{1} \mu \dd{P}[L^{(r)} = 0] - e^{\theta_{1} \ell_{1}} \alpha^{(r)}_{e} (\Delta^{(r)}_{1} - r \theta_{1} \sr{E}^{(r)}_{1}) + \theta_{1} o(r), \nonumber\\
 & \quad = c_{1} \mu (1 - e^{\theta_{1} \ell_{1}}) \dd{P}[L^{(r)} = 0] + r e^{\theta_{1}\ell_{1}} \big( b_{1} \mu  \varphi^{(r)}_{1}(0) +  \theta_{1} \alpha^{(r)}_{e} \sr{E}^{(r)}_{1}\big) + \theta_{1} o(r),
\end{align}
Similarly, substituting \eq{2LQ-D01-2} into \eq{2LQ-BAR1} with $\theta_{1} = 0$ and applying \eq{2LQ-R2-2}, we have
\begin{align}
\label{eq:2LQ-BAR 4b}
 & \frac 12 r c_{2} \mu \sigma_{2}^{2} \big(\beta_{2} - \theta_{2}\big) \psi^{(r)}_{2}(\theta_{2}) = \alpha^{(r)}_{e} e^{\theta_{2}\ell_{1}} \big(\Delta^{(r)}_{2} + r \theta_{2} \sr{E}^{(r)}_{2} \big) + \theta_{2} o(r) \nonumber\\
 & \quad = r e^{\theta_{2}\ell_{1}} \big( \mu b_{2} \varphi^{(r)}_{2}(0) + \theta_{2} \alpha^{(r)}_{e} \sr{E}^{(r)}_{2} \big) + \theta_{2} o(r).
\end{align}

We next compute $\dd{P}[L^{(r)} = 0]$ for \eq{2LQ-BAR 4a}.
\begin{lemma}
\label{mylem:2LQ-L0-1}
\begin{align}
\label{eq:2LQ-L0-1}
  \dd{P}[L^{(r)} = 0] =  \begin{cases}
  \frac {r} {2 c_{1} \mu \ell_{1}} ( c_{1} \mu \sigma_{1}^{2} \varphi^{(r)}_{1}(0) + 2 \alpha^{(r)}_{e} \sr{E}^{(r)}_{1}) + o(r), & b= 0,\\
  \frac {r e^{\beta_{1} \ell_{1}}}{c_{1} \mu (e^{\beta_{1} \ell_{1}} - 1)} \big(b_{1} \mu \varphi^{(r)}_{1}(0) + \beta_{1} \alpha^{(r)}_{e} \sr{E}^{(r)}_{1} \big) + o(r),  & b \not= 0.
\end{cases}
\end{align}
\end{lemma}

\begin{proof}
From \eq{2LQ-BAR 4a}, we have
\begin{align}
\label{eq:2LQ-L0-2}
  \dd{P}[L^{(r)} = 0] & = r \frac {c_{1} \mu \sigma_{1}^{2} \big(\theta_{1} - \beta_{1}\big) \psi^{(r)}_{1}(\theta_{1}) + 2 e^{\theta_{1} \ell_{1}} (b_{1} \mu \varphi^{(r)}_{1}(0) + \theta_{1} \alpha^{(r)}_{e} \sr{E}^{(r)}_{1})} {2 c_{1} \mu (e^{\theta_{1} \ell_{1}} - 1)} + o(r).
\end{align}
If $b_{1}=0$, then $\beta_{1} = b_{1} = 0$, and therefore \eq{2LQ-L0-2} yields
\begin{align*}
  \dd{P}[L^{(r)} = 0] & =  r \lim_{\theta_{1} \to 0} \frac {\theta_{1} (c_{1} \mu \sigma_{1}^{2} \psi^{(r)}_{1}(\theta_{1}) + 2 e^{\theta_{1} \ell_{1}}\alpha^{(r)}_{e} \sr{E}^{(r)}_{1})} {2 c_{1} \mu (e^{\theta_{1} \ell_{1}} - 1)} + o(r)  \nonumber\\
  & = \frac {r} {2 c_{1} \mu \ell_{1}} ( c_{1} \mu \sigma_{1}^{2} \varphi^{(r)}_{1}(0) + 2 \alpha^{(r)}_{e} \sr{E}^{(r)}_{1}) + o(r),
\end{align*}
which proves \eq{2LQ-L0-1} for $b_{1} = 0$. If $b_{1} \not= 0$, then substituting $\theta_{1} = \beta_{1}$ into \eq{2LQ-L0-2} yields
\begin{align*}
  \dd{P}[L^{(r)} = 0] = \frac {r e^{\beta_{1} \ell_{1}}}{c_{1} \mu (e^{\beta_{1} \ell_{1}} - 1)} \big(b_{1} \mu \varphi^{(r)}_{1}(0) + \beta_{1} \alpha^{(r)}_{e} \sr{E}^{(r)}_{1} \big) + o(r),
\end{align*}
which is \eq{2LQ-L0-1} for $b_{1} \not= 0$.
\end{proof}

We also compute $\alpha^{(r)}_{e} \Delta^{(r)}_{1}$ to be used later.

\begin{lemma}
\label{mylem:2LQ-D1-2}
\begin{align}
\label{eq:2LQ-D1-2}
 & \alpha^{(r)}_{e} \Delta^{(r)}_{1} = \begin{cases}
  \frac {r} {2  \ell_{1}} ( c_{1} \mu \sigma_{1}^{2} \varphi^{(r)}_{1}(0) + 2 \alpha^{(r)}_{e} \sr{E}^{(r)}_{1}) + o(r),  & b_{1} = 0, \\
  \frac {r } {e^{\beta_{1} \ell_{1}} - 1} \Big(b_{1} \mu \varphi^{(r)}_{1}(0) + \frac 1{c_{1} \mu} \beta_{1} e^{\beta_{1} \ell_{1}} \alpha^{(r)}_{e} \sr{E}^{(r)}_{1} \Big) + o(r), & b \not= 0.
\end{cases}
\end{align}
\end{lemma}
This lemma is immediate from substituting $\dd{P}^{(r)}_{e}[L^{(r)} = 0]$ of \mylem{2LQ-L0-1} into \eq{2LQ-R1-2}. We are now ready to compute $\psi^{(r)}_{i}(\theta_{i})$ for $i = 1,2$, which will be a key to prove \mythr{2LQ-main}.

\begin{lemma}
\label{mylem:2LQ-psi12}
\begin{align}
\label{eq:2LQ-psi1}
 \psi^{(r)}_{1}(\theta_{1}) & = \begin{cases}
 \frac {e^{\theta_{1} \ell_{1}} - 1} {\ell_{1} \theta_{1}} \Big(\varphi^{(r)}_{1}(0) + \frac {2 \alpha^{(r)}_{e} \sr{E}^{(r)}_{1}}{c_{1} \mu \sigma_{1}^{2}} \Big) - \frac {2 \alpha^{(r)}_{e} \sr{E}^{(r)}_{1}}{c_{1} \mu \sigma_{1}^{2}} e^{\theta_{1} \ell_{1}} + o(1),  & b_{1} = 0,  \\
\frac {\beta_{1} }{ e^{\beta_{1} \ell_{1}} - 1} \frac {e^{\beta_{1} \ell_{1}} - e^{\theta_{1} \ell_{1}}} {\beta_{1} - \theta_{1}} \Big( \varphi^{(r)}_{1}(0) + \frac {2 \alpha^{(r)}_{e} \sr{E}^{(r)}_{1}} {c_{1} \mu \sigma_{1}^{2}} \Big) - \frac {2\alpha^{(r)}_{e} \sr{E}^{(r)}_{1} } { c_{1} \mu \sigma_{1}^{2}} e^{\theta_{1} \ell_{1}} + o(1), & b_{1} \not= 0,
\end{cases}\\
\label{eq:2LQ-psi2}
    \psi^{(r)}_{2}(\theta_{2}) & = \frac {\beta_{2} e^{\theta_{2} \ell_{1}}} {\beta_{2} - \theta_{2}} \Big( \varphi^{(r)}_{2}(0) + \frac {2 \alpha^{(r)}_{e} \sr{E}^{(r)}_{2}} {c_{2} \mu \sigma_{2}^{2} } \Big) - \frac {2 \alpha^{(r)}_{e} \sr{E}^{(r)}_{2}} {c_{2} \mu \sigma_{2}^{2} } e^{\theta_{2} \ell_{1}} + o(1).
\end{align}
\end{lemma}

\begin{proof}
We first compute $\psi^{(r)}_{1}(\theta_{1})$. Assume that $b_{1} = 0$. Substituting $\dd{P}^{(r)}_{e}[L^{(r)} = 0]$ of \mylem{2LQ-L0-1} into \eq{2LQ-BAR 4a} for $\beta_{1} = b_{1} = 0$, we have
\begin{align*}
 & - r c_{1} \mu \sigma_{1}^{2} \theta_{1} \psi^{(r)}(\theta_{1}) \nonumber\\
 & \quad =  (1 - e^{\theta_{1} \ell_{1}}) \frac {r} {\ell_{1}} ( c_{1} \mu \sigma_{1}^{2} \varphi^{(r)}_{1}(0) + 2 \alpha^{(r)}_{e} \sr{E}^{(r)}_{1}) + r 2 e^{\theta_{1} \ell_{1}} \theta_{1} \alpha^{(r)}_{e} \sr{E}^{(r)}_{1} + o(r),
\end{align*}
which proves \eq{2LQ-psi1} for $b_{1} = 0$. Next assume that $b_{1} \not= 0$. Similarly to the case of $b_{1} = 0$, it follows from \mylem{2LQ-L0-1} and \eq{2LQ-BAR 4a} for $b_{1} \not= 0$ that
\begin{align*}
 & \frac {r}{2} c_{1} \mu \sigma_{1}^{2} (\beta_{1} - \theta_{1}) \psi^{(r)}(\theta_{1}) \nonumber\\
 & \quad = \frac {r e^{\beta_{1} \ell_{1}}(1 - e^{\theta_{1} \ell_{1}}) }{ (e^{\beta_{1} \ell_{1}} - 1)} \big(b_{1} \mu \varphi^{(r)}_{1}(0) + \beta_{1} \alpha^{(r)}_{e} \sr{E}^{(r)}_{1} \big) + r e^{\theta_{1} \ell_{1}} \big( b_{1} \mu \varphi^{(r)}(0) + \theta_{1} \alpha^{(r)}_{e} \sr{E}^{(r)}_{1} \big) + o(r) \nonumber\\
 & \quad = \frac {r b_{1} \mu \varphi^{(r)}(0) }{ e^{\beta_{1} \ell_{1}} - 1} \big(e^{\beta_{1} \ell_{1}} - e^{\theta_{1} \ell_{1}}) + r \frac {\beta_{1} e^{\beta_{1} \ell_{1}}(1 - e^{\theta_{1} \ell_{1}}) + \theta_{1} e^{\theta_{1} \ell_{1}} (e^{\beta_{1} \ell_{1}} - 1)} { (e^{\beta_{1} \ell_{1}} - 1)} \alpha^{(r)}_{e} \sr{E}^{(r)}_{1} + o(r).
\end{align*}
Hence, we have
\begin{align*}
  \psi^{(r)}_{1}(\theta_{1}) & = \frac {\beta_{1} \varphi^{(r)}(0) }{ e^{\beta_{1} \ell_{1}} - 1} \frac {e^{\beta_{1} \ell_{1}} - e^{\theta_{1} \ell_{1}}} {\beta_{1} - \theta_{1}} \nonumber\\
  & \quad + \frac { \beta_{1} e^{\beta_{1} \ell_{1}} - \theta_{1} e^{\theta_{1} \ell_{1}} + (\theta_{1} - \beta_{1}) e^{\theta_{1} \ell_{1}} e^{\beta_{1} \ell_{1}} } { c_{1} \mu \sigma_{1}^{2} (e^{\beta_{1} \ell_{1}} - 1) (\beta_{1} - \theta_{1}) } 2 \alpha^{(r)}_{e} \sr{E}^{(r)}_{1} + o(r) \nonumber\\
  & = \frac {\beta_{1} \varphi^{(r)}(0) }{ e^{\beta_{1} \ell_{1}} - 1} \frac {e^{\beta_{1} \ell_{1}} - e^{\theta_{1} \ell_{1}}} {\beta_{1} - \theta_{1}} \nonumber\\
  & \quad + \frac { \beta_{1} (e^{\beta_{1} \ell_{1}} - e^{\theta_{1} \ell_{1}} ) + (\beta_{1} - \theta_{1}) e^{\theta_{1} \ell_{1}} - (\beta_{1} -\theta_{1}) e^{\theta_{1} \ell_{1}} e^{\beta_{1} \ell_{1}} } { c_{1} \mu \sigma_{1}^{2} (e^{\beta_{1} \ell_{1}} - 1) (\beta_{1} - \theta_{1}) } 2 \alpha^{(r)}_{e} \sr{E}^{(r)}_{1} + o(r) \nonumber\\
  & = \frac {\beta_{1} }{ e^{\beta_{1} \ell_{1}} - 1} \frac {e^{\beta_{1} \ell_{1}} - e^{\theta_{1} \ell_{1}}} {\beta_{1} - \theta_{1}} \Big( \varphi^{(r)}(0) + \frac {2 \alpha^{(r)}_{e} \sr{E}^{(r)}_{1}} {c_{1} \mu \sigma_{1}^{2}} \Big) - \frac {2 e^{\theta_{1} \ell_{1}} } { c_{1} \mu \sigma_{1}^{2}} \alpha^{(r)}_{e} \sr{E}^{(r)}_{1} + o(r)
\end{align*}
Thus, we have \eq{2LQ-psi1} for $b_{1} \not= 0$. We next compute $\psi^{(r)}_{2}(\theta_{2})$. From \eq{2LQ-BAR 4b}, we have
\begin{align*}
   & \psi^{(r)}_{2}(\theta_{2}) = \frac {e^{\theta_{2}\ell_{1}} }{\beta_{2} - \theta_{2}}\Big(\beta_{2} \varphi^{(r)}_{2}(0) + (\beta_{2} - (\beta_{2} - \theta_{2} )) \frac {2 \alpha^{(r)}_{e} \sr{E}^{(r)}_{2}} {c_{2} \mu \sigma_{2}^{2}} \Big) + \theta_{2} o(r),
\end{align*}
which proves \eq{2LQ-psi2}.
\end{proof}

\section{Proof of \mythr{2LQ-main}}
\label{mysec:proof-main}
\setnewcounter

From \mylem{2LQ-psi12}, we can see that, if
\begin{align}
\label{eq:E-vanish}
  \lim_{r \downarrow 0} \sr{E}^{(r)}_{1} = \lim_{r \downarrow 0} \sr{E}^{(r)}_{2} = 0,
\end{align}
then the limits of $\psi^{(r)}_{1}(\theta_{1})$ and $\psi^{(r)}_{1}(\theta_{2})$ can be obtained as long as the limits of $\varphi^{(r)}_{1}(0)$ and $\varphi^{(r)}_{1}(0)$ exist. This suggests that \eq{E-vanish} must be a key condition for \mythr{2LQ-main}. Taking this into account, we separate the proof of \mythr{2LQ-main} in two steps.

\begin{proposition}
\label{mypro:2LQ-main-2}
If \eq{E-vanish} holds, then \mythr{2LQ-main} holds without the assumption (\sect{sequence}.f).
\end{proposition}

\begin{proposition}
\label{mypro:2LQ-3f}
Under the assumptions (\sect{sequence}.a)--(\sect{sequence}.e), if
\begin{align}
\label{eq:2LQ-E12}
  \lim_{r \downarrow 0} (\sr{E}^{(r)}_{1} + \sr{E}^{(r)}_{2}) = 0,
\end{align}
then \eq{E-vanish} holds. In particular, (\sect{sequence}.f) implies \eq{2LQ-E12} and therefore \eq{E-vanish}.
\end{proposition}

These propositions are proved in the next two subsections. Obviously, they prove \mythr{2LQ-main}. To prove \mypro{2LQ-3f}, we will use the following lemma, which is proved in \myapp{2LQ-lem51}. 

\begin{lemma}
\label{mylem:2LQ-Pell-lim-1}
Under the assumptions (\sect{sequence}.a)--(\sect{sequence}.e), we have
\begin{align}
\label{eq:2LQ-Re-Es-lim-1}
 & \lim_{r \downarrow 0} \dd{E}^{(r)}_{s}[(\widehat{R}^{(r)}_{e})^{i} 1(L^{(r)}(0) = \ell^{(r)}_{1})] = 0, \qquad i=1,2,
\end{align}
if either $\lambda_{2} \not= \lambda_{2}$ or
\begin{align}
\label{eq:2LQ-Pell-lim-1}
  & \lim_{r \downarrow 0} \dd{P}^{(r)}[L^{(r)} = \ell^{(r)}_{1}] = 0.
\end{align}
\end{lemma}

\begin{lemma}
\label{mylem:2LQ-Pell-lim-2}
Under the assumptions (\sect{sequence}.a)--(\sect{sequence}.e), \eq{2LQ-Re-Es-lim-1} implies \eq{2LQ-E12}.
\end{lemma}

\begin{remark}
\label{myrem:2LQ-3f}
By Lemmas \mylemt{2LQ-Pell-lim-1} and \mylemt{2LQ-Pell-lim-2} and \mypro{2LQ-3f}, \eq{2LQ-Pell-lim-1} is sufficient for \eq{E-vanish}. We conjecture that \eq{2LQ-Pell-lim-1} is true because $L^{(r)}(\cdot)$ would weakly converge to a reflected diffusion whose stationary distribution has a density. However, the verification of \eq{2LQ-Pell-lim-1} would need a process limit for $L^{(r)}(\cdot)$, so it is beyond the present framework.
\pend
\end{remark}

\subsection{Proof of \mypro{2LQ-main-2}}
\label{mysec:proof-2LQ-main-2}

Throughout this proof, the parts (i), (ii) and (iii) of \mythr{2LQ-main} under the condition \eq{E-vanish} instead of (\sect{sequence}.f) are denoted by the part (i)', (ii)' and (iii)', respectively. 

We first prove the part (i)'. Since \eq{2LQ-phi-bound 1} holds, \mycor{2LQ-psi-phi-diff} implies that
\begin{align}
\label{eq:2LQ-psi-phi-diff 1}
  \lim_{r \downarrow 0} |\psi^{(r)}_{i}(\theta_{i})/\varphi^{(r)}_{i}(0) - \varphi^{(r)}_{i|cd}(\theta_{i})| = 0.
\end{align}

By this formula, we can immediately see that \mylem{2LQ-psi12} and \eq{E-vanish} yield
\begin{align}
\label{eq:2LQ-phi-limit 1}
 & \lim_{r \downarrow 0} \varphi^{(r)}_{1|cd}(\theta_{1}) = \widetilde{\varphi}_{1|cd}(\theta_{1}) \equiv 
\begin{cases}
 \ds \frac {e^{\theta_{1} \ell_{1}} - 1}{\ell_{1} \theta_{1}}, &  b_{1}=0, \\
 \ds \frac {\beta_{1}} {e^{\beta_{1} \ell_{1}} - 1} \frac {e^{\beta_{1} \ell_{1}}  - e^{\theta_{1} \ell_{1}}} {\beta_{1} - \theta_{1}}, \quad & b_{1} \not= 0,
\end{cases}\\
\label{eq:2LQ-phi-limit 2}
 & \lim_{r \downarrow 0} \varphi^{(r)}_{2|cd}(\theta_{2}) = \widetilde{\varphi}_{2|cd}(\theta_{2}) \equiv e^{\theta_{2} \ell_{1}} \frac {\beta_{2}}{\beta_{2} - \theta_{2}}.
\end{align}
Since $\widetilde{\varphi}_{i|cd}$ is the moment generating function of $\widetilde{\nu}_{i|cd}$, \eq{2LQ-phi-limit 1} and \eq{2LQ-phi-limit 2} are equivalent to that $\nu^{(r)}_{i|cd}$ weakly converges to $\widetilde{\nu}_{i|cd}$ for $i=1,2$, respectively. Thus, the part (i)' is proved.

We next prove the part (ii)'. For this, we compute the limits of $\varphi^{(r)}_{i}(0)$ for $i=1,2$ under the assumptions in the part (ii)'. First consider the case that $T^{(r)}_{s}$ is exponentially distributed for all $r \in (0,t]$. In this case, $R^{(r)}_{s}(0-)$ has the same distribution as $T^{(r)}_{s}$ under $\dd{P}^{(r)}_{e}$ by Lemma 2.1 and Theorem 3.1 of \cite{Miya2024}, so, by \mylem{2LQ-basic-2}  for $k=1$,
\begin{align*}
  \mu^{(r)} \dd{E}^{(r)}_{e}[\widehat{R}^{(r)}_{s-} 1(L^{(r)}(0-) \; = \ell^{(r)}_{1})] & \; = \dd{P}^{(r)}_{e}[L^{(r)}(0-) = \ell^{(r)}_{1}] + o(r).
\end{align*}
Applying this fact to the definitions of $\Delta^{(r)}_{i}$ yields
\begin{align}
\label{eq:2LQ-D-diff-1}
 & \lambda^{(r)}_{2} \Delta^{(r)}_{1} - \lambda^{(r)}_{1} \Delta^{(r)}_{2}  \nonumber\\
 & \quad = (\lambda^{(r)}_{2} - \lambda^{(r)}_{1}) (\mu^{(r)} \dd{E}^{(r)}_{e}[\widehat{R}^{(r)}_{s-} 1(L^{(r)}(0-) = \ell^{(r)}_{1})] - \dd{P}^{(r)}_{e}[L^{(r)}(0-) = \ell^{(r)}_{1}]) = o(r).
\end{align}

Compute $\alpha^{(r)}_{e} \Delta^{(r)}_{1}$ by \mylem{2LQ-D1-2} and \eq{E-vanish}, and recall \eq{2LQ-R2-2}, then
\begin{align}
\label{eq:2LQ-D1-3}
 & \alpha^{(r)}_{e} \Delta^{(r)}_{1} = \begin{cases}
  \frac {r} {2  \ell_{1}} c_{1} \mu \sigma_{1}^{2} \varphi^{(r)}_{1}(0) + o(r),  & b_{1} = 0, \\
  \frac {r } {e^{\beta_{1} \ell_{1}} - 1} b_{1} \mu \varphi^{(r)}_{1}(0) + o(r), & b \not= 0,\\
\end{cases}\\
\label{eq:2LQ-D2-1}
   & \alpha^{(r)}_{e} \Delta^{(r)}_{2} = r\mu b_{2} \varphi^{(r)}_{2}(0) + o(r).
\end{align}

Then, for $b_{1}=0$, \eq{2LQ-D-diff-1} implies that
\begin{align*}
  \lambda_{2} \frac {c_{1} \mu  \sigma_{1}^{2}} {2  \ell_{1}} \varphi^{(r)}_{1}(0) - \lambda_{1} \mu b_{2} \varphi^{(r)}_{2}(0) = o(1).
\end{align*}
Since $\varphi^{(r)}_{1}(0) + \varphi^{(r)}_{2}(0) = 1$ and $\mu c_{i} = \lambda_{i}$ for $i=1,2$, this can be written as
\begin{align}
\label{eq:2LQ-phi0-1}
  \frac {\sigma_{1}^{2}} {2  \ell_{1}} \varphi^{(r)}_{1}(0) = \frac {b_{2}}{c_{2}} (1-\varphi^{(r)}_{1}(0)) + o(1).
\end{align}
Hence, 
\begin{align}
\label{eq:2LQ-varphi-i0-1}
  \lim_{r \downarrow 0} \varphi^{(r)}_{i}(0) = \widetilde{\varphi}_{i}(0) \equiv \begin{cases}
\ds \frac {2b_{2} \ell_{1}}{c_{2} \sigma_{1}^{2} + 2 b_{2} \ell_{1}} = \widetilde{A}_{1} > 0, & i=1,\\
\ds \frac {c_{2} \sigma_{1}^{2}}{c_{2} \sigma_{1}^{2} + 2 b_{2} \ell_{1}} = \widetilde{A}_{2} > 0,  & i=2.
\end{cases}
\end{align}

Similarly, for $b_{1} \not= 0$, from \eq{2LQ-D1-3} and \eq{2LQ-D2-1}, we have
\begin{align}
\label{eq:2LQ-phi0-2}
  \frac {\lambda_{2} b_{1} \mu } {e^{\beta_{1} \ell_{1}} - 1} \varphi^{(r)}_{1}(0) - \lambda_{1} \mu b_{2} \varphi^{(r)}_{2}(0) = o(1).
\end{align}
Hence, $\mu c_{i} = \lambda_{i}$ for $i=1,2$ yields
\begin{align}
\label{eq:2LQ-varphi-i1-1}
  \lim_{r \downarrow 0} \varphi^{(r)}_{i}(0) = \widetilde{\varphi}_{i}(0) \equiv \begin{cases}
\ds \frac {c_{1}b_{2} (e^{\beta_{1} \ell_{1}} - 1)} {c_{2} b_{1} + c_{1} b_{2} (e^{\beta_{1} \ell_{1}} - 1)} = \widetilde{A}_{1} > 0, & i=1,\\
\ds \frac {c_{2}b_{1}} {c_{2}b_{1} + c_{1}b_{2} (e^{\beta_{1} \ell_{1}} - 1)} = \widetilde{A}_{2} > 0, & i=2.
\end{cases}
\end{align}

We next consider the case that $\lambda^{(r)}_{1} - \lambda^{(r)}_{2} = O(r)$. Define $\ol{\varphi}_{1}(0)$ and $\ul{\varphi}_{1}(0)$ as
\begin{align*}
  \ol{\varphi}_{1}(0) = \limsup_{r \downarrow 0} \varphi^{(r)}_{1}(0), \qquad \ul{\varphi}_{1}(0) = \liminf_{r \downarrow 0} \varphi^{(r)}_{1}(0),
\end{align*}
and let $\{\ol{r}_{n} \in (0,1]; n \ge 1\}$ and $\{\ul{r}_{n} \in (0,1]; n \ge 1\}$ be the sequences such that $\varphi^{(\ol{r}_{n})}_{1}(0)$ and $\varphi^{(\ul{r}_{n})}_{1}(0)$ attain $\ol{\varphi}_{1}(0)$ and $\ul{\varphi}_{1}(0)$ as $n \to \infty$, respectively. We aim to show $\ol{\varphi}_{1}(0) = \ul{\varphi}_{1}(0)$ using $\lambda^{(r)}_{1} - \lambda^{(r)}_{2} = O(r)$.

To see this, we first consider $\ol{\varphi}_{1}(0)$. Since $\varphi^{(\ol{r}_{n})}_{1}(0)$ converges to $\ol{\varphi}_{1}(0)$, applying \eq{E-vanish} and \eq{2LQ-psi-phi-diff 2} to \mylem{2LQ-psi12}, we have
\begin{align}
\label{eq:2LQ-phi1n-lim}
  \lim_{n \to \infty} \varphi^{(\ol{r}_{n})}_{1}(\theta_{1}) = \lim_{n \to \infty} \psi^{(\ol{r}_{n})}_{1}(\theta_{1}) = \widetilde{\varphi}_{1|cd}(\theta_{1}) \lim_{n \to \infty} \varphi^{(\ol{r}_{n})}_{1}(0) = \widetilde{\varphi}_{1|cd}(\theta_{1}) \ol{\varphi}_{1}(0).
\end{align}
Since the distribution $\widetilde{\nu}_{1}$ has the bounded density $\overline{\varphi}_{1}(0) h_{1|cd}$ on $[0,\ell_{1}]$, this implies that
\begin{align}
\label{eq:2LQ-Pelln-lim}
 & \lim_{n \to \infty} \dd{P}[L^{(\ol{r}_{n})} = \ell^{(\ol{r}_{n})}_{0}] \le \lim_{n \to \infty} \dd{P}\left[\ol{r}_{n} L^{(\ol{r}_{n})} \in [\ell_{1} - \varepsilon, \ol{r}_{n} \ell^{(\ol{r}_{n})}_{0}]\right] \nonumber\\
 & \quad \le \lim_{n \to \infty} \ol{\varphi}_{1}(0) \int_{\ell_{1}- \varepsilon}^{\ell_{1}} h_{1|cd}(x) dx \le \varepsilon \sup_{x \in [\ell_{1} - \varepsilon, \ell_{1}]} h_{1|cd}(x), \qquad \forall \varepsilon > 0,
\end{align}
because $0 \le \ol{\varphi}_{1}(0) \le 1$. Hence, $\dd{P}[L^{(\ol{r}_{n})} = \ell^{(\ol{r}_{n})}_{0}] = o(1)$ as $n \to \infty$. By \mylem{2LQ-Pell-lim-1}, this implies that
\begin{align}
\label{eq:2LQ-ReEs-rn}
  \lim_{n \to \infty} \dd{E}^{(\ol{r}_{n})}_{s}[\widehat{R}^{(\ol{r}_{n})}_{e} 1(L^{(\ol{r}_{n})}(0) = \ell^{(\ol{r}_{n})}_{0})] = 0.
\end{align}
Then, from the definitions of $\Delta^{(r)}_{i}$ and $\lambda^{(\ol{r}_{n})}_{1} - \lambda^{(\ol{r}_{n})}_{2} = O(\ol{r}_{n})$, we have
\begin{align}
\label{eq:2LQ-D-diff-2}
  \Delta^{(\ol{r}_{n})}_{1} - \Delta^{(\ol{r}_{n})}_{2} = (\lambda^{(\ol{r}_{n})}_{1} - \lambda^{(\ol{r}_{n})}_{2}) \dd{E}^{(\ol{r}_{n})}_{s}[\widehat{R}^{(\ol{r}_{n})}_{e} 1(L^{(\ol{r}_{n})}(0) = \ell^{(\ol{r}_{n})}_{0})] = o(\ol{r}_{n}),
\end{align}
If $b_{1} = 0$, then \eq{2LQ-R2-2} and \mylem{2LQ-D1-2} yield
\begin{align*}
  \alpha^{(\ol{r}_{n})}_{e} (\Delta^{(\ol{r}_{n})}_{1} - \Delta^{(\ol{r}_{n})}_{2}) = \frac {r c_{1} \mu  \sigma_{1}^{2}} {2  \ell_{1}} \varphi^{(\ol{r}_{n})}_{1}(0) - r \mu b_{2} \varphi^{(\ol{r}_{n})}_{2}(0) + o(\ol{r}_{n}),
\end{align*}
Hence, by \eq{2LQ-D-diff-2} and $\varphi^{(\ol{r}_{n})}_{1}(0) + \varphi^{(\ol{r}_{n})}_{2}(0) = 1$,
\begin{align*}
  \frac {c_{1} \mu  \sigma_{1}^{2}} {2  \ell_{1}} \varphi^{(\ol{r}_{n})}_{1}(0) = \mu b_{2} (1 - \varphi^{(\ol{r}_{n})}_{1}(0)) + o(1).
\end{align*}
This is identical with \eq{2LQ-phi0-1} for $r = \ol{r}_{n}$ because $c_{1} = c_{2}$ by $\lambda_{1} = \lambda_{2}$. Hence, we have \eq{2LQ-varphi-i0-1} for $r = \ol{r}_{n}$ and $b_{1} = 0$. Similarly, if $b_{1} \not= 0$, then, again by \eq{2LQ-R2-2}, \eq{2LQ-D2-1} and \eq{2LQ-D-diff-2} for $r = \ol{r}_{n}$, we have
\begin{align*}
  \frac {b_{1} \mu } {e^{\beta_{1} \ell_{1}} - 1} \varphi^{(\ol{r}_{n})}_{1}(0) - \mu b_{2} \varphi^{(\ol{r}_{n})}_{2}(0) = o(1).
\end{align*}
This is identical with \eq{2LQ-phi0-2} for $r = \ol{r}_{n}$ by $\lambda_{1} = \lambda_{2}$. Hence, we have \eq{2LQ-varphi-i1-1} for $r = \ol{r}_{n}$ and $b_{1} \not= 0$. Thus, we have
\begin{align*}
  \ol{\varphi}_{1}(0) = \lim_{n \to \infty} \varphi^{(\ol{r}_{n})}_{1}(0) = \widetilde{\varphi}_{1}(0) > 0.
\end{align*}

We can repeat the exactly same arguments for $\{\ul{r}_{n}; n \ge 0\}$, and therefore we also have
\begin{align*}
  \ul{\varphi}_{1}(0) = \lim_{n \to \infty} \varphi^{(\ul{r}_{n})}_{1}(0) = \widetilde{\varphi}_{1}(0) > 0.
\end{align*}
Hence, $\ol{\varphi}_{1}(0) = \ul{\varphi}_{1}(0)$ as claimed, and therefore $\lim_{r \downarrow 0} \varphi^{(r)}_{1}(0) = \widetilde{\varphi}_{1}(0) > 0$. Hence, the condition \eq{2LQ-phi-bound 1} in the part (i)' is satisfied, so we can use the results in that part. Namely, we have \eq{2LQ-phi-limit 1} and \eq{2LQ-phi-limit 2}, equivalently, \eq{2LQ-cd-limit 1}. From this together with \eq{2LQ-varphi-i0-1} and \eq{2LQ-varphi-i1-1}, we have \eq{2LQ-nu-limit 1}.

It remains to prove \eq{2LQ-alpha-limit 1} in the part (ii)'. To see this, apply the assumption (\sect{sequence}.c), \eq{2LQ-varphi-i0-1}, \eq{2LQ-varphi-i1-1} and \mylem{2LQ-L0-1} to \eq{2LQ-alpha-2}, then we have
\begin{align*}
  \lim_{r \downarrow 0} \alpha^{(r)}_{e} = \lim_{r \downarrow 0} (c^{(r)}_{1} \mu^{(r)} (\varphi^{(r)}_{1}(0) - \dd{P}[L^{(r)}=0]) + c^{(r)}_{2} \mu^{(r)} \varphi^{(r)}_{2}(0)) = \lambda_{1} \widetilde{\varphi}_{1}(0) + \lambda_{2} \widetilde{\varphi}_{2}(0),
\end{align*}
which proves \eq{2LQ-alpha-limit 1}.

Finally, we prove the part (iii)'. Assume that $\lambda_{1} > \lambda_{2}$, then $\Delta^{(r)}_{1} > \Delta^{(r)}_{2}$ by \eq{2LQ-D-diff-2}. For $b=0$, this yields
\begin{align*}
  \frac {c_{1} \mu  \sigma_{1}^{2}} {2  \ell_{1}} \varphi^{(r)}_{1}(0) > \mu c_{2} b_{2} (1 - \varphi^{(r)}_{1}(0)) + o(1),
\end{align*}
and therefore $\varphi^{(r)}_{1}(0) > \widetilde{A}_{1} > 0$. Similarly, for $b \not=0$, we also have $\varphi^{(r)}_{1}(0) > \widetilde{A}_{1} > 0$. Hence, \eq{2LQ-phi-bound 1} holds for $i=1$, so $\nu^{(r)}_{1|cd}$ weakly converges to $\widetilde{\nu}_{1|cd}$ as $r \downarrow 0$ by the part (i)'. Similarly, if $\lambda_{1} < \lambda_{2}$, then \eq{2LQ-phi-bound 1} holds for $i=2$, so $\nu^{(r)}_{2|cd}$ weakly converges to $\widetilde{\nu}_{2|cd}$. Thus, the proof of \mypro{2LQ-main-2} is completed.

\subsection{Proof of \mypro{2LQ-3f}}
\label{mysec:proof-2LQ-3f}

We prove that \eq{2LQ-E12} implies \eq{E-vanish} by reductio ad absurdum. First, assume that $\sr{E}^{(r)}_{1}$ does not vanish as $r \downarrow 0$. Then, there are a sequence $\{r_{n} \in (0,1]; n \ge 0\}$ and a constant $\delta_{1} \not= 0$ such that $r_{n} \downarrow 0$ and $\sr{E}^{(r_{n})}_{1} \to \delta_{1}$ as $n \to \infty$, because $\sr{E}^{(r)}_{1}$ is uniformly bounded for $r \in (0,1]$ by \mycor{2LQ-Re-Rs 1}.

We next choose a subsequence $\{r'_{n} \in (0,1]; n \ge 0\} \subset \{r_{n} \in (0,1]; n \ge 0\}$ such that $\alpha^{(r)}_{e}$ and $\varphi^{(r'_{n})}_{1}(0)$ converge to constants $a_{e} > 0$ and $d_{1} \in [0,1]$, respectively, as $n \to \infty$, which are possible because $\alpha^{(r)}_{e}$ is uniformly bounded away from $0$ and above by \mycor{alpha 2} and $\varphi^{(r'_{n})}_{1}(0) \in [0,1]$. Then, by \mylem{2LQ-psi12}, we have
\begin{align*}
 \lim_{n \to \infty} \varphi^{(r)}_{1}(\theta_{1}) & = \begin{cases}
 \frac {e^{\theta_{1} \ell_{1}} - 1} {\ell_{1} \theta_{1}} \Big(d_{1} + \frac {2 a_{e} \delta_{1}}{c_{1} \mu \sigma_{1}^{2}} \Big) - \frac {2 a_{e} \delta_{1}}{c_{1} \mu \sigma_{1}^{2}} e^{\theta_{1} \ell_{1}},  & b_{1} = 0,  \\
\frac {\beta_{1} }{ e^{\beta_{1} \ell_{1}} - 1} \frac {e^{\beta_{1} \ell_{1}} - e^{\theta_{1} \ell_{1}}} {\beta_{1} - \theta_{1}} \Big( d_{1} + \frac {2 a_{e} \delta_{1}}{c_{1} \mu \sigma_{1}^{2}} \Big) - \frac {2 a_{e} \delta_{1}}{c_{1} \mu \sigma_{1}^{2}} e^{\theta_{1} \ell_{1}}, & b_{1} \not= 0.
\end{cases}
\end{align*}
This shows that $\nu^{(r)}_{1}$ weakly converges to a finite measure on $[0,\ell_{1}]$, but this measure cannot have a negative mass at $\ell_{1}$, so we must have that $\delta_{1} < 0$ because $\delta_{1} \not= 0$. 

We next consider the sequence $\{\sr{E}^{(r'_{n})}_{2}; n \ge 1\}$. If this sequence does not vanish as $n \to \infty$, then there is a subsequence 
$\{r''_{n} \in (0,1]; n \ge 0\} \subset \{r'_{n} \in (0,1]; n \ge 0\}$ such that $\sr{E}^{(r''_{n})}_{2}$ converges some $\delta_{2} \not= 0$ as $n \to \infty$. Then, by the exactly same argument as used for $\{\sr{E}^{(r_{n})}_{1}; n \ge 1\}$ using \eq{2LQ-psi2} instead of \eq{2LQ-psi1}, we must have that $\delta_{2} < 0$. Hence, we have
\begin{align*}
  \lim_{r''_{n} \downarrow 0} (\sr{E}^{(r''_{n})}_{1} + \sr{E}^{(r''_{n})}_{2}) = \delta_{1} + \delta_{2} < 0.
\end{align*}
This contradicts \eq{2LQ-E12}, and therefore $\sr{E}^{(r)}_{1}$ must vanish as $r \downarrow 0$. By \eq{2LQ-E12}, this further implies that $\sr{E}^{(r)}_{2}$ vanishes as $r \downarrow 0$. Hence, we have \eq{E-vanish}.

We now prove that (\sect{sequence}.f) implies \eq{2LQ-E12}. By the definitions of $\sr{E}^{(r)}_{1}$ and $\sr{E}^{(r)}_{2}$, we have
\begin{align}
\label{eq:2LQ-E1+2}
 & \sr{E}^{(r)}_{1} + \sr{E}^{(r)}_{2} = - 2^{-1} \dd{E}^{(r)}_{s}[  \{ ((\lambda^{(r)}_{1} \sigma^{(r)}_{e,1})^{2} + 2) (\lambda^{(r)}_{1} \widehat{R}^{(r)}_{e}) - (\lambda^{(r)}_{1} \widehat{R}^{(r)}_{e})^{2} \} 1(L^{(r)}(0) = \ell^{(r)}_{1})] \nonumber\\
  & \qquad + 2^{-1} \dd{E}^{(r)}_{s}[\{ (\lambda^{(r)}_{2} (\sigma^{(r)}_{e,2})^{2} + 2) (\lambda^{(r)}_{2} \widehat{R}^{(r)}_{e}) - (\lambda^{(r)}_{2} \widehat{R}^{(r)}_{e})^{2} \} 1(L^{(r)}(0) = \ell^{(r)}_{1})].
\end{align}
We separately consider the two cases that $\lambda_{1} = \lambda_{2}$ and $\lambda_{1} \not= \lambda_{2}$. First, consider the case that $\lambda_{1} = \lambda_{2}$. In this case, $\sigma_{e,1}^{2} = \sigma_{e,2}^{2}$ by (\sect{sequence}.f). Hence, \eq{2LQ-E1+2} implies \eq{2LQ-E12}. Next consider the case that $\lambda_{1} \not= \lambda_{2}$. In this case, we have \eq{2LQ-Re-Es-lim-1} by \mylem{2LQ-Pell-lim-1}. Applying \eq{2LQ-Re-Es-lim-1} to \eq{2LQ-E1+2}, we have \eq{2LQ-E12} because $\lambda^{(r)}_{i}$ and $\sigma^{(r)}_{e,i}$ converge to constants $\lambda_{i}$ and $\sigma_{e,i}$, respectively, as $r \downarrow 0$ for $i=1,2$. This completes the proof of \mypro{2LQ-3f}.

\section{Concluding remarks}
\label{mysec:concluding}
\setnewcounter

In this section, we intuitively discuss two questions, what is a process limit for the 2-level $GI/G/1$ queue in heavy traffic and how to verify \eq{2LQ-Ai} in \myrem{2LQ-main}.

\subsection{Process limit}
\label{mysec:process-limit}

Up to now, we have considered the limit of the sequence of the scaled stationary distributions in heavy traffic. From the assumptions (\sect{sequence}.a)--(\sect{sequence}.f), one can guess this limiting distribution must be the stationary distribution of the reflected diffusion on $[0,\infty)$ whose Brownian component changes its drift and variance at the level $\ell_{1}$.

Recall that, if one of the conditions in (ii) of \mythr{2LQ-main} is satisfied, then, for $b_{1} \not= 0$, the stationary distribution $\widetilde{\nu}$ is characterized by the stationary equation, that is, BAR:
\begin{align}
\label{eq:2LQ-RBM 1}
 & \sum_{i=1}^{2} \Big(\! - b_{i} \theta_{i} + \frac 12 c_{i} \sigma_{i}^{2} \theta_{i}^{2}\Big) \widetilde{\varphi}_{i}(\theta_{i}) + \frac {\beta_{1} e^{\beta_{1} \ell_{1}}} {e^{\beta_{1} \ell_{1}} - 1} \theta_{1} \frac {c_{1} \sigma^{2}_{1}}{2} \widetilde{\varphi}_{1}(0) - \frac {\beta_{1}e^{\theta_{1} \ell_{1}}} {e^{\beta_{1} \ell_{1}} - 1}  \theta_{1} \frac {c_{1} \sigma^{2}_{1}}{2} \widetilde{\varphi}_{1}(0) \nonumber\\
 & \qquad + \beta_{2} e^{\theta_{2}\ell_{1}} \theta_{2} \frac {c_{2} \sigma^{2}_{2}}{2} \widetilde{\varphi}_{2}(0) = 0, \qquad \theta \equiv (\theta_{1}, \theta_{2}) \in \dd{R} \times \dd{R}_{-},
\end{align}
which can be obtained from \eq{2LQ-BAR1} and \eq{2LQ-D 1} similarly to $D^{(r)}(\theta,\theta)$ from \eq{D11-2} and \eq{2LQ-BAR2}, recall that $\beta_{i} = 2 b_{i}/(c_{i} \sigma_{i}^{2})$ for $i=1,2$. The BAR for $b_{1} = 0$ is obtained from \eq{2LQ-RBM 1} by letting $b_{1} \downarrow 0$. So, we only consider the case that $b_{1} \not= 0$.

We like to identify this stationary distribution $\widetilde{\nu}$ by that of a reflected diffusion process. Assume that the diffusion scaled process $L^{(r)}_{\rm dif}(\cdot) \equiv \{r L^{(r)}(r^{-2} t); t \ge 0\}$ weakly converges to a nonnegative valued stochastic process $Z(\cdot) \equiv \{Z(t); t \ge 0\}$ as $r \downarrow 0$. We aim to characterize this $Z(\cdot)$ by a stochastic integral equation. To this end, define real-valued functions $\sigma(x)$ and $b(x)$ for $x \in \dd{R}_{+}$ as
\begin{align*}
 & b(x) = b_{1} \mu 1(x \le \ell_{1}) + b_{2} \mu 1(x > \ell_{1}), \\
 & \sigma(x) = \sqrt{c_{1}} \mu \sigma_{1} 1(x \le \ell_{1}) + \sqrt{c_{2}} \mu \sigma_{2} 1(x > \ell_{1}), 
\end{align*}
then one may expect that $Z(\cdot)$ satisfies the following stochastic integral equation.
\begin{align}
\label{eq:2LQ-BM-Z 1}
  Z(t) = Z(0) & + \int_{0}^{t} \sigma(Z(u)) dW(u) - \int_{0}^{t} b(Z(u)) du + Y(t), \qquad t \ge 0,
\end{align}
where $W(\cdot)$ is the standard Brownian motion, $Y(\cdot) \equiv \{Y(t); t \ge 0\}$ is a non-deceasing process satisfying that $\int_{0}^{t} 1(Z(u) > 0) Y(du) = 0$ for $t \ge 0$. Recently, Miyazawa \cite{Miya2024b} (see also \cite{AtarCastReim2024a,AtarWola2024,Miya2024d}) shows that \eq{2LQ-BM-Z 1} has a unique weak solution, and $Z(t)$ has a stationary distribution if only if $b_{2} > 0$, which corresponds to $\rho^{(r)}_{2} < 1$. 

Assume that $Z(\cdot)$ is stationary, then, applying the Ito formula to \eq{2LQ-BM-Z 1} for test function $f(x) = e^{\theta x}$, we have
\begin{align}
\label{eq:2LQ-RBM 2}
 & \sum_{i=1}^{2} \Big(\! - b_{i} \theta + \frac 12 c_{i} \sigma_{i}^{2} \theta^{2}\Big) \widehat{\varphi}_{i}(\theta) + \frac {\beta_{1} e^{\beta_{1} \ell_{1}}} {e^{\beta_{1} \ell_{1}} - 1} \theta \frac {c_{1} \sigma^{2}_{1}}{2} \widehat{\varphi}_{1}(0) = 0, \qquad \theta \le 0.
\end{align}
where $\widehat{\varphi}_{1}(\theta) = \dd{E}[e^{\theta Z} 1(Z \le \ell_{1})]$ and $\widehat{\varphi}_{2}(\theta) = \dd{E}[e^{\theta Z} 1(Z > \ell_{1})]$ for random variable $Z$ subject to the stationary distribution of the $Z(\cdot)$. \eq{2LQ-RBM 2} is the stationary equation for $Z(\cdot)$. Then, we can see that \eq{2LQ-RBM 1} for $\theta_{1} = \theta_{2} = \theta$ agrees with \eq{2LQ-RBM 2} if and only if $c_{1} = c_{2}$, equivalently, $\lambda_{1} = \lambda_{2}$. Hence, if $L^{(r)}_{\rm dif}(\cdot)$ weakly converges to $Z(\cdot)$ as $r \downarrow 0$, then we must have $\lambda_{1} = \lambda_{2}$. Of course, this does not exclude the possibility that $L^{(r)}_{\rm dif}(\cdot)$ weakly converges to some process other than $Z(\cdot)$. We conjecture that this is indeed possible because the stationary distribution of $L^{(r)}_{\rm dif}(\cdot)$ weakly converges to the $\widetilde{\nu}$ of \eq{2LQ-alpha-t-1} for $\lambda_{1} \not= \lambda_{2}$ at least when $T_{s}$ is exponentially distribution by (ii) of \mythr{2LQ-main}. Thus, we have two challenging problems.

\begin{itemize}
\item [(a)] Prove that $L^{(r)}_{\rm dif}(\cdot)$ weakly converges to $Z(\cdot)$ when $\lambda_{1} = \lambda_{2}$.
\item [(b)] Find a process limit of $L^{(r)}_{\rm dif}(\cdot)$ when $\lambda_{1} \not= \lambda_{2}$.
\end{itemize}

Both problems require theoretical investigations based on process limits, which is beyond the present framework. So, we leave them for future study.

\subsection{Intuitive verification of \eq{2LQ-Ai}}
\label{mysec:intuitive}

We intuitively show \eq{2LQ-Ai}. First assume that $b_{1} \not= 0$. Under the stationary setting, let $\tau^{(r)}_{1}$ be the first time for $L^{(r)}(t)$ to get into $\ell^{(r)}_{1}$ from below or above. Suppose that $L^{(r)}(t)$ first moves to $\ell^{(r)}_{1} + 1$ after time $\tau^{(r)}_{1}$, and returns to $\ell^{(r)}_{1}$ at time $\tau^{(r)}_{2}$ after time $B^{(r)}_{2}$. Since the relative frequency of moving up at level $\ell^{(r)}_{1}$ is guessed to be $f^{(r)}_{\rm up} \equiv \lambda^{(r)}_{2}/(\lambda^{(r)}_{2} + c^{(r)}_{1} \mu^{(r)})$ and $B^{(r)}_{2}$ is stochastically almost identical with the busy period of the $GI/G/1$ queue with inter-arrival time $T^{(r)}_{e,2}$ and service time $T^{(r)}_{s}/c^{(r)}_{2}$, the mean sojourn time for $L^{(r)}(t)$ to be above $\ell^{(r)}_{1}$ in the cycle $(\tau^{(r)}_{1}, \tau^{(r)}_{2}]$ is
\begin{align}
\label{eq:2LQ-RBM B2}
  f^{(r)}_{\rm up} \dd{E}[B^{(r)}_{2}] \sim \frac {\lambda^{(r)}_{2}} {\lambda^{(r)}_{2} + c^{(r)}_{1} \mu^{(r)}} \times \frac 1{c^{(r)}_{2} \mu^{(r)} (1 - \rho^{(r)}_{2})} = \frac {c_{2}}{c_{1} + c_{2}} \times \frac 1{r b_{2} \mu} + o(1),
\end{align}
where ``$\sim$'' stands for the same order as $r \downarrow 0$, that is, for non-zero valued functions $f_{1}$ and $f_{2}$, $f_{1}(r) \sim f_{2}(r)$ means that $\lim_{r \downarrow 0} f_{1}(r)/f_{2}(r) = 1$.

On the other hand, when $L^{(r)}(t)$ first moves to $\ell^{(r)}_{1} - 1$ after time $\tau^{(r)}_{1}$, let $B^{(r)}_{1}$ be the time to returns to $\ell^{(r)}_{1}$ from below.  Let $f^{(r)}_{\rm down} = 1 - f^{(r)}_{\rm up}$, then, similarly to $B^{(r)}_{2}$ but, taking $b_{1} \not = 0$ and the reflecting boundary at level $0$ into account, it follows from (4.25) of \cite{Miya2024} that the mean sojourn time for $L^{(r)}(t)$ to be below $\ell^{(r)}_{1}$ in the cycle $(\tau^{(r)}_{1}, \tau^{(r)}_{2}]$ is
\begin{align}
\label{eq:2LQ-RBM B1}
  f^{(r)}_{\rm down} \dd{E}[B^{(r)}_{1}] & \sim \frac {c^{(r)}_{1} \mu^{(r)}} {\lambda^{(r)}_{2} + c^{(r)}_{1} \mu^{(r)}} \times \frac 1{\lambda^{(r)}_{1} \dd{P}^{(r)}_{s}[L^{(r)}(0)=\ell^{(r)}_{1}]}  \nonumber\\
  & = \frac {c_{1}}{c_{1} + c_{2}} \times \frac {e^{\beta_{1} \ell_{1}} - 1} {r b_{1} \mu} + o(1).
\end{align}
Since $f^{(r)}_{\rm up} B^{(r)}_{2} + f^{(r)}_{\rm down} B^{(r)}_{1}$ is the mean cycle time to return to $\ell^{(r)}_{1}$, it follows from the cycle formula of a regenerative process that
\begin{align*}
  \lim_{r \downarrow 0} \nu^{(r)}_{1}(\dd{R}_{+}) = \lim_{r \downarrow 0} \frac {f^{(r)}_{\rm down} \dd{E}[B^{(r)}_{1}]} {f^{(r)}_{\rm down} \dd{E}[B^{(r)}_{1}] + f^{(r)}_{\rm up} \dd{E}[B^{(r)}_{2}]} = \frac {c_{1} b_{2} (e^{\beta_{1} \ell_{1}} - 1)}{c_{2} b_{1} + c_{1} b_{2} (e^{\beta_{1} \ell_{1}} - 1)} = \widetilde{A}_{1}.
\end{align*}
Thus, for $b_{1} \not= 0$, \eq{2LQ-Ai} is intuitively obtained for $i=1$, and it is similarly obtained for $i=2$. The case for $b_{1} = 0$ is similarly derived. Thus, we have support of \eq{2LQ-Ai}. \vspace{-1.5ex}

\appendix

\section*{Appendix}
\setnewcounter
\setcounter{section}{1}

\subsection{Proof of \mylem{alpha 1}}
\label{myapp:alpha 1}

Let $N^{(r)}_{e,\wedge}(\cdot)$ and $N^{(r)}_{e,\vee}(\cdot)$ be the renewal processes whose inter-counting times are $T_{e,1}(n) \wedge T_{e,2}(n)$ and $T_{e,1}(n) \vee T_{e,2}(n)$, respectively, for $n \ge 1$. Obviously, $0 < 1/\dd{E}[T_{e,1} \vee T_{e,2}] = \dd{E}[N^{(r)}_{e,\vee}(t)] \le \dd{E}[N^{(r)}_{e}(t)] \le \dd{E}[N^{(r)}_{e,\wedge}(t)] = 1/\dd{E}[T_{e,1} \wedge T_{e,2}] < \infty$ under the stationary framework. Hence, $\alpha^{(r)}_{e} = \dd{E}[N^{(r)}_{e}(1)]$ is finite and positive. To prove $\alpha^{(r)}_{s} = \alpha^{(r)}_{e}$, consider the stationary process $X^{(r)}_{a}(t) \equiv L^{(r)}(t) \wedge a$ for each constant $a > 0$. Then, $X^{(r)}_{a}(t)$ is piecewise constant, has the increment $X^{(r)}_{a}(u) - X^{(r)}_{a}(u-) = 1(L^{(r)}(u-) \le a-1)$ at the counting instant $u$ of $N^{(r)}_{e}$, and has the increment $X^{(r)}_{a}(u) - X^{(r)}_{a}(u-) = - 1(L^{(r)}(u) \le a)$ at the counting instant $u$ of $N^{(r)}_{s}$. Hence, from an elementary calculus similar to \eq{time-evolution-1}, we have
\begin{align*}
  X^{(r)}_{a}(1) - X^{(r)}_{a}(0) & = \int_{0}^{1} 1(L^{(r)}(u-) \le a-1) N^{(r)}_{e}(du) - \int_{0}^{1} 1(L^{(r)}(u) \le a) N^{(r)}_{s}(du).
\end{align*}
Since $X^{(r)}_{a}(t)$ is stationary by (\sect{sequence}.e) and bounded by $a$, this equation yields
\begin{align*}
  \dd{E}\Big[\int_{0}^{1} 1(L^{(r)}(u-) \le a-1) N^{(r)}_{e}(du)\Big] = \dd{E}\Big[\int_{0}^{1} 1(L^{(r)}(u) \le a) N^{(r)}_{s}(du)\Big].
\end{align*}
Letting $a \to \infty$, the left-hand side of this equation converges to $\dd{E}[N^{(r)}_{e}(1)] = \alpha^{(r)}_{e}$, while it right-hand side converges to $\dd{E}[N^{(r)}_{s}(1)] = \alpha^{(r)}_{s}$. Hence, we  have $\alpha^{(r)}_{s} = \alpha^{(r)}_{e}$.

\subsection{Proofs of \mycor{2LQ-mean-Q}}
\label{myapp:2LQ-mean-Q}

We only verify \eq{2LQ-mean-L 1} for $b_{1} \not= 0$ because it is similarly proved for $b_{1} = 0$. Since $\sr{E}^{(r)}_{1} = o(1)$ by \mypro{2LQ-3f}, it follows from \eq{2LQ-L0-1} and \eq{2LQ-A 1} that
\begin{align*}
  \frac 1r \dd{P}[L^{(r)} = 0] & = \frac {b_{1} e^{\beta_{1} \ell_{1}}}{c_{1}(e^{\beta_{1} \ell_{1}} - 1)} \frac {c_{1}b_{2} (e^{\beta_{1} \ell_{1}} - 1)} {c_{2}b_{1} + c_{1}b_{2} (e^{\beta_{1} \ell_{1}} - 1)} + o(1)  \nonumber\\
  & = \frac {b_{1}b_{2} e^{\beta_{1} \ell_{1}}} {c_{2}b_{1} + c_{1}b_{2} (e^{\beta_{1} \ell_{1}} - 1)} + o(1).
\end{align*}
On the other hand, from \eq{2LQ-nu-tilde}, 
\begin{align*}
   & \lim_{r \downarrow 0} \dd{E}[r L^{(r)}] \\
   & \quad = \frac {c_{1} \sigma_{1}^{2} (e^{\beta_{1} \ell_{1}} - 1 - \beta_{1} \ell_{1})} {2 b_{1} (e^{\beta_{1} \ell_{1}} - 1)} \frac {c_{1}b_{2} (e^{\beta_{1} \ell_{1}} - 1)} {c_{2}b_{1} + c_{1}b_{2} (e^{\beta_{1} \ell_{1}} - 1)} + \left(\ell_{1} + \frac {c_{2} \sigma_{2}^{2}}{2b_{2}} \right) \frac {c_{2}b_{1}} {c_{2}b_{1} + c_{1}b_{2} (e^{\beta_{1} \ell_{1}} - 1)}\\
   & \quad = \frac {c_{1}^{2} b_{2}^{2} \sigma_{1}^{2} (e^{\beta_{1} \ell_{1}} - 1) + 2 b_{1} b_{2} \ell_{1} (c_{2} b_{1} - c_{1} b_{2}) + c_{2}^{2} b_{1}^{2} \sigma_{2}^{2}} {2 b_{1} b_{2} (c_{2}b_{1} + c_{1}b_{2} (e^{\beta_{1} \ell_{1}} - 1))}.
\end{align*}
Hence, we have \eq{2LQ-mean-L 1}.

\subsection{Proof of \mypro{2LQ-exponential}}
\label{myapp:2LQ-exponential-proof}

We prove \mypro{2LQ-exponential} only for $b_{1} \not= 0$ because the case for $b_{1} = 0$ can be similarly proved. We first consider this 2-level $M/M/1$ queue without index $r$ for notational convenience. Because of the exponential assumption, this model can be described by a continuous-time Markov chain. However, $L(\cdot)$ itself cannot be a Markov chain because the arrival rate may depend on the queue length at the last arrival instant as discussed at the end of \sectn{main}. Taking this into account, we describe this model by the continuous time Markov chain with state space
\begin{align*}
  S = \{(j,\ell);j=1,2, \ell=0,1,\ldots, \ell_{1}\} \cup \{(3,\ell); \ell \ge \ell_{1} + 1\},
\end{align*}
and its transition diagram is depicted in \myfig{T-diagram}. Denote this Markov chain by $X_{\exp}(\cdot) \equiv \{(J(t), L(t)); t \ge 0\}$. 
\begin{figure}[h] 
   \centering
   \includegraphics[height=4.5cm]{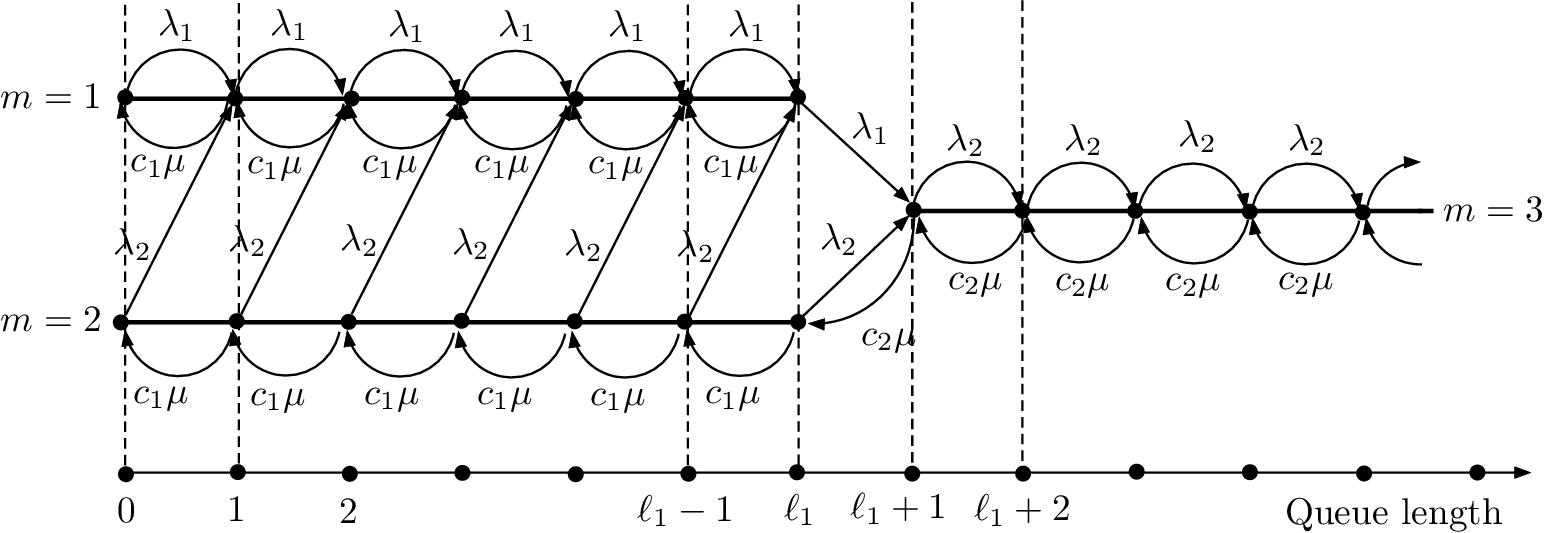} 
  \caption{Transition diagram}
   \label{myfig:T-diagram}
\end{figure}

We partition the state space $S$ of $X_{\exp}(\cdot)$ into the following three disjoint subsets.
\begin{align*}
  S_{j} = \{(j,\ell) \in S; \ell \in [0,\ell_{1}]\}, \quad j=1,2, \qquad S_{3} = \{(3,\ell): \ell \ge \ell_{1}+1\}.
\end{align*}
By the assumptions of \mypro{2LQ-exponential}, $X_{M}(\cdot)$ has a stationary distribution. Let $(J,L)$ be a random vector subject to this stationary distribution, and define $p_{j,\ell}$ as
\begin{align*}
  p_{j,\ell} = \dd{P}[J=j, L=\ell], \qquad j=1,2,3, (j,\ell) \in S_{j}.
\end{align*}

Obviously, the stationary distribution, $\{p_{j,\ell}; (j,\ell) \in S\}$, is uniquely determined by the following stationary equations. For state in $S_{1}$,
\begin{align}
\label{eq:p10}
  & \lambda_{1} p_{1,0} = c_{1} \mu p_{1,1}, \\
\label{eq:p1i}
  & (\lambda_{1} + c_{1} \mu) p_{1,\ell} = \lambda_{1} p_{1,\ell-1} + c_{1} \mu p_{1,\ell+1} + \lambda_{2} p_{2,\ell-1}, \quad \ell \in [1,\ell_{1}-1],\\
\label{eq:p1i+1}
  & (\lambda_{1} + c_{1} \mu) p_{1,\ell_{1}} = \lambda_{1} p_{1,\ell_{1}-1} + \lambda_{2} p_{2,\ell_{1}-1}, 
\end{align}
and, for states in $S_{2}$,
\begin{align}
\label{eq:p20}
  & \lambda_{2} p_{2,0} = c_{1} \mu p_{2,1}, \\
\label{eq:p2i}
  & (\lambda_{2} + c_{1} \mu) p_{2,i} = c_{1} \mu p_{2,\ell+1}, \qquad \ell \in [1,\ell_{1}-1],\\
  \label{eq:p2i+1}
  & (\lambda_{2} + c_{1} \mu) p_{2,\ell_{1}} = c_{2} \mu p_{3,\ell_{1}+1}, 
\end{align}
For states in $S_{3}$, it follows from the equation for $(3,\ell_{1}+1)$ and the local balance that
\begin{align}
\label{eq:p-ell+1}
 & (\lambda_{2} + c_{2} \mu) p_{3,\ell_{1}+1} = \lambda_{1} p_{1,\ell_{1}} + \lambda_{2} p_{2,\ell_{1}} + c_{2} \mu p_{\ell_{1}+2},\\
\label{eq:pi1+}
 & \lambda_{2} p_{3,\ell} = c_{2} \mu p_{3,\ell+1} , \qquad \ell \ge \ell_{1}+1.
\end{align}
In what follows, we use the following notations.
\begin{align*}
  \rho_{j,k} =  \frac {\lambda_{j}}{c_{k} \mu }, \qquad \gamma_{j,k} = \frac {c_{j}\mu} {c_{1} \mu + \lambda_{k}}, \qquad j,k = 1,2.
\end{align*}

We first note that, from \eq{pi1+},
\begin{align}
\label{eq:pi+}
 & p_{3,\ell} = \rho_{22}^{\ell - (\ell_{1}+1)} p_{3,\ell_{1}+1}, \qquad \ell \ge \ell_{1}+1.
\end{align}
Thus, $p_{3,\ell}$ for $\ell \ge \ell_{1}+1$ is simply determined by $p_{3,\ell_{1}+1}$.

We next consider $p_{2,\ell}$ for $\ell \in [0,\ell_{1}-1]$. From \eq{p20} and \eq{p2i}, we have
\begin{align}
\label{eq:p20i}
 & p_{2,0} = \frac {c_{1} \mu} {\lambda_{2}} p_{2,1} = \rho_{21}^{-1}  p_{2,1}, \\
\label{eq:p2i-ell}
 & p_{2,\ell} = \gamma_{12} p_{2,\ell+1} = \gamma_{12}^{\ell_{1} - \ell} p_{2,\ell_{1}}, \quad \ell \in [1,\ell_{1}-1].
\end{align}
Thus, $p_{2,\ell}$ for $\ell \in [0,\ell_{1}-1]$ is determined by $p_{2,\ell_{1}}$. Furthermore, \eq{p2i-ell} can be written by \eq{p20} as
\begin{align}
\label{eq:p2i-ell-}
 & p_{2,\ell} = \gamma_{12}^{-1} p_{2,\ell-1} = \gamma_{12}^{-\ell+1} p_{2,1} = \gamma_{12}^{-\ell+1} \rho_{21} p_{2,0}, \quad \ell \in [1,\ell_{1}-1].
\end{align}

Let us consider relations among the stationary probabilities, $p_{1,0}, p_{1,\ell_{1}}$, $p_{2,\ell_{1}}$ and $p_{3,\ell_{1}+1}$ at the boundaries states of $S_{j}$ for $j=1,2,3$. We call them boundary probabilities. There is one more boundary probability, $p_{2,0}$, but this probability is ignorable compared with $p_{2,\ell_{1}}$ for large $\ell_{1}$ as $\rho_{21} p_{2,0} = \gamma_{12}^{\ell_{1}-1} p_{2,\ell_{1}-1} = \gamma_{12}^{\ell_{1}} p_{2,\ell_{1}}$ by \eq{p2i-ell-} and \eq{p2i-ell}, so we do not consider it. We first note that, from \eq{p2i+1},
\begin{align}
   \label{eq:p2i-ell+}
 & p_{2,\ell_{1}} = \gamma_{22} p_{3,\ell_{1}+1}.
\end{align}
Since $c_{2} \mu p_{\ell_{1}+2} = \lambda_{2} p_{3,\ell_{1}+1}$ by \eq{pi1+} for $\ell=\ell_{1}+1$, \eq{p-ell+1} can be written as
\begin{align}
\label{eq:p1-ell+1}
   & p_{3,\ell_{1}+1} = \rho_{12} p_{1,\ell_{1}} + \rho_{22} p_{2,\ell_{1}},
\end{align}
and, substituting $p_{3,\ell_{1}+1} = \gamma_{22}^{-1} p_{2,\ell_{1}}$ of \eq{p2i-ell+} to this equation, we have
\begin{align*}
  (\gamma_{22}^{-1} - \rho_{22}) p_{2,\ell_{1}} = \rho_{12} p_{1,\ell_{1}}.
\end{align*}
Since $\gamma_{22}^{-1} - \rho_{22} = \frac {c_{1}\mu + \lambda_{2}}{c_{2} \mu} - \frac {\lambda_{2}}{c_{2}\mu} = \frac {c_{1}} {c_{2}} = \rho_{12} \rho_{11}^{-1}$, this equation yields
\begin{align}
  \label{eq:p1-ell+11}
   & p_{2,\ell_{1}} = \rho_{11} p_{1,\ell_{1}}.
\end{align}
Thus, the boundary probabilities $p_{1,\ell_{1}}$, $p_{2,\ell_{1}}$ and $p_{3,\ell_{1}+1}$ are explicitly related by \eq{p2i-ell+} and \eq{p1-ell+11}. We also like to see such a relation between $p_{1,0}$ and $p_{1,\ell_{1}}$, but this is not easy.

To resolve this problem, we consider $p_{1,\ell}$ for $\ell \in [0,\ell_{1}-1]$. From \eq{p1i},
\begin{align*}
  \lambda_{1} p_{1,\ell} - c_{1} \mu p_{1,\ell+1} = \lambda_{1} p_{1,\ell-1} - c_{1} \mu p_{1,\ell} + \lambda_{2} p_{2,\ell-1}, \qquad \ell \in [1,\ell_{1}-1],
\end{align*}
then this equation can be written as
\begin{align*}
  \rho_{11} p_{1,\ell} - p_{1,\ell+1} = \rho_{11} p_{1,\ell-1} - p_{1,\ell} + \rho_{21} p_{2,\ell-1}, \qquad \ell \in [1,\ell_{1}-1].
\end{align*}
Summing up both sides of this equation for $\ell \in [1,m]$ for $m \in [1,\ell_{1}-1]$, we have
\begin{align}
\label{eq:p1-j}
  \rho_{11} p_{1,m} - p_{1,m+1} = \rho_{21} \sum_{\ell=1}^{m} p_{2,\ell-1}, \qquad m \in [0,\ell_{1}-1],
\end{align}
because $\rho_{11} p_{1,0} - p_{1,1} = 0$ by \eq{p10}. This equation is valid also for $m=0$ by the convention that the empty sum vanishes, which we always take.

Hence, applying \eq{p20i} and \eq{p2i-ell} to \eq{p1-j}, we have
\begin{align*}
  \rho_{11} p_{1,m} - p_{1,m+1} & = \rho_{21} \left( \rho_{21}^{-1} p_{2,1} + \sum_{\ell=1}^{m-1} \gamma_{12}^{\ell_{1} - \ell} p_{2,\ell_{1}} \right) \nonumber\\
  & = \left( 1 + \rho_{21} \sum_{\ell=1}^{m-1} \gamma_{12}^{- (\ell-1)} \right) \gamma_{12}^{\ell_{1} - 1} p_{2,\ell_{1}} \nonumber\\
  & = \left( 1 + \rho_{21} \frac {(\gamma_{12}^{- (m-1)} - 1) \gamma_{12}} {1 - \gamma_{12}} \right) \gamma_{12}^{\ell_{1} - 1} p_{2,\ell_{1}}, \qquad m \in [1,\ell_{1}-1].
\end{align*}
Multiplying both sides by $\rho_{11}^{-(m+1)}$,
\begin{align*}
 & \rho_{11}^{-m} p_{1,m} - \rho_{11}^{-(m+1)} p_{1,m+1}  \nonumber\\
  & \quad = \rho_{11}^{-(m+1)} \left( 1 + \rho_{21} \frac {(\gamma_{12}^{- (m-1)} - 1) \gamma_{12}} {1 - \gamma_{12}} 1(m \ge 1) \right) \gamma_{12}^{\ell_{1} - 1} p_{2,\ell_{1}}, \qquad m \in [0,\ell_{1}-1].
\end{align*}
Summing up both sides of this equation from $m=0$ to $m=k - 1\le \ell_{1}-1$, we have
\begin{align*}
 & p_{1,0} - \rho_{11}^{-k} p_{1,k} = \sum_{m=0}^{k-1} \rho_{11}^{-(m+1)} \gamma_{12}^{\ell_{1} - 1} p_{2,\ell_{1}} + \sum_{m=1}^{k-1} \frac { (\rho_{11} \gamma_{12})^{- (m-1)} - \rho_{11}^{-(m-1)} } {1 - \gamma_{12}} \rho_{11}^{-2} \rho_{21} \gamma_{12}^{\ell_{1}} p_{2,\ell_{1}}\nonumber\\
 & \quad = \frac {\rho_{11}^{-k} - 1} {1 - \rho_{11}} \gamma_{12}^{\ell_{1} - 1} p_{2,\ell_{1}} + \left(\frac {(\rho_{11} \gamma_{12})^{-(k-1)} - 1}{(\rho_{11} \gamma_{12})^{-1} - 1} - \frac {\rho_{11}^{-(k-1)} - 1} {\rho_{11}^{-1} - 1} \right)  \frac {\rho_{11}^{-2} \rho_{21} \gamma_{12}^{\ell_{1}} p_{2,\ell_{1}}} {1 - \gamma_{12}} \nonumber\\
 & \quad = \frac {\rho_{11}^{-k} - 1} {1 - \rho_{11}} \gamma_{12}^{\ell_{1} - 1} p_{2,\ell_{1}} + \left(\frac {(\rho_{11} \gamma_{12})^{-(k-1)} - 1}{(\rho_{11} \gamma_{12})^{-1} - 1} \rho_{11}^{-2} - \frac {\rho_{11}^{-k} - \rho_{11}^{-1}} {1 - \rho_{11}} \right)  \frac {\rho_{21} \gamma_{12}^{\ell_{1}} p_{2,\ell_{1}}} {1 - \gamma_{12}}
\end{align*}
for $ k \in [1,\ell_{1}]$. Hence, multiplying both sides of this equation by $\rho_{11}^{k}$, we have
\begin{align}
\label{eq:p1k}
  p_{1,k} & = \rho_{11}^{k} p_{1,0} - \frac {1 - \rho_{11}^{k}} {1 - \rho_{11}} \gamma_{12}^{\ell_{1} - 1} p_{2,\ell_{1}} + \frac {1 - \rho_{11}^{k-1}} {1 - \rho_{11}} \frac {\rho_{21}} {1 - \gamma_{12}} \gamma_{12}^{\ell_{1}} p_{2,\ell_{1}} \nonumber\\
  & \quad - \frac {\gamma_{12}^{-(k-1)} - \rho_{11}^{k-1}}{(\rho_{11} \gamma_{12})^{-1} - 1} \frac {\rho_{11}^{-1} \rho_{21}} {1 - \gamma_{12}} \gamma_{12}^{\ell_{1}} p_{2,\ell_{1}}, \qquad k \in [1,\ell_{1}].
\end{align}
In particular, for $k=\ell_{1}$, substituting $p_{2,\ell_{1}} = \rho_{11} p_{1,\ell_{1}}$ of \eq{p1-ell+11}, we have
\begin{align}
\label{eq:p1ell}
  \rho_{11}^{\ell_{1}} \frac{p_{1,0}} {p_{1,\ell_{1}}} & = 1 + \frac {1 - \rho_{11}^{\ell_{1}}} {1 - \rho_{11}} \gamma_{12}^{\ell_{1} - 1} \rho_{11} - \frac {1 - \rho_{11}^{\ell_{1}-1}} {1 - \rho_{11}} \frac {\rho_{11} \rho_{21}} {1 - \gamma_{12}} \gamma_{12}^{\ell_{1}} \nonumber\\
  & \quad + \frac {\rho_{11} \gamma_{12} - \rho_{11}^{\ell_{1}} \gamma_{12}^{\ell_{1}}}{(\rho_{11} \gamma_{12})^{-1} - 1} \frac {\rho_{11}^{-1} \rho_{21}} {1 - \gamma_{12}}.
\end{align}
Thus, $p_{1,0}$ is explicitly related to $p_{1,\ell_{1}}$.

We now apply these computations to the $r$-th 2-level $M/M/1$ queue, and compute the limit of its scaled stationary distribution through moment generating functions. For the $r$-th system, we put the superscript $^{(r)}$ as we have used for the $r$-th 2-level $GI/G/1$ queue. For example, $p^{(r)}_{j,\ell}$ is the stationary probability of $\{(J,L) = (j,\ell)\}$ for $(j,\ell) \in S$.

Define the moment generating functions $g^{(r)}_{j}(\theta)$ for $j=1,2,3$ as
\begin{align*}
  g^{(r)}_{j}(\theta) = \sum_{k=0}^{\ell^{(r)}_{1}} e^{r \theta k} p^{(r)}_{j,k}, \quad \theta \in \dd{R}, j=1,2, \qquad g^{(r)}_{3}(\theta) = \sum_{k=\ell^{(r)}_{1}+1}^{\infty} e^{r\theta k} p^{(r)}_{3,k}, \quad \theta \le 0,
\end{align*}
and define $g^{(r)}(\theta) = \sum_{j=1}^{3} g^{(r)}_{j}(\theta)$. Obviously, $g^{(r)}(\theta)$ is the moment generating function of $rL^{(r)}$. Thus, our task is to compute the limit of $g^{(r)}(\theta)$ as $r \downarrow 0$. To this end, we compute $g^{(r)}_{j}(\theta)$ for $j=1,2,3$, using the normalization condition that $g^{(r)}(0) = 1$. However, it should be noted that this condition is not sufficient to prove the weak convergence of the distribution of $rL^{(r)}$, for which it is required that $\lim_{\theta \uparrow 0} \lim_{r \downarrow 0} g^{(r)}(\theta) = 1$. This convergence will be verified through explicitly computing $g^{(r)}_{j}(\theta)$ for $j=1,2,3$.

We first compute $g^{(r)}_{j}(\theta)$ for $j=2,3$ because they are easier. It is immediate from \eq{p2i-ell}, \eq{p20i} and \eq{pi+} that
\begin{align}
\label{eq:gr2}
  g_{2}^{(r)}(\theta) & = p^{(r)}_{2,0} + e^{\ell^{(r)}_{1} r \theta} \sum_{\ell=1}^{\ell^{(r)}_{1}} e^{-(\ell^{(r)}_{1} - \ell) r \theta} (\gamma^{(r)}_{12})^{\ell^{(r)}_{1}-\ell} p^{(r)}_{2,\ell^{(r)}_{1}} \nonumber\\
  & =  \left( (\rho^{(r)}_{21})^{-1} (\gamma^{(r)}_{12})^{\ell^{(r)}_{1} - 1} + e^{r \theta} \frac { e^{\ell^{(r)}_{1} r \theta} - (\gamma^{(r)}_{12})^{\ell^{(r)}_{1}}} {e^{r \theta} - \gamma^{(r)}_{12}} \right) p^{(r)}_{2,\ell^{(r)}_{1}},\\
\label{eq:gr3}
  g^{(r)}_{3}(\theta) & = \frac {1}{1 - e^{r \theta} \rho^{(r)}_{2}} e^{(\ell^{(r)}_{1} + 1)r \theta} p^{(r)}_{3,\ell^{(r)}_{1}+1}.
\end{align}

On the other hand, we compute $g^{(r)}_{1}(\theta)$ from \eq{p1k} as given below.
\begin{align}
\label{eq:gr1-1}
    g^{(r)}_{1}(\theta) & = \sum_{k=0}^{\ell^{(r)}_{1}} (e^{r\theta} \rho^{(r)}_{1})^{k} p^{(r)}_{1,0} - \sum_{k=1}^{\ell^{(r)}_{1}} \frac {e^{r \theta k} - (e^{r\theta} \rho^{(r)}_{1})^{k}} {1 - \rho^{(r)}_{1}} (\gamma^{(r)}_{12})^{\ell^{(r)}_{1} - 1} p^{(r)}_{2,\ell^{(r)}_{1}} \nonumber\\
    & \quad + \sum_{k=1}^{\ell^{(r)}_{1}} \frac {e^{r \theta (k-1)} - (e^{r \theta} \rho^{(r)}_{1})^{k-1}} {1 - \rho^{(r)}_{1}} \frac {e^{r\theta} \rho^{(r)}_{21}} {1 - \gamma^{(r)}_{12}} (\gamma^{(r)}_{12})^{\ell^{(r)}_{1}} p^{(r)}_{2,\ell^{(r)}_{1}} \nonumber\\
  & \quad - \sum_{k=1}^{\ell^{(r)}_{1}} \frac {(e^{r\theta} (\gamma^{(r)}_{12})^{-1})^{k-1} - (e^{r\theta} \rho^{(r)}_{1})^{k-1}} {(\rho^{(r)}_{1} \gamma^{(r)}_{12})^{-1} - 1} \frac {e^{r \theta} (\rho^{(r)}_{1})^{-1} \rho^{(r)}_{21}} {1 - \gamma^{(r)}_{12}} (\gamma^{(r)}_{12})^{\ell^{(r)}_{1}} p^{(r)}_{2,\ell_{1}}.
\end{align}
To ease further compuations, we introduce the following functions.
\begin{align*}
  f^{(r)}_{1}(\theta) = \sum_{k=0}^{\ell^{(r)}_{1}-1} e^{r \theta k}, \qquad f^{(r)}_{\rho_{1}}(\theta) = \sum_{k=0}^{\ell^{(r)}_{1}-1} (e^{r \theta} \rho^{(r)}_{1})^{k}, \qquad f^{(r)}_{\gamma^{-1}_{12}}(\theta) = \sum_{k=0}^{\ell^{(r)}_{1}-1} (e^{r \theta} (\gamma^{(r)}_{12})^{-1})^{k},
\end{align*}
which are computed as
\begin{align*}
 & f^{(r)}_{1}(\theta) = \frac {1 - (e^{r \theta})^{\ell^{(r)}_{1}}} {1 - e^{r \theta}} 1(\theta \not=0) + \ell^{(r)}_{1} 1(\theta =0), \qquad f^{(r)}_{\rho_{1}}(\theta) = \frac {1 - (e^{r \theta} \rho^{(r)}_{1})^{\ell^{(r)}_{1}}} {1 - e^{r \theta}\rho^{(r)}_{1}},\\
 & f^{(r)}_{\gamma^{-1}_{12}}(\theta) = \frac {e^{r \theta \ell^{(r)}_{1}} - (\gamma^{(r)}_{12})^{\ell^{(r)}_{1}}} {e^{r \theta} - \gamma^{(r)}_{12}} (\gamma^{(r)}_{12})^{-\ell^{(r)}_{1}+1} .
\end{align*}
Then, using these function, \eq{gr1-1} becomes
\begin{align}
\label{eq:gr1-2}
  g^{(r)}_{1}(\theta) & = \left(f^{(r)}_{\rho_{1}}(\theta) + (e^{r\theta} \rho^{(r)}_{1})^{\ell^{(r)}_{1}}\right) p^{(r)}_{1,0} - \left(f^{(r)}_{1}(\theta) - \rho^{(r)}_{1} f^{(r)}_{\rho_{1}}(\theta)\right) \frac {e^{r\theta} (\gamma^{(r)}_{12})^{\ell^{(r)}_{1} - 1}} {1 - \rho^{(r)}_{1}}  p^{(r)}_{2,\ell^{(r)}_{1}} \nonumber\\
  & \quad + \left(f^{(r)}_{1}(\theta) - f^{(r)}_{\rho_{1}}(\theta)\right) \frac {e^{r\theta} \rho^{(r)}_{21}} {(1 - \rho^{(r)}_{1}) (1 - \gamma^{(r)}_{12})} (\gamma^{(r)}_{12})^{\ell^{(r)}_{1}} p^{(r)}_{2,\ell^{(r)}_{1}} \nonumber\\
  & \quad - \left( f^{(r)}_{\gamma^{-1}_{12}}(\theta) - f^{(r)}_{\rho_{1}}(\theta) \right) \frac {e^{r \theta} (\rho^{(r)}_{1})^{-1} \rho^{(r)}_{21}} {((\rho^{(r)}_{1} \gamma^{(r)}_{12})^{-1} - 1)(1 - \gamma^{(r)}_{12})} (\gamma^{(r)}_{12})^{\ell^{(r)}_{1}} p^{(r)}_{2,\ell^{(r)}_{1}}
\end{align}

We now compute the limits of $g^{(r)}_{j}(0)$ for $j=1,2,3$. For this, we note that the following asymptotic properties are obtained from the assumptions (\sect{sequence}.b)--(\sect{sequence}.e), 
\begin{align}
\label{eq:gamma12 ex-decay}
  (\gamma^{(r)}_{12})^{\ell^{(r)}_{1}} = O(\gamma_{12}^{\ell_{1}/r}), \qquad \mbox{as } r \downarrow 0.
\end{align}
Hence, $(\gamma^{(r)}_{12})^{\ell^{(r)}_{1}}$ decays exponentially fast as the function of $1/r$. On the other hand,
\begin{align}
\label{eq:1-rho-limit}
 & \lim_{r \downarrow 0} (1 - \rho^{(r)}_{\ell})/r = \beta_{\ell},  \qquad \ell=1,2,\\
\label{eq:rho-ell-limit}
 & \lim_{r \downarrow 0} (\rho^{(r)}_{1})^{\ell^{(r)}_{1}} = \lim_{r \downarrow 0} \left(1 - \beta_{1} r + o(r) \right)^{\ell_{1}/r + o(r)} = e^{- \beta_{1} \ell_{1}},
\end{align}
where recall that $\beta_{\ell} = b_{\ell}/c_{\ell}$. Another important asymptotic properties are the ratios of the boundary probabilities. They are presented in the next lemma.

\begin{lemma}
\label{mylem:boundary-p}
For the sequence of the 2-level $M/M/1$ queue satisfying (\sect{sequence}.b)--(\sect{sequence}.e),
\begin{align}
\label{eq:boundary-p}
 & \lim_{r \downarrow 0} \frac {p^{(r)}_{2,\ell_{1}}} {p^{(r)}_{1,\ell_{1}}} = 1, \qquad \lim_{r \downarrow 0} \frac {p^{(r)}_{3,\ell_{1}+1}} {p^{(r)}_{1,\ell_{1}}} = \frac {\lambda_{1} + \lambda_{2}} {\lambda_{2}}, \qquad
\lim_{r \downarrow 0} \frac {p^{(r)}_{1,0}} {p^{(r)}_{1,\ell_{1}}} = \frac {\lambda_{1} + \lambda_{2}} {\lambda_{2}} e^{\beta_{1} \ell_{1}}. \end{align}
\end{lemma}
\begin{proof}
The first two limits in \eq{boundary-p} aer immediate from \eq{p1-ell+11} and \eq{p2i-ell+} because
\begin{align*}
  \lim_{r \downarrow 0} \gamma^{(r)}_{22} = \lim_{r \downarrow 0} \frac {c^{(r)}_{2} \mu^{(r)}} {c^{(r)}_{1} \mu + \lambda^{(r)}_{2}} = \frac {c_{2} \mu} {c_{1} \mu + \lambda_{2}} = \frac {\lambda_{2}} {\lambda_{1} + \lambda_{2}}.
\end{align*}
For the third limit, we have, from \eq{p1ell} for the $r$-th system,
\begin{align*}
  \lim_{r \downarrow 0} \frac{p^{(r)}_{1,0}} {p^{(r)}_{1,\ell_{1}}} & = \lim_{r \downarrow 0} (\rho^{(r)}_{1})^{-\ell^{(r)}_{1}} \left( 1 + \frac {\rho^{(r)}_{1} \gamma^{(r)}_{12}} {(\rho^{(r)}_{1} \gamma^{(r)}_{12})^{-1} - 1} \frac {(\rho^{(r)}_{1})^{-1} \rho^{(r)}_{21}} {1 - \gamma^{(r)}_{12}}\right).
\end{align*}
Hence, applying \eq{rho-ell-limit}, we have the third limit of \eq{boundary-p}.
\end{proof}

We now tentatively assume that there exists finite $\widetilde{p}_{1,\ell_{1}}$ such that
\begin{align}
\label{eq:tp1-ell 1}
  \widetilde{p}_{1,\ell_{1}} = \lim_{r \downarrow 0} \frac {p^{(r)}_{1,\ell^{(r)}_{1}}} {r}.
\end{align}
The existence of this $\widetilde{p}_{1,\ell_{1}}$ will be verified at the final step (see \eq{gr0-limit}). Applying these asymptotic properties, it follows from \eq{gr2} and \eq{gr3} with help of \eq{tp1-ell 1} that
\begin{align}
\label{eq:gr20-limit}
  \lim_{r \downarrow 0} g_{2}^{(r)}(0) & =  \lim_{r \downarrow 0} \left( (\rho^{(r)}_{21})^{-1} (\gamma^{(r)}_{12})^{\ell^{(r)}_{1} - 1} + \frac { 1 - (\gamma^{(r)}_{12})^{\ell^{(r)}_{1}}} {1 - \gamma^{(r)}_{12}} \right) p^{(r)}_{2,\ell^{(r)}_{1}} \nonumber\\
  & = \frac {c_{1} \mu + \lambda_{2}} {\lambda_{2}} \lim_{r \downarrow 0} p^{(r)}_{2,\ell^{(r)}_{1}} = \frac {\lambda_{1} + \lambda_{2}} {\lambda_{2}} \lim_{r \downarrow 0} p^{(r)}_{1,\ell^{(r)}_{1}} = 0,\\
\label{eq:gr30-limit}
  \lim_{r \downarrow 0} g^{(r)}_{3}(0) & = \lim_{r \downarrow 0} \frac {1}{1 - \rho^{(r)}_{2}} p^{(r)}_{3,\ell^{(r)}_{1}+1} = \frac {c_{2}}{b_{2}} \lim_{r \downarrow 0} \frac 1r p^{(r)}_{3,\ell^{(r)}_{1}+1} \nonumber\\
  & = \beta_{2}^{-1} \gamma_{22}^{-1} \lim_{r \downarrow 0} \frac 1r p^{(r)}_{2,\ell^{(r)}_{1}} = \frac {\lambda_{1} + \lambda_{2}} {\lambda_{2}} \beta_{2}^{-1} \widetilde{p}_{1,\ell_{1}}.
\end{align}

To compute the limit of $g^{(r)}_{1}(0)$, we first compute it from \eq{gr1-2} as
\begin{align}
\label{eq:gr10}
   g^{(r)}_{1}(0) & = \left(f^{(r)}_{\rho_{1}}(0) + (\rho^{(r)}_{1})^{\ell^{(r)}_{1}}\right) p^{(r)}_{1,0} - \left(f^{(r)}_{1}(0) - \rho^{(r)}_{1} f^{(r)}_{\rho_{1}}(0)\right) \frac {(\gamma^{(r)}_{12})^{\ell^{(r)}_{1} - 1}} {1 - \rho^{(r)}_{1}}  p^{(r)}_{2,\ell^{(r)}_{1}} \nonumber\\
  & \quad + \left(f^{(r)}_{1}(0) - f^{(r)}_{\rho_{1}}(0)\right) \frac {\rho^{(r)}_{21}} {(1 - \rho^{(r)}_{1}) (1 - \gamma^{(r)}_{12})} (\gamma^{(r)}_{12})^{\ell^{(r)}_{1}} p^{(r)}_{2,\ell^{(r)}_{1}} \nonumber\\
  & \quad - \left( f^{(r)}_{\gamma^{-1}_{12}}(0) - f^{(r)}_{\rho_{1}}(0) \right) \frac {(\rho^{(r)}_{1})^{-1} \rho^{(r)}_{21}} {((\rho^{(r)}_{1} \gamma^{(r)}_{12})^{-1} - 1)(1 - \gamma^{(r)}_{12})} (\gamma^{(r)}_{12})^{\ell^{(r)}_{1}} p^{(r)}_{2,\ell^{(r)}_{1}}.
\end{align}
To proceed to further compute, we note that
\begin{align*}
 & \lim_{r \downarrow 0} r f^{(r)}_{1}(0) = \lim_{r \downarrow 0} r \ell^{(r)}_{1} = \ell_{1},\\
 & \lim_{r \downarrow 0} r f^{(r)}_{\rho_{1}}(0) = \lim_{r \downarrow 0} r \frac {1 - (\rho^{(r)}_{1})^{\ell^{(r)}_{1}}} {1 - \rho^{(r)}_{1}} = \beta_{1}^{-1} (1 - e^{- \beta_{1} \ell_{1}}),\\
 & \lim_{r \downarrow 0} f^{(r)}_{\gamma_{12}^{-1}}(0) (\gamma^{(r)}_{12})^{\ell^{(r)}_{1}} = \lim_{r \downarrow 0} \frac {1 - (\gamma^{(r)}_{12})^{\ell^{(r)}_{1}}} {1 - \gamma^{(r)}_{12}} \gamma^{(r)}_{12} = \frac {c_{1} \mu} {\lambda_{2}} = \rho_{21}^{-1}.
\end{align*}
Then, we have, using \mylem{boundary-p},
\begin{align*}
 & \lim_{r \downarrow 0} \left(f^{(r)}_{\rho_{1}}(0) + (\rho^{(r)}_{1})^{\ell^{(r)}_{1}}\right) p^{(r)}_{1,0} = \beta_{1}^{-1} (1 - e^{- \beta_{1} \ell_{1}}) \lim_{r \downarrow 0} \frac 1{r} p^{(r)}_{1,0} \nonumber\\
  & \quad = \beta_{1}^{-1} \left(e^{\beta_{1} \ell_{1}} - 1\right) \frac {\lambda_{1} + \lambda_{2}} {\lambda_{2}} \widetilde{p}_{1,\ell_{1}},\\
  & \lim_{r \downarrow 0} \left(f^{(r)}_{1}(0) - \rho^{(r)}_{1} f^{(r)}_{\rho_{1}}(0)\right) \frac {(\gamma^{(r)}_{12})^{\ell^{(r)}_{1} - 1}} {1 - \rho^{(r)}_{1}}  p^{(r)}_{2,\ell^{(r)}_{1}} = 0,\\
  & \lim_{r \downarrow 0} \left(f^{(r)}_{1}(0) - f^{(r)}_{\rho_{1}}(0)\right) \frac {\rho^{(r)}_{21}} {(1 - \rho^{(r)}_{1}) (1 - \gamma^{(r)}_{12})} (\gamma^{(r)}_{12})^{\ell^{(r)}_{1}} p^{(r)}_{2,\ell^{(r)}_{1}} = 0,
\end{align*}
and
\begin{align*}
 & \lim_{r \downarrow 0} \left( f^{(r)}_{\gamma^{-1}_{12}}(0) - f^{(r)}_{\rho_{1}}(0) \right) \frac {(\rho^{(r)}_{1})^{-1} \rho^{(r)}_{21}} {((\rho^{(r)}_{1} \gamma^{(r)}_{12})^{-1} - 1)(1 - \gamma^{(r)}_{12})} (\gamma^{(r)}_{12})^{\ell^{(r)}_{1}} p^{(r)}_{2,\ell^{(r)}_{1}}\\
 & \quad = \lim_{r \downarrow 0} f^{(r)}_{\gamma^{-1}_{12}}(0) (\gamma^{(r)}_{12})^{\ell^{(r)}_{1}} \frac {(\rho^{(r)}_{1})^{-1} \rho^{(r)}_{21}} {((\rho^{(r)}_{1} \gamma^{(r)}_{12})^{-1} - 1)(1 - \gamma^{(r)}_{12})} p^{(r)}_{2,\ell^{(r)}_{1}} = 0.
\end{align*}

Hence, from \eq{gr10}
\begin{align}
\label{eq:gr10-limit}
  \lim_{r \downarrow 0} g^{(r)}_{1}(0) = \frac {\lambda_{1} + \lambda_{2}} {\lambda_{2}} \beta_{1}^{-1} (e^{\beta_{1} \ell_{1}} - 1) \widetilde{p}_{1,\ell_{1}}.
\end{align}
Hence, by \eq{gr20-limit} and \eq{gr30-limit}, we have
\begin{align}
\label{eq:gr0-limit}
  1 = \lim_{r \downarrow 0} g^{(r)}(0) = \frac {\lambda_{1} + \lambda_{2}} {\lambda_{2}} \left(\beta_{2}^{-1} + \beta_{1}^{-1} (e^{\beta_{1} \ell_{1}} - 1) \right) \widetilde{p}_{1,\ell_{1}}.
\end{align}
This proves that $\widetilde{p}_{1,\ell_{1}}$ indeed exists and is finite, and given by
\begin{align}
\label{eq:tp1-ell 2}
  \widetilde{p}_{1,\ell_{1}} = \frac {\lambda_{2} \beta_{1} \beta_{2}} {(\lambda_{1} + \lambda_{2}) (\beta_{1} + \beta_{2} (e^{\beta_{1} \ell_{1}} - 1))}.
\end{align}

We are now in a final step to find $\widetilde{g}(\theta) \equiv \lim_{r \downarrow 0} g^{(r)}(\theta)$ for $\theta < 0$. This can be obtained in the following way. From \eq{gr3} and \mylem{boundary-p},
\begin{align*}
  \lim_{r \downarrow 0} g^{(r)}_{3}(\theta) & = \lim_{r \downarrow 0} \frac {1}{1 - e^{r \theta} \rho^{(r)}_{2}} e^{(\ell^{(r)}_{1} + 1)r \theta} p^{(r)}_{3,\ell^{(r)}_{1}+1} \nonumber\\
  & = \lim_{r \downarrow 0} \frac {r}{1 - (1+r\theta + o(r\theta)) (1 - r \beta_{2} + o(r))} e^{\ell^{(r)}_{1} \theta} \frac 1r p^{(r)}_{3,\ell^{(r)}_{1}+1} \nonumber\\
  & = \frac {e^{\ell_{1} \theta}} {\beta_{2} - \theta} \frac {\lambda_{1} + \lambda_{2}} {\lambda_{2}} \lim_{r \downarrow 0} \frac 1r p^{(r)}_{1,\ell^{(r)}_{1}}  \nonumber\\
  & = \frac {e^{\ell_{1} \theta} \beta_{2}} {\beta_{2} - \theta} \frac {\beta_{1}} {\beta_{1} + \beta_{2} (e^{\beta_{1} \ell_{1}} - 1)},
\end{align*}
while, from \eq{gr1-2} and \mylem{boundary-p},
\begin{align*}
  \lim_{r \downarrow 0} g^{(r)}_{1}(\theta) & = \lim_{r \downarrow 0} \left(f^{(r)}_{\rho_{1}}(\theta) + (e^{r\theta} \rho^{(r)}_{1})^{\ell^{(r)}_{1}}\right) p^{(r)}_{1,0} = \lim_{r \downarrow 0} \frac {1 - (e^{r \theta} \rho^{(r)}_{1})^{\ell^{(r)}_{1}+1}} {1 - e^{r \theta} \rho^{(r)}_{1}} p^{(r)}_{1,0} \nonumber\\
  & = \frac {1 - e^{\ell_{1} (\theta - \beta_{1})}} {\beta_{1} - \theta} \frac {\lambda_{1} + \lambda_{2}} {\lambda_{2}} e^{\beta_{1} \ell_{1}} \widetilde{p}_{1,\ell_{1}} \nonumber\\
  & = \frac {\beta_{1}} {e^{\beta_{1} \ell_{1}} - 1} \frac {e^{\beta_{1} \ell_{1}} - e^{\ell_{1} \theta}} {\beta_{1} - \theta} \frac {\beta_{2} (e^{\beta_{1} \ell_{1}} - 1)} {\beta_{1} + \beta_{2} (e^{\beta_{1} \ell_{1}} - 1)}.
\end{align*}
Further, $\lim_{r \downarrow 0} g^{(r)}_{2}(\theta) = 0$ by \eq{gr20-limit}. Hence, we have, for $\theta \le 0$ satisfying $\theta \not= \beta_{1}$,
\begin{align*}
  \lim_{r \downarrow 0} g^{(r)}(\theta) & = \lim_{r \downarrow 0} g^{(r)}_{1}(\theta) + \lim_{r \downarrow 0} g^{(r)}_{3}(\theta) \nonumber\\
  & = \frac {\beta_{1}} {e^{\beta_{1} \ell_{1}} - 1} \frac {e^{\beta_{1} \ell_{1}} - e^{\ell_{1} \theta}} {\beta_{1} - \theta} \frac {\beta_{2} (e^{\beta_{1} \ell_{1}} - 1)} {\beta_{1} + \beta_{2} (e^{\beta_{1} \ell_{1}} - 1)} + \frac {e^{\ell^{(r)}_{1} \theta} \beta_{2}} {\beta_{2} - \theta} \frac {\beta_{1}} {\beta_{1} + \beta_{2} (e^{\beta_{1} \ell_{1}} - 1)}.
\end{align*}
This completes the proof of \mypro{2LQ-exponential}.

\subsection{Proofs of \mylem{2LQ-basic-2}}
\label{myapp:2LQ-basic-2}

We first prove \eq{2LQ-Te 1} and \eq{2LQ-Ts 1}. Apply $f(z,\vc{x}) = x_{1}^{k+1}$ to \eq{BAR-1}, then
\begin{align*}
  (k+1) \dd{E}[(R^{(r)}_{e})^{k}] & = \alpha^{(r)}_{e} \dd{E}[(T^{(r)}_{e,1})^{k+1})] \dd{P}^{(r)}_{e}[L^{(r)}(0) \le \ell^{(r)}_{1}]\\
  & \quad + \alpha^{(r)}_{e} \dd{E}[(T^{(r)}_{e,2})^{k+1})] \dd{P}^{(r)}_{e}[L^{(r)}(0) > \ell^{(r)}_{1}],
\end{align*}
where \eq{BAR-1} is applicable since the right-hand side of this equation is finite. Hence, \eq{2LQ-Te 1} is obtained. Similarly, we have \eq{2LQ-Ts 1} applying $f(z,\vc{x}) = x_{2}^{k+1}$ to \eq{BAR-1} because
\begin{align*}
  f'(X^{(r)}(t)) = - (k+1) (R^{(r)}_{s}(t))^{k} (1 - 1(L^{(r)}(t) = 0)),
\end{align*}
and $R^{(r)}(t)$ has the same distribution as $T^{(r)}_{s}$ when $L^{(r)}(t) = 0$.
We next prove \eq{2LQ-Tu 1}. For this, we use the following inequality,
\begin{align}
\label{eq:2LQ-x-inequality 1}
  |(x \wedge 1/r)^{k} - x^{k}| & \; = \; x^{k} 1(x > 1/r) < r x^{k+1} 1(x > 1/r)  \nonumber\\
  & < r^{2} x^{k+2} 1(x > 1/r), \qquad x \ge 0, k \ge 1.
\end{align}
Hence, the last inequality in \eq{2LQ-x-inequality 1} for $k=1$ and the uniform boundedness of $\dd{E}[(T^{(r)}_{u})^{3}$ yield
\begin{align}
\label{eq:2LQ-Re-bound 1}
 & |\dd{E}[\widehat{T}^{(r)}_{u}] - \dd{E}[T^{(r)}_{u}]| \le r^{2} \dd{E}[(T^{(r)}_{u})^{3} 1(T^{(r)}_{u} > 1/r)] = o(r^{2}).
\end{align}
On the other hand, the second inequality in \eq{2LQ-x-inequality 1} yields
\begin{align}
\label{eq:2LQ-Re-bound 2}
 & |\dd{E}[(\widehat{T}^{(r)}_{u})^{2}] - \dd{E}[(T^{(r)}_{u})^{2}]| \le r \dd{E}[(T^{(r)}_{u})^{2} 1(T^{(r)}_{u} > 1/r)] = o(r).
\end{align}
Thus, \eq{2LQ-Tu 1} is proved for $k=1,2$. We next prove \eq{2LQ-Rv 1} for $v = e$. Since $\dd{E}[(R^{(r)}_{e})^{k}]$ is uniformly bounded for $k=1,2$ by \eq{2LQ-Ts 1}, we can prove \eq{2LQ-Rv 1} by the same argument as the proof of \eq{2LQ-Tu 1}, but the order of $r$ is one less because of \eq{2LQ-Ts 1}. We omit the proof for $v = s$ because it is similar to the case of $v=e$ except for \eq{2LQ-Te 1} to be used instead of \eq{2LQ-Ts 1}.

\subsection{Proofs of Lemmas \mylemt{2LQ-asymp 1} and \mylemt{2LQ-asymp 2}}
\label{myapp:2LQ-asymp}

\begin{proof}[Proof of \mylem{2LQ-asymp 1}]
For a $(n+1)$-times continuously differentiable function $f$, we have, from Taylor expansion (e.g. see Theorem 19.9 of \cite{Rudi1987}), 
\begin{align}
\label{eq:Taylor 1}
  \left|f(x) - \sum_{k=0}^{n} \frac 1{k!} f^{(k)}(0) x^{k}\right| & = \frac 1{n!} \left|\int_{0}^{x} (x-y)^{n} f^{(n+1)}(y) dy\right| \nonumber\\
  & \le \frac 1{(n+1)!} |x|^{n+1} \max_{|u| \le |x|} |f^{(n+1)}(u)|, \quad x \in \dd{R}, n \ge 0.
\end{align}
Applying $f(x) = e^{-x}$ to this inequality, we have $|e^{-x} - 1 + x| \le \frac 12 |x|^{2} e^{|x|}$. Hence, for each fixed $\theta \in \dd{R}$, as $r \downarrow 0$
\begin{align*}
 & e^{r \theta \ell^{(r)}_{1}} = e^{\theta (\ell_{1} + o(r))} = e^{\theta \ell_{1}} (1+\theta o(r)) + o(\theta o(r)) = e^{\theta \ell_{1}} (1+ o(\theta r)),\\
 & e^{r \theta (\ell^{(r)}_{1}+1)} = e^{\theta r \ell^{(r)}_{1}} e^{r \theta} = e^{\theta \ell_{1}} (1+o(\theta r)) (1 + \theta r + o((\theta r)^{2})) = e^{\theta \ell_{1}} (1+ \theta r+ o(\theta r)).
\end{align*}
Thus, \eq{2LQ-exp 1} is obtained. \eq{2LQ-exp 2} and \eq{2LQ-exp 3} are immediate from these asymptotic expansions.
\end{proof}

\begin{proof}[Proof of \mylem{2LQ-asymp 2}]
Since $\zeta^{(r)}(r\theta) \widehat{R}^{(r)}_{s-}$ and $\eta^{(r)}_{i}(r\theta) \widehat{R}^{(r)}_{e}$ are uniformly bounded in $r < a$ for a sufficiently small $a>0$ and vanish as $r \downarrow 0$, it follows from \mylem{2LQ-eta-zeta 1} and \eq{Taylor 1} that
\begin{align*}
 & e^{-\zeta^{(r)}(r\theta) \widehat{R}^{(r)}_{s-}}  \nonumber\\
 & \quad = 1 - \zeta^{(r)}(r\theta) \widehat{R}^{(r)}_{s-} + \frac 12 (\zeta^{(r)}(r\theta) \widehat{R}^{(r)}_{s-})^{2} + O \big(|\zeta^{(r)}(r\theta) \widehat{R}^{(r)}_{s-}|^{3} \max_{r |\theta| < a} e^{|\zeta^{(r)}(r\theta) \widehat{R}^{(r)}_{s-}|} \big) \nonumber\\
 & \quad = 1 - \left(- \mu^{(r)} r\theta + \frac 12 (\mu^{(r)})^{3} (\sigma^{(r)}_{s})^{2} (r\theta)^{2} + o((r\theta)^{2}) \right) \widehat{R}^{(r)}_{s-}  \nonumber\\
 & \qquad + \frac 12 \left(- \mu^{(r)} r\theta + \frac 12 (\mu^{(r)})^{3} (\sigma^{(r)}_{s})^{2} (r\theta)^{2} + o((r\theta)^{2}) \right)^{2} (\widehat{R}^{(r)}_{s-})^{2} \nonumber\\
 & \qquad + (\widehat{R}^{(r)}_{s-})^{3} O \big((r \theta)^{3} \max_{r |\theta| < a} e^{|\zeta^{(r)}(r\theta) \widehat{R}^{(r)}_{s-}|} \big) \nonumber\\
 & \quad = 1 + r\theta (\mu^{(r)} \widehat{R}^{(r)}_{s-}) - \frac 12 (r\theta)^{2} \left( (\mu^{(r)} \sigma^{(r)}_{s})^{2} (\mu^{(r)} \widehat{R}^{(r)}_{s-}) - (\mu^{(r)} \widehat{R}^{(r)}_{s-})^{2} \right) + S_{3}(\widehat{R}^{(r)}_{s-}) o((r\theta)^{2}),
\end{align*}
where recall that $S_{3}(\widehat{R}^{(r)}_{s-}) = \widehat{R}^{(r)}_{s-}+ (\widehat{R}^{(r)}_{s-})^{2} + (\widehat{R}^{(r)}_{s-})^{3}$. This proves \eq{2LQ-exp-zeta 1}, and \eq{2LQ-exp-zeta 2} and \eq{2LQ-exp-eta 2} are similarly proved.
\end{proof}

\subsection{Proof of \mylem{2LQ-Pell-lim-1}}
\label{myapp:2LQ-lem51}

We first assume $\lambda_{1} \not= \lambda_{2}$, and prove \eq{2LQ-Re-Es-lim-1} for $i=1$. Consider the BAR \eq{2LQ-BAR1} for $\theta_{2} = \theta_{1} = \theta \le 0$. From \eq{2LQ-D 1}, we have
\begin{align}
\label{eq:2LQ-D11-1}
   & D^{(r)}(\theta,\theta) = \alpha^{(r)}_{e} e^{r \theta (\ell^{(r)}_{1} + 1)} \dd{E}^{(r)}_{s}[(e^{-\eta^{(r)}_{1}(r\theta) \widehat{R}^{(r)}_{e}}  - e^{-\eta^{(r)}_{2}(r\theta) \widehat{R}^{(r)}_{e}}) 1(L^{(r)}(0) = \ell^{(r)}_{1})].
\end{align}zz
Applying \eq{2LQ-exp 1}, \eq{2LQ-exp-eta 2} and \mycor{2LQ-Re-Rs 1} to this equation, we have
\begin{align*}
  D^{(r)}(\theta,\theta) & = \alpha^{(r)}_{e} e^{\theta \ell_{1}} (1+ r\theta + o(r\theta)) \dd{E}^{(r)}_{s} \big[(r \theta (\lambda^{(r)}_{2} - \lambda^{(r)}_{1}) \widehat{R}^{(r)}_{e} \nonumber\\
 & \quad +  2^{-1} r^{2} \theta^{2} \{(\lambda^{(r)}_{2} \sigma^{(r)}_{e,2})^{2} \lambda^{(r)}_{2} - (\lambda^{(r)}_{1} \sigma^{(r)}_{e,1})^{2} \lambda^{(r)}_{1}\} \widehat{R}^{(r)}_{e} \nonumber\\
 & \quad - 2^{-1} r^{2} \theta^{2} \{(\lambda^{(r)}_{2})^{2} - (\lambda^{(r)}_{1})^{2}\} (\widehat{R}^{(r)}_{e})^{2}) 1(L^{(r)}(0) = \ell^{(r)}_{1}) \big] + o((r \theta)^{2}), \nonumber\\
 & = \alpha^{(r)}_{e} e^{\theta \ell_{1}} \big\{ r \theta (\lambda^{(r)}_{2} - \lambda^{(r)}_{1}) \dd{E}^{(r)}_{s} \big[\widehat{R}^{(r)}_{e}  1(L^{(r)}(0) = \ell^{(r)}_{1}) \big] + (r \theta)^{2} (\sr{E}^{(r)}_{1} + \sr{E}^{(r)}_{2})\big\} \nonumber\\
 & \quad + o((r \theta)^{2}).
\end{align*}
Since $\sr{E}^{(r)}_{1} + \sr{E}^{(r)}_{2}$ is uniformly bounded in $r \in (0,1]$, this implies that
\begin{align}
\label{eq:D11-2}
 D^{(r)}(\theta,\theta) & = r \theta \alpha^{(r)}_{e} e^{\theta \ell_{1}} ( \lambda^{(r)}_{2} - \lambda^{(r)}_{1}) \dd{E}^{(r)}_{s} \big[ \widehat{R}^{(r)}_{e} 1(L^{(r)}(0) = \ell^{(r)}_{1}) \big] + O((r\theta)^{2}).
\end{align}
Substitute this formula into the BAR \eq{2LQ-BAR1} for $\theta_{2} = \theta_{1} = \theta \le 0$, which is given by
\begin{align}
\label{eq:2LQ-BAR2}
 & \frac 12 \sum_{i=1,2} c_{i} \sigma_{i}^{2} \big((\beta_{i} - \theta) \mu \theta r^{2} + o((\theta r)^{2})\big) \psi^{(r)}_{i}(\theta) \nonumber\\
 & \quad = c_{1} \left(\mu r\theta - \mu^{3} \sigma_{s}^{2} r^{2} \theta^{2}/2 + o((\theta r)^{2})\right) \dd{E}[ 1(L^{(r)} = 0) g^{(r)}_{1,r \theta}(R^{(r)})] + D^{(r)}(\theta,\theta).
\end{align}
Then, picking up the terms of order $r \theta$, we have
\begin{align*}
 & c_{1} \mu \dd{E}[1(L^{(r)} = 0) g^{(r)}_{1,r\theta}(R^{(r)})]+ \alpha^{(r)}_{e} e^{\theta \ell_{1}} (\lambda_{2} - \lambda_{1}) \dd{E}_{s}[\widehat{R}^{(r)}_{e} 1(L^{(r)}(0) = \ell^{(r)}_{1})] = O(r\theta).
\end{align*}
Since $\dd{E}[ 1(L^{(r)} = 0) g^{(r)}_{1,r \theta}(R^{(r)})] = O(r)$ by \mycor{2LQ-limit 1} and \mylem{2LQ-L0-1}, this yields
\begin{align}
\label{eq:2LQ-lim-L0-1}
  (\lambda_{2} - \lambda_{1}) \dd{E}_{s}[\widehat{R}^{(r)}_{e} 1(L^{(r)}(0) = \ell^{(r)}_{1})] = O(r).
\end{align}
Since $\lambda_{1} \not= \lambda_{2}$, \eq{2LQ-lim-L0-1} implies \eq{2LQ-Re-Es-lim-1} for $i=1$. To prove \eq{2LQ-Re-Es-lim-1} for $i=2$, we derive the following formula by Schwartz's inequality.
\begin{align}
\label{eq:2LQ-i=2}
& \big\{\dd{E}^{(r)}_{s}[(\widehat{R}^{(r)}_{e})^{2} 1(L^{(r)}(0) = \ell^{(r)}_{1}]\big\}^{2} = \big\{\dd{E}^{(r)}_{s}[(\widehat{R}^{(r)}_{e})^{3/2} (\widehat{R}^{(r)}_{e})^{1/2} 1^{2}(L^{(r)}(0) = \ell^{(r)}_{1})] \big\}^{2} \nonumber\\
 & \quad \le \dd{E}^{(r)}_{s}[(\widehat{R}^{(r)}_{e})^{3} 1(L^{(r)}(0) = \ell^{(r)}_{1})] \dd{E}^{(r)}_{s}[\widehat{R}^{(r)}_{e} 1(L^{(r)}(0) = \ell^{(r)}_{1})].
\end{align}
Since $\dd{E}^{(r)}_{s}[(\widehat{R}^{(r)}_{e})^{3} 1(L^{(r)}(0) = \ell^{(r)}_{1})]$ is uniformly bounded in $r$ by \mycor{2LQ-Re-Rs 1}, \eq{2LQ-Re-Es-lim-1} for $i=1$ implies that for $i=2$. Thus, \eq{2LQ-Re-Es-lim-1} is proved.

We next prove that \eq{2LQ-Pell-lim-1} implies \eq{2LQ-Re-Es-lim-1}. For this proof, we use the BAR \eq{BAR-1} for $f(z,x_{1},x_{2}) = (x_{1} \wedge 1/r) (x_{2} \wedge 1/r) 1(z = \ell^{(r)}_{0})$, then we have
\begin{align}
\label{eq:2LQ-ReRs-1}
 & \dd{E}[R^{(r)}_{e} 1(R^{(r)}_{e} \le 1/r, L^{(r)} = \ell^{(r)}_{0})] + \dd{E}[R^{(r)}_{s} 1(R^{(r)}_{s} \le 1/r, L^{(r)} = \ell^{(r)}_{0})] \nonumber\\
 & \quad = \alpha^{(r)}_{e} \dd{E}[\widehat{T}^{(r)}_{e,2}] \dd{E}^{(r)}_{e}[\widehat{R}^{(r)}_{s-} 1(L^{(r)}(0) = \ell^{(r)}_{0})] + \alpha^{(r)}_{e} \dd{E}[\widehat{T}^{(r)}_{s}] \dd{E}^{(r)}_{s}[\widehat{R}^{(r)}_{e} 1(L^{(r)}(0) = \ell^{(r)}_{0})].
\end{align}
Then, we can see that the left-hand side of this equation vanishes as $r \downarrow 0$ by \eq{2LQ-Pell-lim-1} because, by Schwartz's inequality,
\begin{align*}
  \big(\dd{E}[R^{(r)}_{v}1(R^{(r)}_{v} \le 1/r, L^{(r)} = \ell^{(r)}_{0})]\big)^{2} \le \dd{E}[(R^{(r)}_{v})^{2}] \dd{P}[L^{(r)} = \ell^{(r)}_{0}], \qquad v= e,s.
\end{align*}
Hence, we have \eq{2LQ-Re-Es-lim-1} for $i=1$. It for $i=2$ follows from \eq{2LQ-i=2} as we have argued.

\subsection{Proof of \mylem{2LQ-Pell-lim-2}}
\label{myapp:2LQ-lem52}

In the proof of \mypro{2LQ-3f}, under the assumption $\lambda_{1} \not= \lambda_{2}$, \eq{2LQ-E12} is obtained applying \eq{2LQ-Re-Es-lim-1} to \eq{2LQ-E1+2}. Obviously, if \eq{2LQ-Re-Es-lim-1} holds, then this proof is also valid without any extra assumption including $\lambda_{1} \not= \lambda_{2}$. Hence, the lemma is proved.

\subsection*{Acknowledgement} This study is motivated by the BAR approach, coined by Jim Dai (e.g., see \cite{BravDaiMiya2024}). I am grateful to him for his grand design about it, and to Rami Atar for providing his recent papers \cite{AtarCastReim2024a,AtarWola2024}. I also benefit from Evsey Morozov for his comments on this paper.


\end{document}